\documentclass[12pt,reqno]{amsart}
\usepackage[left=1in, right=1in, top=1in]{geometry}
\usepackage{amsfonts}
\usepackage{comment}
\usepackage{cases}
\usepackage{latexsym}
\usepackage[all]{xy}
\usepackage{stmaryrd}
\usepackage{amsmath,amssymb,amscd,bbm,amsthm,mathrsfs,dsfont}
\usepackage{tikz}
\usepackage{extarrows}
\usepackage{pgflibraryarrows}
\usepackage{pgflibrarysnakes}
\usepackage{pgfplots}
\usepackage[all]{xy}
\usepackage[numbers,sort&compress]{natbib}
\usepackage{hypernat}
\usepackage{color, hyperref}

\definecolor{blue}{rgb}{1.0,0.0,0.9}
\hypersetup{colorlinks, breaklinks,
            linkcolor=blue,urlcolor=blue,
            anchorcolor=blue, citecolor=blue}

\usepackage{fancyhdr}
\usepackage{amsxtra,ifthen}
\usepackage{verbatim}

\numberwithin{equation}{section}

\theoremstyle{plain}
\newtheorem{theorem}{Theorem}[section]

\newtheorem{proposition}[theorem]{Proposition}

\newtheorem{corollary}[theorem]{Corollary}

\theoremstyle{definition}

\allowdisplaybreaks[4]

\makeatletter
\@namedef{subjclassname@2020}{\textup{2020}  Mathematics Subject Classification}
\makeatother
\begin{document}


\title[The behaviour of moving points on curves: A  rotating frame approach]
{The behaviour of moving points on curves: A  rotating frame approach}

\author{Dong Han }

\address{Han: School of Mathematics and Information Science, Henan Polytechnic University, Jiaozuo 454000, P. R. China}

\email{lishe@hpu.edu.cn}

\begin{abstract}In this paper, we construct rotating frames for curves, including plane curves, space curves and curves on surfaces. Hence, the behaviour of an arbitrary moving point on a curve can be seen as the composite of linear motion and rotation.  Conversely, it can also be proved that a curve can  be determined by the two motions of a moving point on it, namely, linear motion and rotation. Thus, we obtain a new binary mathematical formation mechanism for curves based on  the aforementioned two motions. Finally, we apply this rotating frame method to the study of the behaviour of moving points on ellipses.

 \end{abstract}



\thanks{The work is supported by National Natural Science Foundation of China (No.11926415).}


\maketitle


\section{Introduction}

A fixed point $O$ in $n$-dimensional affine space, together with the set of ordered basis vectors  $\overrightarrow{e}_1$, $\overrightarrow{e}_2$, $\cdots$, $\overrightarrow{e}_n$, is called a frame, which is a fundamental concept in mathematics. Different frame systems can be used to reveal distinct  properties of an object from varying perspectives. Consequently, it is a defining feature of frame theory to select appropriate frame systems based on specific problems.

There  has been a long history to study curves with different frames in Euclidean, Minkowski, or other special spaces (\cite{{Bishop}},   \cite{{Liang}}-\cite{Li2}, \cite{ZhaoGao} etc.). For curves, the main frame is the Frenet frame in differential
geometry (\cite{Kuhnel}, \cite{LiYuanZhang}-\cite{Tapp} etc.).  The Frenet frame of a space curve consists of three unit orthogonal vectors, namely, unit tangent, principal normal, and binormal vectors.   Using the Frenet frame, we can observe the structure of a curve near a point clearly  and achieve shape-based classification of regular curves. Precisely because the Frenet frame serves as a pivotal tool,  differential geometry of curves has become the most classical curve theory to date. Although the Frenet frame  is a powerful analytical tool for curves, it cannot be defined at the points where the second-order derivatives of the curve (univariate vector-valued function) are equal to  zero vectors. This leads to the popularization of orthonormal frame systems defined along a curve. Among them, the most representative is the Bishop frame (\cite{{Bishop}}).

This paper aims to present a new orthonormal frame system for curves, namely the rotating frame. The main characteristic of this frame system is that it enables us to observe two different types of motion of a moving point on a curve, that is, linear motion and rotation.
On this basis, we can  discuss the differential properties of the  two motions just mentioned.  We also show that a curve can be determined in reverse by the linear motion and rotation of its moving point. It can thus be stated  that  linear motion and rotation constitute a binary formation mechanism for curves. Moreover, compared to some traditional frames in differential geometry,  this  rotating frame has the following differences:

(i) It applies to both plane and space curves. It is comparatively more proficient in handling general plane curves.

(ii) It can be placed both on and off a curve. If placed off a curve, then  it will be called a general rotating frame of the curve, and the  external (or global) differential properties of the linear motion and rotation of the moving points on the curve can be obtained (see Proposition  \ref{xxsec2.1}, \ref{xxsec2.5}, \ref{xxsec3.1} etc.); while if placed on a curve, then it will be called a local rotating frame of the curve,  and the internal (or local)  differential properties of the two motions can be obtained (see Proposition  \ref{xxsec2.6}, \ref{xxsec3.4}, \ref{xxsec3.5}  etc.). In addition,  local rotating frames enable the classification of curves based on their shape (see Theorem \ref{xxsec2.7}, \ref{xxsec3.6} etc.).

The outline of this article is as follows.  $\S$  2 is devoted to  plane curves, where plane curves
with general rotating frames   are introduced in $\S$ 2.1-2.3,  plane curves
with local rotating frames   in $\S$ 2.4,  and straight lines with rotating frames in  $\S$ 2.5.
Space curves are considered in $\S$ 3, including space curves  with general rotating frames in $\S$ 3.1, space curves  with local rotating frames in $\S$ 3.2, and two space curves with symmetrical rotating frames  in $\S$ 3.3.  $\S$ 4 is to extend the results of $\S$ 3 to the curves on surfaces. We conclude the paper with $\S$  5 by giving an application of  rotating frames to the study of the behaviour of moving points on  ellipses. In Appendix A and B, we present some complicated and tedious computational processes which are fairly necessary for the completeness of our paper.

\section{The behaviour of moving points on plane curves }

\subsection{Plane curves with general rotating frames 1}
A simple curve of class $C^2$ in $\mathbb{R}^2$ can usually be given by the following vector function:
\begin{equation}\label{xxsecpla1}
  \overrightarrow{r} (t)=\{x(t),y(t)\}, \ \ \ t_0 < t< t_1.
\end{equation}
Set up a rotating frame at the origin $O$ of $\mathbb{R}^2$ as follows:
\begin{equation}\label{xxsecpla2}
 \{O;\overrightarrow{e}_1(t), \overrightarrow{e}_2(t) \}, \ \ \ t_0 < t< t_1,
\end{equation}
where $\overrightarrow{e}_1(t)= \frac{\{x(t),y(t)\}}{\sqrt{x^2(t)+y^2(t)}}$, $\overrightarrow{e}_2(t)=\frac{\{-y(t),x(t)\}}{\sqrt{x^2(t)+y^2(t)}}$. Frame (\ref{xxsecpla2})  is a general rotating frame of (\ref{xxsecpla1}).  If the coordinates of (\ref{xxsecpla2}) are  $\xi$, $\eta$,
then the equation of (\ref{xxsecpla1}) under (\ref{xxsecpla2}) can be written in the form \begin{align}
  \xi(t)= \sqrt{x^2(t)+y^2(t) },\   \eta(t)=0, \ \ \ t_0 < t< t_1.\label{xxsecpla3}
\end{align}

\textbf{Note:}  In this subsection, when we speak of $t$, we mean, unless otherwise stated, $t_0 <t< t_1$, and  (\ref{xxsecpla1}) does not pass through  $O(0,0)$. Additionally, the capital letter $O$  represents  both the origin of $\mathbb{R}^2$\,(Section 2, 5) and $\mathbb{R}^3$\,(Section 3, 4).

With the help of (\ref{xxsecpla2}), one can observe that there are two kinds of motions included in the behaviour of the moving points on (\ref{xxsecpla1}), namely linear motion and rotation. This fact can be described in detail by the following proposition.
\begin{proposition}\label{xxsec2.1}  Let $P$ be a moving point on (\ref{xxsecpla1}), and let its parameter be $t$.
\begin{enumerate}
  \item[$\mathrm{(i)}$]  For $\mathcal{D}=|\overrightarrow{OP}|$, we have
  \begin{align}\frac{ \mathrm{d}\mathcal{D} }{\mathrm{d}t}=\frac{x(t) x'(t) +y(t) y'(t) }{\sqrt{x^2(t) +y^2(t)}}.\notag
\end{align}
\item[$\mathrm{(ii)}$] The rotational speed of $\overrightarrow{OP}$ with respect to $t$
 is $\frac{|x(t)y'(t)-x'(t)y(t)|}{x^2(t)+y^2(t)}$.
\end{enumerate}
\end{proposition}
\begin{proof} (i) Note that $\mathcal{D}=\sqrt{x^2(t) +y^2(t)}$, which equals to $\xi(t)$ in (\ref{xxsecpla3}); hence, the conclusion holds.

(ii) First we have $$\overrightarrow{e}_{\overrightarrow{OP}} = \frac{\{x(t) ,y(t) \}}{\sqrt{x^2(t) +y^2(t) }}.$$
Thus
\begin{align}\frac{\mathrm{d}\overrightarrow{e}_{\overrightarrow{OP}} }{\mathrm{d}t}
 =&-\frac{1}{2}(x^2 +y^2)^{-\frac{3}{2}}\times2(x x' +y y' )\{x ,y \}+(x^2 +y^2)^{-\frac{1}{2}}\{x' , y' \}\notag\\
 =&-(x^2 +y^2)^{-\frac{3}{2}}  (x x' +y y' )\{x ,y \}\notag\\&+(x^2 +y^2)^{-\frac{3}{2}}(x^2 +y^2)\{x' , y' \}\notag\\
 =&-(x^2 +y^2)^{-\frac{3}{2}}\{x^2 x' + y y'x ,x x'y +y^2 y' \}\notag\\
  &+(x^2 +y^2)^{-\frac{3}{2}}\{x^2 x' +y^2 x' , x^2 y' +y^2 y' \}\notag\\
   =&(x^2 +y^2)^{-\frac{3}{2}}\{ - y y' x+y^2 x' ,  -x x'y +x^2 y'  \}\notag\\
  =&(x^2 +y^2)^{-\frac{3}{2}} \{y (-y'x +y x'),  x (-x'y +x y' )\}\notag\\
    =&(x^2 +y^2)^{-\frac{3}{2}} \{y (-x y' +x'y),  x (-x'y +x y' )\}\notag\\
 =&\frac{(x y' -x'y )\{ -y ,  x  \}}{(x^2 +y^2)^{\frac{3}{2}}},\label{xxsecpla4}\end{align}
$$|\frac{\mathrm{d}\overrightarrow{e}_{\overrightarrow{OP}} }{\mathrm{d}t}|= \frac{|x y' -x'y |}{x^2 +y^2 }.$$

\end{proof}

\textbf{Note:} In the proof process of the above proposition, we omit the parameter $t$ and write $x(t)$\,(resp. $y(t)$) as $x$\,(resp. $y$) for the sake of simplicity. In the following text, we will adopt similar omissions without further explanation.

We have already demonstrated that the moving point $P$ of (\ref{xxsecpla1}) satisfies Proposition \ref{xxsec2.1}(i) and (\ref{xxsecpla4}) at the same time. Next we will prove that  (\ref{xxsecpla1}) can also be determined completely by Proposition \ref{xxsec2.1}(i) and (\ref{xxsecpla4}) provided the initial values are given.
\begin{theorem}\label{xxsec2.2} Let $O$ be the origin of $\mathbb{R}^2$. Suppose that $P$ is a moving point in $\mathbb{R}^2$ satisfying the following conditions:
\begin{enumerate}
  \item[$\mathrm{1)}$]  for $\mathcal{D}=|\overrightarrow{OP}|$, we have
  \begin{align}
 &\frac{ \mathrm{d}\mathcal{D}}{\mathrm{d}t}=\frac{x(t)x'(t)+y(t)y'(t)}{ \sqrt{x^2(t)+y^2(t)} }, \mathcal{D}|_{t=t_0}=\sqrt{x^2(t_0)+y^2(t_0)}; \notag
\end{align}
\item[$\mathrm{2)}$] $\frac{\mathrm{d}\overrightarrow{e}_{\overrightarrow{OP}}}{\mathrm{d}t}
 =\frac{(x(t)y'(t)-x'(t)y(t))\{ -y(t),  x(t) \}}{ (x^2(t)+y^2(t))^{\frac{3}{2}}}$, $\overrightarrow{e}_{\overrightarrow{OP}}|_{t=t_0}= \frac{\{x(t_0),y(t_0)\}}{\sqrt{x^2(t_0)+y^2(t_0)}}$.
\end{enumerate}
Here  $t_0 \leq t\leq t_1$. Then the trajectory of $P$ is  $$\overrightarrow{r} (t)=\{x(t),y(t)\},\ \ \ t_0 \leq t\leq t_1.$$

\end{theorem}
\begin{proof} Let $\overrightarrow{OP}(t)=\{f_1(t),f_2(t)\}$, $t_0 \leq t\leq t_1$. It follows from 2) that $f_1=\lambda(t)x(t)$, $f_2=\lambda(t)y(t)$, where $\lambda(t)>0$.
By 1) we have
 \begin{equation}
{f_1}^2(t)+{f_2}^2(t) = x^2(t)+y^2(t)  , \ \ \ t_0 \leq t\leq t_1.\notag
\end{equation}
Taking  $f_1=\lambda(t)x(t)$, $f_2=\lambda(t)y(t)$ into the above expression, we obtain
$$ \lambda^2(t)(x^2(t)+y^2(t) )= x^2(t)+y^2(t)   , \ \ \ t_0 \leq t\leq t_1.$$
So $\lambda(t)=\pm 1$, and we must have $\lambda(t)=1$. This completes the proof.
\end{proof}

Notice, from  Proposition \ref{xxsec2.1}(i), that
 \begin{align}
\frac{ \mathrm{d}^2\mathcal{D}}{\mathrm{d}t^2}
  =& -(x^2 +y^2  )^{-\frac{3}{2}}(x x' +y y' )^2\notag\\
  &+(x^2 +y^2  )^{-\frac{1}{2}}(x'^2 + x x'' +y'^2 +y y'' ) .\label{xxsecpla5}
\end{align}
So we can also write the above theorem into the following form.
\begin{theorem}\label{xxsec2.3} Let $O$ be the origin of $\mathbb{R}^2$. Suppose that $P$ is a moving point in $\mathbb{R}^2$ satisfying the following conditions:
\begin{enumerate}
  \item[$\mathrm{1)}$]  for $\mathcal{D}=|\overrightarrow{OP}|$, we have
  \begin{align}
&\frac{ \mathrm{d}^2\mathcal{D}}{\mathrm{d}t^2}
   =\frac{-(x(t)x'(t)+y(t)y'(t))^2+(x^2(t)+y^2(t) )(x'^2(t)+ x(t)x''(t)+y'^2(t)+y(t)y''(t)) }{(x^2(t)+y^2(t) )^{\frac{3}{2}}},  \notag\\
 &\frac{ \mathrm{d}\mathcal{D}}{\mathrm{d}t}|_{t=t_0}=\frac{x(t_0)x'(t_0)+y(t_0)y'(t_0)}{ \sqrt{x^2(t_0)+y^2(t_0)} }, \mathcal{D}|_{t=t_0}=\sqrt{x^2(t_0)+y^2(t_0)}; \notag
\end{align}
\item[$\mathrm{2)}$] $\frac{\mathrm{d}\overrightarrow{e}_{\overrightarrow{OP}}}{\mathrm{d}t}
 =\frac{(x(t)y'(t)-x'(t)y(t))\{ -y(t),  x(t) \}}{ (x^2(t)+y^2(t))^{\frac{3}{2}}}$, $\overrightarrow{e}_{\overrightarrow{OP}}|_{t=t_0}= \frac{\{x(t_0),y(t_0)\}}{\sqrt{x^2(t_0)+y^2(t_0)}}$.
\end{enumerate}
Here  $t_0 \leq t\leq t_1$. Then the trajectory of $P$ is  $$\overrightarrow{r} (t)=\{x(t),y(t)\},\ \ \ t_0 \leq t\leq t_1.$$

\end{theorem}
\begin{proof} The proof is completely similar to that of Theorem \ref{xxsec2.2}.
\end{proof}

\subsection{Parameter transformation }  Let $t=g(h)\in C^2$  be a monotonic function, where  $h_0\leq h\leq h_1$. We might as well assume $g(h_0)=t_0$, $g(h_1)=t_1$.  Now (\ref{xxsecpla1}) can be rewritten as
\begin{equation}\label{xxsecpla6}
  \overrightarrow{r}(h) =\{x(g(h)), y(g(h))\},\ h_0< h<h_1.
\end{equation}
Correspondingly,  (\ref{xxsecpla2}) becomes
\begin{equation}\label{xxsecpla7}
 \{O; \overrightarrow{e}_1 (g(h)), \overrightarrow{e}_2 (g(h)) \}, \ \ \  h_0< h<h_1,
\end{equation}
where $\overrightarrow{e}_1(g(h))= \frac{\{x(g(h)),y(g(h))\}}{\sqrt{x^2(g(h))+y^2(g(h))}}$, $\overrightarrow{e}_2(g(h))=\frac{\{-y(g(h)),x(g(h))\}}{\sqrt{x^2(g(h))+y^2(g(h))}}$.
If the coordinates are still $\xi, \eta$, then the equation of (\ref{xxsecpla6}) under (\ref{xxsecpla7}) becomes
\begin{equation}\label{xxsecpla8}
  \xi(g(h))= \sqrt{x^2(g(h))+y^2(g(h)) },\  \eta(g(h))=0, \ \ \ h_0< h<h_1.
\end{equation}

\textbf{Note:} In this subsection, we default to  $h_0< h<h_1$, and  (\ref{xxsecpla7}) does not pass through $(0,0)$.

Similar to Proposition \ref{xxsec2.1}, we have
\begin{proposition}\label{xxsec2.4} Let $P$ be a moving point on (\ref{xxsecpla7}), and let its parameter be $h$.
\begin{enumerate}
  \item[$\mathrm{(i)}$]  For $\mathcal{D}=|\overrightarrow{OP}|$, we have
  \begin{align}\frac{ \mathrm{d}\mathcal{D} }{\mathrm{d}h}=\frac{[x(g(h)) x'(t)  +y(g(h))  y'(t)]g'(h)}{\sqrt{x^2(g(h)) +y^2(g(h))} }.\notag
\end{align}
  \item[$\mathrm{(ii)}$] The rotational speed of $\overrightarrow{OP}$ with respect to $h$ is $$\frac{g'(h)|x(g(h))y'(t) -x'(t) y(g(h)) |}{x^2(g(h)) +y^2(g(h)) }.$$
\end{enumerate}
\end{proposition}
\begin{proof} (i) Since $\mathcal{D}=\sqrt{x^2(g(h))+y^2(g(h)) }$ (equals to $\xi(g(h))$ in (\ref{xxsecpla8})), we get the result.

(ii) Note that $$\overrightarrow{e}_{\overrightarrow{OP}} = \frac{\{x(g(h))  ,y(g(h))  \}}{\sqrt{x^2(g(h))  +y^2(g(h))  }}.$$
 We therefore get
\begin{align}\frac{\mathrm{d}\overrightarrow{e}_{\overrightarrow{OP}} }{\mathrm{d}h}
 =&-\frac{1}{2}(x^2  +y^2 )^{-\frac{3}{2}}2(x  x'(t)  +y  y'(t)  )g'(h)\{x  ,y  \}\notag\\&+(x^2  +y^2 )^{-\frac{1}{2}}\{x'(t)  , y'(t)  \}g'(h)\notag\\
 =&-(x^2  +y^2 )^{-\frac{3}{2}}g'(h) (x  x'(t)  +y  y'(t)  )\{x  ,y  \}\notag\\&+(x^2  +y^2 )^{-\frac{3}{2}}g'(h)(x^2  +y^2 )\{x'(t)  , y'(t)  \}\notag\\
 =&-(x^2  +y^2 )^{-\frac{3}{2}}g'(h)\{x^2  x'(t)  + y  y'(t)x  ,
 x  x'(t)y  +y^2  y'(t)  \}\notag\\
  &+(x^2  +y^2 )^{-\frac{3}{2}}g'(h)\{x^2  x'(t)  +y^2  x'(t) , x^2  y'(t)  +y^2  y'(t)  \}\notag\\
   =&(x^2  +y^2 )^{-\frac{3}{2}}g'(h)\{ - y  y'(t)x  +y^2  x'(t)  ,
   -x  x'(t) y +x^2  y'(t)   \}\notag\\
  =&(x^2  +y^2 )^{-\frac{3}{2}}g'(h)\{y  (- y'(t)x  +y  x'(t) ),   x  (-x'(t) y +x  y'(t)  )\} \notag\\
  =&(x^2  +y^2 )^{-\frac{3}{2}}g'(h)\{y  (- x y'(t)  +x'(t)y  ),   x  (-x'(t) y +x  y'(t)  )\} \notag\\
 =&g'(h)( x y'(t)  -x'(t) y )\frac{\{ -y  ,  x   \}}{ (x^2  +y^2 )^{\frac{3}{2}}}.\notag\end{align}
It follows that
$$|\frac{\mathrm{d}\overrightarrow{e}_{\overrightarrow{OP}} }{\mathrm{d}h}|= \frac{|g'(h)(x y'(t) -  x'(t)y ) |}{x^2  +y^2  }.$$
This is exactly what we want.

\end{proof}

\subsection{Plane curves with general rotating frames 2}

The origin of frame (\ref{xxsecpla2}) is $O(0,0)$. Next we will consider the rotating frame with  a point different from it as the origin. Let us choose a point $A$ in $\mathbb{R}^2$ outside the curve (\ref{xxsecpla1}), and let its coordinates  be $(a_1,a_2)$\,($\neq(0,0)$). Set up a rotating frame at $A$ as follows:
\begin{equation}\label{xxsecpla9}
 \{A;\overrightarrow{e}_1(t), \overrightarrow{e}_2(t) \}, \ \ \ t_0 < t< t_1,
\end{equation}
where $\overrightarrow{e}_1(t)= \frac{\{x(t)-a_1,y(t)-a_2\}}{\sqrt{(x(t)-a_1)^2+(y(t)-a_2)^2}}$, $\overrightarrow{e}_2(t)=\frac{\{-y(t)+a_2,x(t)-a_1\}}{\sqrt{(x(t)-a_1)^2+(y(t)-a_2)^2}}$. Let the coordinates of (\ref{xxsecpla9}) be $\xi_1$, $\eta_1$. Then  the equation of (\ref{xxsecpla1}) under (\ref{xxsecpla9}) is
\begin{equation}\label{xxsecpla10}
  \xi_1(t)= \sqrt{(x(t)-a_1)^2+(y(t)-a_2)^2 }, \, \eta_1(t)=0,\ \ \ t_0 < t< t_1.
\end{equation}

\textbf{Note:} In this subsection, we default to  $t_0 <t< t_1$, and  (\ref{xxsecpla1}) does not include  the point $(a_1,a_2)$.

Now the behaviour of the moving points on (\ref{xxsecpla1}) can be described by the following  proposition.
\begin{proposition}\label{xxsec2.5}  Let $P$ be a moving point on (\ref{xxsecpla1}), and let its parameter be $t$.
\begin{enumerate}
  \item[$\mathrm{(i)}$] For  $\mathcal{D}=|\overrightarrow{AP}|$, we have
   \begin{align}
\frac{ \mathrm{d}\mathcal{D}}{\mathrm{d}t}
  =& \frac{(x(t) -a_1)x'(t) +(y(t) -a_2)y'(t)}{\sqrt{(x(t)-a_1)^2+(y(t) -a_2)^2}} . \notag
\end{align}
\item[$\mathrm{(ii)}$] The rotational speed of $\overrightarrow{AP}$  with respect to $t$ is $$\frac{|(x(t) -a_1)y'(t) -x'(t) (y(t) -a_2)|}{(x(t) -a_1)^2+(y(t)-a_2)^2 }.$$
\end{enumerate}
\end{proposition}
\begin{proof} (i) Evidently, we have $\mathcal{D}=\sqrt{(x(t)-a_1)^2+(y(t)-a_2)^2 }$\,(equal to $\xi_1(t)$ in (\ref{xxsecpla10})), so the result follows.

(ii) Note that $$\overrightarrow{e}_{\overrightarrow{AP}} = \frac{\{x(t) -a_1,y(t) -a_2\}}{\sqrt{(x(t) -a_1)^2+(y(t) -a_2)^2}}.$$
Therefore, using the result of (\ref{xxsecpla4}), we have
\begin{align}\frac{\mathrm{d}\overrightarrow{e}_{\overrightarrow{AP}} }{\mathrm{d}t}
=&[(x -a_1)y' -x'(y -a_2) ]\frac{\{ -(y -a_2),  (x -a_1) \}}{ [(x -a_1)^2+(y -a_2)^2]^{\frac{3}{2}}}.\notag\end{align}
Consequently,
$$|\frac{\mathrm{d}\overrightarrow{e}_{\overrightarrow{AP}} }{\mathrm{d}t}|= \frac{|(x -a_1)y' -x'(y -a_2) |}{(x -a_1)^2+(y -a_2)^2 }.$$

\end{proof}

From Proposition \ref{xxsec2.5}(i), one can derive \begin{align}
\frac{ \mathrm{d}^2\mathcal{D}}{\mathrm{d}t^2}
  =& -[(x -a_1)^2+(y -a_2)^2]^{-\frac{3}{2}}[(x -a_1)x' +(y -a_2)y' ]^2\notag\\
  &+[(x -a_1)^2+(y -a_2)^2]^{-\frac{1}{2}}[x'^2 + (x -a_1)x'' +y'^2 +(y -a_2)y''  ] .\label{xxsecpla11}
\end{align}

The comparisons of Proposition \ref{xxsec2.1}(i) and Proposition \ref{xxsec2.5}(i) (resp. (\ref{xxsecpla5}) and (\ref{xxsecpla11})), as well as Proposition \ref{xxsec2.1}(ii) and Proposition \ref{xxsec2.5}(ii),  show that first-order (resp. second-order) derivatives and rotational speed provided by  $O$ and $A$ for a given point of (\ref{xxsecpla1}) are not the same in  general. Assuming the origin of the rotating frame is the  power source, then selecting a suitable origin for a rotating frame is a matter worthy of consideration.

\subsection{Plane curves with local rotating frames}

Let the coordinates of point $P$ on (\ref{xxsecpla1}) be $(x(t),y(t))$. We set up a rotating frame at $P$ as follows:
\begin{equation}\label{xxsecpla12}
 \{P(t);\overrightarrow{e}_1(\Delta t), \overrightarrow{e}_2(\Delta t) \}, \ \ \ t_0 < t< t_1,
\end{equation}
where $\overrightarrow{e}_1(\Delta t)= \frac{\{x(t+\Delta t)-x(t),y(t+\Delta t)-y(t)\}}{\sqrt{(x(t+\Delta t)-x(t))^2+(y(t+\Delta t)-y(t))^2}}$ and $\overrightarrow{e}_2(\Delta t)=\frac{\{-y(t+\Delta t)+y(t),x(t+\Delta t)-x(t)\}}{\sqrt{(x(t+\Delta t)-x(t))^2+(y(t+\Delta t)-y(t))^2}}$ with $\Delta t> 0$.
Frame (\ref{xxsecpla12}) is a local rotating frame of (\ref{xxsecpla1}). If we denote the coordinates of (\ref{xxsecpla12}) by $\xi_P$ and $\eta_P$, then the equation of (\ref{xxsecpla1}) near point $P$ under (\ref{xxsecpla12}) can be written as
\begin{equation}\label{xxsecpla13}
  \xi_P(\Delta t)= \sqrt{f_1^2(\Delta t)+f_2^2(\Delta t) }, \, \eta_P(\Delta t)=0,
\end{equation}
where $f_1(\Delta t)=x(t+\Delta t)-x(t)$ and $f_2(\Delta t)=y(t+\Delta t)-y(t)$ with $\Delta t> 0$.

\textbf{Note:} In this subsection, We assume by default that $t_0 <t< t_1$, $\Delta t>0$ is sufficiently small. Moreover, (\ref{xxsecpla1}) is regular.

We are now ready to consider the behaviour of the point $Q$ on (\ref{xxsecpla1}) near  $P$.
\begin{proposition}\label{xxsec2.6} Let $Q$ be a moving point on (\ref{xxsecpla1}) near  $P$, and let its parameter be $\Delta t$.
\begin{enumerate}
  \item[$\mathrm{(i)}$] For $\mathcal{D}_P=|\overrightarrow{PQ}|$, we have
\begin{align} \frac{ \mathrm{d}\mathcal{D}_P }{\mathrm{d}\Delta t}
=&\frac{f_1(\Delta t) x'(t+\Delta t) +f_2(\Delta t)  y'(t+\Delta t)}{\sqrt{f_1^2(\Delta t)+f_2^2(\Delta t)}} ,\ \Delta t>0; \notag\\
 {\mathcal{D}_{P}'}_+(0)&=\sqrt{x'^2(t)+y'^2(t)} .\notag\end{align}
\item[$\mathrm{(ii)}$] The rotational speed of $\overrightarrow{PQ}$ with respect to $\Delta t$ is
 \begin{align}&\frac{|y'(t+\Delta t) f_1(\Delta t) -f_2(\Delta t)x'(t+\Delta t) |}{f_1^2(\Delta t)+f_2^2(\Delta t) }, \ \ \Delta t>0;\notag\\
 &\frac{1}{2}|\frac{x'(t)y''(t)-x''(t)y'(t)}{x'^2(t)+ y'^2(t)}|,\  \Delta t\rightarrow 0^+\notag.\end{align}
\end{enumerate}
\end{proposition}
\begin{proof} (i) Since $\mathcal{D}_P=\sqrt{f_1^2(\Delta t)+f_2^2(\Delta t)}$\,(equals to $\xi_P(\Delta t)$ in (\ref{xxsecpla13})), we have
\begin{align}\frac{ \mathrm{d}\mathcal{D}_P }{\mathrm{d}\Delta t}=&\frac{1}{2}(f_1^2(\Delta t)+f_2^2(\Delta t))^{-\frac{1}{2}}(2(f_1(\Delta t) f_1'(\Delta t) +2f_2(\Delta t)f_2'(\Delta t))\notag\\
=&(f_1^2(\Delta t)+f_2^2(\Delta t))^{-\frac{1}{2}}(f_1(\Delta t) x'(t+\Delta t) +f_2(\Delta t)  y'(t+\Delta t) )\notag, \end{align}
where $\Delta t>0$.

Next let us consider  $\Delta t\rightarrow 0^+$. The following Taylor expansions will be used.
\begin{align}x(t+\Delta t)-x(t)=&x'(t)\Delta t+\frac{1}{2}x''(\varsigma_1)\Delta t^2,\label{xxsecpla14}\\
y(t+\Delta t)-y(t)=&y'(t)\Delta t+\frac{1}{2}y''(\varsigma_2)\Delta t^2,\label{xxsecpla15}
\end{align}
where $\varsigma_1, \varsigma_2$ are between $t$ and $t+\Delta t$.
Now we can write
\begin{align}\mathcal{D}_P&=\sqrt{(x'(t)\Delta t+\frac{1}{2}x''(\varsigma_1)\Delta t^2)^2+(y'(t)\Delta t+\frac{1}{2}y''(\varsigma_2)\Delta t^2)^2}\notag\\
&=
\Delta t\sqrt{(x'(t)+\frac{1}{2}x''(\varsigma_1)\Delta t)^2+(y'(t) +\frac{1}{2}y''(\varsigma_2)\Delta t)^2}\notag, \end{align}
where $\Delta t>0$. Hence
\begin{align}{\mathcal{D}_{P}'}_+(0)
 &=\lim_{\Delta t\rightarrow 0^+}\frac{\mathcal{D}_P(\Delta t)-\mathcal{D}_P(0)}{\Delta t}=\sqrt{x'^2(t)+y'^2(t)},\label{xxsecpla16}
\end{align}
where $\mathcal{D}_P(0)=0$.

(ii) We begin by noticing that$$\overrightarrow{e}_{\overrightarrow{PQ}} = \frac{\{f_1(\Delta t) ,f_2(\Delta t) \}}{\sqrt{(f_1^2(\Delta t)+f_2^2(\Delta t)) }}.$$
Therefor, for $\Delta t>0$, we have
\begin{align}\frac{\mathrm{d}\overrightarrow{e}_{\overrightarrow{PQ}} }{\mathrm{d}\Delta t}
 =&-\frac{1}{2}(f_1^2 +f_2^2 )^{-\frac{3}{2}}\times2[f_1  x'(t+\Delta t) +f_2  y'(t+\Delta t) ]
\{f_1  ,f_2  \}\notag\\&+(f_1^2 +f_2^2 )^{-\frac{1}{2}}\{x'(t+\Delta t) , y'(t+\Delta t) \}\notag\\
 =&-(f_1^2 +f_2^2 )^{-\frac{3}{2}}[f_1  x'(t+\Delta t) +f_2 y'(t+\Delta t) ]
 \{f_1  ,f_2  \}\notag\\
&+(f_1^2 +f_2^2 )^{-\frac{3}{2}} (f_1^2 +f_2^2  )\{x'(t+\Delta t) , y'(t+\Delta t) \}\notag\\
 =&-(f_1^2 +f_2^2 )^{-\frac{3}{2}}\notag\\
&\times
 \{f_1^2  x'(t+\Delta t) + f_2  y'(t+\Delta t)f_1 ,
f_1  x'(t+\Delta t)f_2  +f_2^2  y'(t+\Delta t) \}\notag\\
  &+(f_1^2 +f_2^2 )^{-\frac{3}{2}}\notag\\
&\times
  \{f_1^2  x'(t+\Delta t) +f_2^2 x'(t+\Delta t),
 f_1^2  y'(t+\Delta t) +f_2^2  y'(t+\Delta t) \}\notag\\
   =&(f_1^2 +f_2^2 )^{-\frac{3}{2}} \notag\\
&\times\{ -f_2 y'(t+\Delta t)f_1  +f_2^2  x'(t+\Delta t),
 -f_1  x'(t+\Delta t)f_2  +f_1^2 y'(t+\Delta t)  \}\notag\\
  =& (f_1^2 +f_2^2 )^{-\frac{3}{2}} \notag\\
&\times \{f_2  (- y'(t+\Delta t)f_1  +f_2  x'(t+\Delta t)),   f_1  (- x'(t+\Delta t) f_2 +f_1  y'(t+\Delta t) )\}\notag\\
 =& (f_1^2 +f_2^2 )^{-\frac{3}{2}} \notag\\
&\times \{f_2  (- y'(t+\Delta t)f_1  +f_2  x'(t+\Delta t)),  f_1  (- f_2 x'(t+\Delta t)+y'(t+\Delta t)f_1  )\}\notag\\
=&(y'(t+\Delta t)f_1  -f_2 x'(t+\Delta t) ) \frac{\{ -f_2  ,  f_1  \}}{ (f_1^2 +f_2^2 )^{\frac{3}{2}}},\notag\end{align}
 $$|\frac{\mathrm{d}\overrightarrow{e}_{\overrightarrow{PQ}} }{\mathrm{d}\Delta t}|= \frac{|y'(t+\Delta t) f_1(\Delta t) -f_2(\Delta t)x'(t+\Delta t) |}{f_1^2(\Delta t)+f_2^2(\Delta t) }.$$

Next let us consider $\Delta t\rightarrow 0^+$. By (\ref{xxsecpla14}) and (\ref{xxsecpla15}), we know that
 \begin{align}\overrightarrow{e}_{\overrightarrow{PQ}}=& \frac{\{x'(t)\Delta t+\frac{1}{2}x''(\varsigma_1)\Delta t^2,y'(t)\Delta t+\frac{1}{2}y''(\varsigma_2)\Delta t^2\}}{\sqrt{(x'(t)\Delta t+\frac{1}{2}x''(\varsigma_1)\Delta t^2)^2+(y'(t)\Delta t+\frac{1}{2}y''(\varsigma_2)\Delta t^2)^2}}\notag\\
=&
\frac{\{x'(t)+\frac{1}{2}x''(\varsigma_1)\Delta t,y'(t)+\frac{1}{2}y''(\varsigma_2)\Delta t\}}{\sqrt{(x'(t)+\frac{1}{2}x''(\varsigma_1)\Delta t)^2+(y'(t)+\frac{1}{2}y''(\varsigma_2)\Delta t)^2}},\notag\end{align}
where $\Delta t>0$. Hence
\begin{align}&\lim_{\Delta t\rightarrow 0^+}\frac{\overrightarrow{e}_{\overrightarrow{PQ}}(\Delta t)-\overrightarrow{e}_{\overrightarrow{PQ}}(0)}{\Delta t}\notag\\
 =&\lim_{\Delta t\rightarrow 0^+} \frac{\frac{\{x'(t)+\frac{1}{2}x''(\varsigma_1)\Delta t,y'(t)+\frac{1}{2}y''(\varsigma_2)\Delta t\}}{\sqrt{(x'(t)+\frac{1}{2}x''(\varsigma_1)\Delta t)^2+(y'(t)+\frac{1}{2}y''(\varsigma_2)\Delta t)^2}}-\frac{\{x'(t),y'(t)\}}{\sqrt{x'^2(t)+y'^2(t)}}}{\Delta t}
  .\label{xxsecpla17}
\end{align}
Before continuing, we need to carry out the following calculations first:
\begin{align}&\lim_{\Delta t\rightarrow 0^+} \frac{\frac{x'(t)+\frac{1}{2}x''(\varsigma_1)\Delta t}{\sqrt{(x'(t)+\frac{1}{2}x''(\varsigma_1)\Delta t)^2+(y'(t)+\frac{1}{2}y''(\varsigma_2)\Delta t)^2}}-\frac{x'(t)}{\sqrt{x'^2(t)+y'^2(t)}}}{\Delta t}\notag\\
 =&\lim_{\Delta t\rightarrow 0^+} \frac{\frac{(x'(t)+\frac{1}{2}x''(\varsigma_1)\Delta t)^2}{(x'(t)+\frac{1}{2}x''(\varsigma_1)\Delta t)^2+(y'(t)+\frac{1}{2}y''(\varsigma_2)\Delta t)^2}-\frac{x'^2(t)}{x'^2(t)+y'^2(t)}}{\Delta t[\frac{x'(t)+\frac{1}{2}x''(\varsigma_1)\Delta t}{\sqrt{(x'(t)+\frac{1}{2}x''(\varsigma_1)\Delta t)^2+(y'(t)+\frac{1}{2}y''(\varsigma_2)\Delta t)^2}}+\frac{x'(t)}{\sqrt{x'^2(t)+y'^2(t)}}]}
 \notag\\
  =&\lim_{\Delta t\rightarrow 0^+}\frac{1}{\Delta t}[\frac{x'(t)+\frac{1}{2}x''(\varsigma_1)\Delta t}{\sqrt{(x'(t)+\frac{1}{2}x''(\varsigma_1)\Delta t)^2+(y'(t)+\frac{1}{2}y''(\varsigma_2)\Delta t)^2}}+\frac{x'(t)}{\sqrt{x'^2(t)+y'^2(t)}}]^{-1}\notag\\&\times
  [(x'(t)+\frac{1}{2}x''(\varsigma_1)\Delta t)^2+(y'(t)+\frac{1}{2}y''(\varsigma_2)\Delta t)^2]^{-1}(x'^2(t)+y'^2(t))^{-1}\notag\\&\times
  [(x'(t)+\frac{1}{2}x''(\varsigma_1)\Delta t)^2(x'^2(t)+y'^2(t))\notag\\&\ \ \ -x'^2(t)(x'(t)+\frac{1}{2}x''(\varsigma_1)\Delta t)^2-x'^2(t)(y'(t)+\frac{1}{2}y''(\varsigma_2)\Delta t)^2]\notag\\
  =&\lim_{\Delta t\rightarrow 0^+}\frac{1}{\Delta t}[\frac{x'(t)+\frac{1}{2}x''(\varsigma_1)\Delta t}{\sqrt{(x'(t)+\frac{1}{2}x''(\varsigma_1)\Delta t)^2+(y'(t)+\frac{1}{2}y''(\varsigma_2)\Delta t)^2}}+\frac{x'(t)}{\sqrt{x'^2(t)+y'^2(t)}}]^{-1}\notag\\&\times
  [(x'(t)+\frac{1}{2}x''(\varsigma_1)\Delta t)^2+(y'(t)+\frac{1}{2}y''(\varsigma_2)\Delta t)^2]^{-1}(x'^2(t)+y'^2(t))^{-1}\notag\\&\times
   [(x'(t)x''(\varsigma_1)\Delta t+\frac{1}{4}x''^2(\varsigma_1)\Delta t^2)(x'^2(t)+y'^2(t))\notag\\&
   \ \ \ -x'^2(t)(x'(t)x''(\varsigma_1)\Delta t+\frac{1}{4}x''^2(\varsigma_1)\Delta t^2+y'(t)y''(\varsigma_2)\Delta t+\frac{1}{4}y''^2(\varsigma_2)\Delta t^2)]\notag\\
  =&\lim_{\Delta t\rightarrow 0^+}[\frac{x'(t)+\frac{1}{2}x''(\varsigma_1)\Delta t}{\sqrt{(x'(t)+\frac{1}{2}x''(\varsigma_1)\Delta t)^2+(y'(t)+\frac{1}{2}y''(\varsigma_2)\Delta t)^2}}+\frac{x'(t)}{\sqrt{x'^2(t)+y'^2(t)}}]^{-1}\notag\\&\times
  [(x'(t)+\frac{1}{2}x''(\varsigma_1)\Delta t)^2+(y'(t)+\frac{1}{2}y''(\varsigma_2)\Delta t)^2]^{-1}(x'^2(t)+y'^2(t))^{-1}\notag\\&\times
   [(x'(t)x''(\varsigma_1)+\frac{1}{4}x''^2(\varsigma_1)\Delta t)(x'^2(t)+y'^2(t))\notag\\&
   \ \ \ -x'^2(t)(x'(t)x''(\varsigma_1)+\frac{1}{4}x''^2(\varsigma_1)\Delta t+y'(t)y''(\varsigma_2) +\frac{1}{4}y''^2(\varsigma_2)\Delta t)]\notag\\
   =& (\frac{2x'(t)}{\sqrt{x'^2(t)+y'^2(t)}})^{-1}
  (x'^2(t)+y'^2(t))^{-2}\notag\\&\times
   [x'(t)x''(t)(x'^2(t)+y'^2(t))-x'^2(t)(x'(t)x''(t)+y'(t)y''(t) )]\notag\\
   =& \frac{1}{2} (x'(t))^{-1}(x'^2(t)+y'^2(t))^{\frac{1}{2}}
  (x'^2(t)+y'^2(t))^{-2}\notag\\&\times
   [x'(t)x''(t)(x'^2(t)+y'^2(t))-x'^2(t)(x'(t)x''(t)+y'(t)y''(t) )]\notag\\
   =& \frac{1}{2} (x'^2(t)+y'^2(t))^{-\frac{3}{2}}
   [x''(t)(x'^2(t)+y'^2(t))-x'(t)(x'(t)x''(t)+y'(t)y''(t) )]\notag\\
    =& \frac{1}{2}
  (x'^2(t)+y'^2(t))^{-\frac{3}{2}}
   (x''(t)y'^2(t)-x'(t)y'(t)y''(t) )\notag\\
     =& \frac{1}{2}
  (x'^2(t)+y'^2(t))^{-\frac{3}{2}}
   y'(t)(x''(t)y'(t)-x'(t)y''(t) )
  ,\label{xxsecpla18}\end{align}
  \begin{align}
&\lim_{\Delta t\rightarrow 0^+} \frac{\frac{y'(t)+\frac{1}{2}y''(\varsigma_2)\Delta t}{\sqrt{(x'(t)+\frac{1}{2}x''(\varsigma_1)\Delta t)^2+(y'(t)+\frac{1}{2}y''(\varsigma_2)\Delta t)^2}}-\frac{y'(t)}{\sqrt{x'^2(t)+y'^2(t)}}}{\Delta t}\notag\\
 =&\lim_{\Delta t\rightarrow 0^+} \frac{\frac{(y'(t)+\frac{1}{2}y''(\varsigma_2)\Delta t)^2}{(x'(t)+\frac{1}{2}x''(\varsigma_1)\Delta t)^2+(y'(t)+\frac{1}{2}y''(\varsigma_2)\Delta t)^2}-\frac{y'^2(t)}{x'^2(t)+y'^2(t)}}{\Delta t[\frac{y'(t)+\frac{1}{2}y''(\varsigma_2)\Delta t}{\sqrt{(x'(t)+\frac{1}{2}x''(\varsigma_1)\Delta t)^2+(y'(t)+\frac{1}{2}y''(\varsigma_2)\Delta t)^2}}+\frac{y'(t)}{\sqrt{x'^2(t)+y'^2(t)}}]}
 \notag\\
  =&\lim_{\Delta t\rightarrow 0^+}\frac{1}{\Delta t}[\frac{y'(t)+\frac{1}{2}y''(\varsigma_2)\Delta t}{\sqrt{(x'(t)+\frac{1}{2}x''(\varsigma_1)\Delta t)^2+(y'(t)+\frac{1}{2}y''(\varsigma_2)\Delta t)^2}}+\frac{y'(t)}{\sqrt{x'^2(t)+y'^2(t)}}]^{-1}\notag\\&\times
  [(x'(t)+\frac{1}{2}x''(\varsigma_1)\Delta t)^2+(y'(t)+\frac{1}{2}y''(\varsigma_2)\Delta t)^2]^{-1}(x'^2(t)+y'^2(t))^{-1}\notag\\&\times
  [(y'(t)+\frac{1}{2}y''(\varsigma_2)\Delta t)^2(x'^2(t)+y'^2(t))\notag\\&\ \ \ -y'^2(t)(x'(t)+\frac{1}{2}x''(\varsigma_1)\Delta t)^2-y'^2(t)(y'(t)+\frac{1}{2}y''(\varsigma_2)\Delta t)^2]\notag\\
  =&\lim_{\Delta t\rightarrow 0^+}\frac{1}{\Delta t}[\frac{y'(t)+\frac{1}{2}y''(\varsigma_2)\Delta t}{\sqrt{(x'(t)+\frac{1}{2}x''(\varsigma_1)\Delta t)^2+(y'(t)+\frac{1}{2}y''(\varsigma_2)\Delta t)^2}}+\frac{y'(t)}{\sqrt{x'^2(t)+y'^2(t)}}]^{-1}\notag\\&\times
  [(x'(t)+\frac{1}{2}x''(\varsigma_1)\Delta t)^2+(y'(t)+\frac{1}{2}y''(\varsigma_2)\Delta t)^2]^{-1}(x'^2(t)+y'^2(t))^{-1}\notag\\&\times
   [(y'(t)y''(\varsigma_2)\Delta t+\frac{1}{4}y''^2(\varsigma_2)\Delta t^2)(x'^2(t)+y'^2(t))\notag\\&
   \ \ \ -y'^2(t)(x'(t)x''(\varsigma_1)\Delta t+\frac{1}{4}x''^2(\varsigma_1)\Delta t^2+y'(t)y''(\varsigma_2)\Delta t+\frac{1}{4}y''^2(\varsigma_2)\Delta t^2)]\notag\\
 =&\lim_{\Delta t\rightarrow 0^+}[\frac{y'(t)+\frac{1}{2}y''(\varsigma_2)\Delta t}{\sqrt{(x'(t)+\frac{1}{2}x''(\varsigma_1)\Delta t)^2+(y'(t)+\frac{1}{2}y''(\varsigma_2)\Delta t)^2}}+\frac{y'(t)}{\sqrt{x'^2(t)+y'^2(t)}}]^{-1}\notag\\&\times
  [(x'(t)+\frac{1}{2}x''(\varsigma_1)\Delta t)^2+(y'(t)+\frac{1}{2}y''(\varsigma_2)\Delta t)^2]^{-1}(x'^2(t)+y'^2(t))^{-1}\notag\\&\times
   [(y'(t)y''(\varsigma_2)+\frac{1}{4}y''^2(\varsigma_2)\Delta t)(x'^2(t)+y'^2(t))\notag\\&
   \ \ \ -y'^2(t)(x'(t)x''(\varsigma_1)+\frac{1}{4}x''^2(\varsigma_1)\Delta t+y'(t)y''(\varsigma_2) +\frac{1}{4}y''^2(\varsigma_2)\Delta t)]\notag\\
   =& (\frac{2y'(t)}{\sqrt{x'^2(t)+y'^2(t)}})^{-1}
  (x'^2(t)+y'^2(t))^{-2}\notag\\&\times
   [y'(t)y''(t)(x'^2(t)+y'^2(t))-y'^2(t)(x'(t)x''(t)+y'(t)y''(t) )]\notag\\
=& \frac{1}{2}(y'(t))^{-1}(x'^2(t)+y'^2(t))^{\frac{1}{2}}
  (x'^2(t)+y'^2(t))^{-2}\notag\\&\times
   [y'(t)y''(t)(x'^2(t)+y'^2(t))-y'^2(t)(x'(t)x''(t)+y'(t)y''(t) )]\notag\\
   =& \frac{1}{2}(x'^2(t)+y'^2(t))^{-\frac{3}{2}}
   [y''(t)(x'^2(t)+y'^2(t))-y'(t)(x'(t)x''(t)+y'(t)y''(t) )]\notag\\
=& \frac{1}{2}
  (x'^2(t)+y'^2(t))^{-\frac{3}{2}}
   (y''(t)x'^2(t)-y'(t)x'(t)x''(t) )\notag\\
     =& \frac{1}{2}
  (x'^2(t)+y'^2(t))^{-\frac{3}{2}}
   x'(t)(y''(t)x'(t)-y'(t)x''(t) )
  .\label{xxsecpla19}
\end{align}
Taking (\ref{xxsecpla18}) and (\ref{xxsecpla19}) into (\ref{xxsecpla17}), we arrive at
\begin{align}&\lim_{\Delta t\rightarrow 0^+}\frac{\overrightarrow{e}_{\overrightarrow{PQ}}(\Delta t)-\overrightarrow{e}_{\overrightarrow{PQ}}(0)}{\Delta t}\notag\\
=&\frac{1}{2}(x'^2(t)+ y'^2(t))^{-\frac{3}{2}}
\{ y'(t)(x''(t)y'(t)-x'(t)y''(t)),x'(t)(y''(t)x'(t)-y'(t)x''(t))\}\notag\\
=&\frac{1}{2}(x'^2(t)+ y'^2(t))^{-\frac{3}{2}}
\{ y'(t)(x''(t)y'(t)-x'(t)y''(t)),x'(t)(x'(t)y''(t)-x''(t)y'(t))\}\notag\\
=&
\frac{(x'(t)y''(t)-x''(t)y'(t))\{- y'(t),x'(t)\}}{2(x'^2(t)+ y'^2(t))^{\frac{3}{2}}}
,\label{xxsecpla20}\end{align}
whence
$$|\lim_{\Delta t\rightarrow 0^+}\frac{\overrightarrow{e}_{\overrightarrow{PQ}}(\Delta t)-\overrightarrow{e}_{\overrightarrow{PQ}}(0)}{\Delta t}|
=\frac{1}{2}\frac{|x'(t)y''(t)-x''(t)y'(t)|}{x'^2(t)+ y'^2(t)}
.$$

\end{proof}

\textbf{Note:} We call (\ref{xxsecpla16}) the local first-order derivative
and (\ref{xxsecpla20}) the local rotational velocity function of (\ref{xxsecpla1}). For convenience,
we will denote (\ref{xxsecpla16}) by $\phi(t)$  and (\ref{xxsecpla20}) by $\overrightarrow{\psi}(t)$, namely,
\begin{align}
 \phi(t)=&\sqrt{x'^2(t)+y'^2(t)},\notag\\
\overrightarrow{\psi}(t)=&\frac{(x'(t)y''(t)-x''(t)y'(t))\{- y'(t),x'(t)\}}{2(x'^2(t)+ y'^2(t))^{\frac{3}{2}}}
.\notag\end{align}

In fact, we can further discuss the local second-order derivative  of (\ref{xxsecpla1}), and the specific process is as follows: In light of the conclusion from Proposition \ref{xxsec2.6}(i) (the case of $\Delta t>0$), we have
 \begin{align}
\frac{ \mathrm{d}^2\mathcal{D}_P}{\mathrm{d}(\Delta t)^2}
  =& -(f_1^2(\Delta t)+f_2^2(\Delta t))^{-\frac{3}{2}}
  (f_1(\Delta t)x'(t+\Delta t) +f_2(\Delta t)  y'(t+\Delta t) )^2\notag\\
  &+(f_1^2(\Delta t)+f_2^2(\Delta t))^{-\frac{1}{2}}\notag\\
&\times(x'^2(t+\Delta t) + f_1(\Delta t) x''(t+\Delta t) +y'^2(t+\Delta t) +f_2(\Delta t) y''(t+\Delta t)),\notag
\end{align}
where $\Delta t>0$. For further calculation, besides (\ref{xxsecpla14}) and (\ref{xxsecpla15}), we also need
\begin{align}x'(t+\Delta t)-x'(t)=&x''(o_1)\Delta t,\notag\\
y'(t+\Delta t)-y'(t)=&y''(o_2)\Delta t,\notag
\end{align}
where $o_1, o_2$ are between $t$ and $t+\Delta t$. Then
\begin{align}
&\frac{ \mathrm{d}^2\mathcal{D}_P}{\mathrm{d}(\Delta t)^2}\notag\\
  =& -[(x'(t)\Delta t+\frac{1}{2}x''(\varsigma_1)\Delta t^2)^2+(y'(t)\Delta t+\frac{1}{2}y''(\varsigma_2)\Delta t^2)^2]^{-\frac{3}{2}}\notag\\&\times
 [(x'(t)\Delta t+\frac{1}{2}x''(\varsigma_1)\Delta t^2)(x'(t)+x''(o_1)\Delta t) +(y'(t)\Delta t+\frac{1}{2}y''(\varsigma_2)\Delta t^2)(y'(t)+y''(o_2) ]^2\notag\\
  &+[(x'(t)\Delta t+\frac{1}{2}x''(\varsigma_1)\Delta t^2)^2+(y'(t)\Delta t+\frac{1}{2}y''(\varsigma_2)\Delta t^2)^2]^{-\frac{1}{2}}\notag\\
&\times[(x'(t)+x''(o_1)\Delta t)^2 + (x'(t)\Delta t+\frac{1}{2}x''(\varsigma_1)\Delta t^2) x''(t+\Delta t) \notag\\&\ \ \ +(y'(t)+y''(o_2)\Delta t)^2 +(y'(t)\Delta t+\frac{1}{2}y''(\varsigma_2)\Delta t^2) y''(t+\Delta t)]\notag\\
 =& -\Delta t^{-3}[(x'(t)+\frac{1}{2}x''(\varsigma_1)\Delta t)^2+(y'(t)+\frac{1}{2}y''(\varsigma_2)\Delta t)^2]^{-\frac{3}{2}}\notag\\&\times
 \Delta t^{2}[(x'(t)+\frac{1}{2}x''(\varsigma_1)\Delta t)(x'(t)+x''(o_1)\Delta t) +(y'(t)+\frac{1}{2}y''(\varsigma_2)\Delta t)(y'(t)+y''(o_2) ]^2\notag\\
  &+\Delta t^{-1}[(x'(t)+\frac{1}{2}x''(\varsigma_1)\Delta t)^2+(y'(t)+\frac{1}{2}y''(\varsigma_2)\Delta t)^2]^{-\frac{1}{2}}\notag\\
&\times[(x'(t)+x''(o_1)\Delta t)^2 + (x'(t)\Delta t+\frac{1}{2}x''(\varsigma_1)\Delta t^2) x''(t+\Delta t)\notag\\&\ \ \  +(y'(t)+y''(o_2)\Delta t)^2 +(y'(t)\Delta t+\frac{1}{2}y''(\varsigma_2)\Delta t^2) y''(t+\Delta t)]\notag\\
=& -\Delta t^{-1}[(x'(t)+\frac{1}{2}x''(\varsigma_1)\Delta t)^2+(y'(t)+\frac{1}{2}y''(\varsigma_2)\Delta t)^2]^{-\frac{3}{2}}\notag\\&\times
 [(x'(t)+\frac{1}{2}x''(\varsigma_1)\Delta t)(x'(t)+x''(o_1)\Delta t) +(y'(t)+\frac{1}{2}y''(\varsigma_2)\Delta t)(y'(t)+y''(o_2)\Delta t) ]^2\notag\\
  &+\Delta t^{-1}[(x'(t)+\frac{1}{2}x''(\varsigma_1)\Delta t)^2+(y'(t)+\frac{1}{2}y''(\varsigma_2)\Delta t)^2]^{-\frac{3}{2}}\notag\\
&\times[(x'(t)+\frac{1}{2}x''(\varsigma_1)\Delta t)^2+(y'(t)+\frac{1}{2}y''(\varsigma_2)\Delta t)^2]\notag\\
&\times
[(x'(t)+x''(o_1)\Delta t)^2  +(y'(t)+y''(o_2)\Delta t)^2 ]\notag\\
&+\Delta t^{-1}[(x'(t)+\frac{1}{2}x''(\varsigma_1)\Delta t)^2+(y'(t)+\frac{1}{2}y''(\varsigma_2)\Delta t)^2]^{-\frac{1}{2}}\notag\\
&\times
[ (x'(t)\Delta t+\frac{1}{2}x''(\varsigma_1)\Delta t^2) x''(t+\Delta t)  +(y'(t)\Delta t+\frac{1}{2}y''(\varsigma_2)\Delta t^2) y''(t+\Delta t)]\notag\\
=& -\Delta t^{-1}[(x'(t)+\frac{1}{2}x''(\varsigma_1)\Delta t)^2+(y'(t)+\frac{1}{2}y''(\varsigma_2)\Delta t)^2]^{-\frac{3}{2}}\notag\\&\times
2(x'(t)+\frac{1}{2}x''(\varsigma_1)\Delta t)(x'(t)+x''(o_1)\Delta t)(y'(t)+\frac{1}{2}y''(\varsigma_2)\Delta t)(y'(t)+y''(o_2)\Delta t) \notag\\
  &+\Delta t^{-1}[(x'(t)+\frac{1}{2}x''(\varsigma_1)\Delta t)^2+(y'(t)+\frac{1}{2}y''(\varsigma_2)\Delta t)^2]^{-\frac{3}{2}}\notag\\
&\times[(x'(t)+\frac{1}{2}x''(\varsigma_1)\Delta t)^2(y'(t)+y''(o_2)\Delta t)^2+(y'(t)+\frac{1}{2}y''(\varsigma_2)\Delta t)^2(x'(t)+x''(o_1)\Delta t)^2 ]\notag\\
&+[(x'(t)+\frac{1}{2}x''(\varsigma_1)\Delta t)^2+(y'(t)+\frac{1}{2}y''(\varsigma_2)\Delta t)^2]^{-\frac{1}{2}}\notag\\
&\times
[ (x'(t)+\frac{1}{2}x''(\varsigma_1)\Delta t) x''(t+\Delta t)  +(y'(t)+\frac{1}{2}y''(\varsigma_2)\Delta t) y''(t+\Delta t)]\notag\\
=& -\Delta t^{-1}[(x'(t)+\frac{1}{2}x''(\varsigma_1)\Delta t)^2+(y'(t)+\frac{1}{2}y''(\varsigma_2)\Delta t)^2]^{-\frac{3}{2}}\notag\\&\times
2(x'^2(t)y'^2(t)+\frac{1}{2}x''(\varsigma_1)\Delta t x'(t)y'^2(t)+x''(o_1)\Delta t x'(t)y'^2(t)\notag\\&\ \ \ +\frac{1}{2}y''(\varsigma_2)\Delta t x'^2(t)y'(t)+y''(o_2)\Delta tx'^2(t)y'(t)+\cdots)\notag\\
  &+\Delta t^{-1}[(x'(t)+\frac{1}{2}x''(\varsigma_1)\Delta t)^2+(y'(t)+\frac{1}{2}y''(\varsigma_2)\Delta t)^2]^{-\frac{3}{2}}\notag\\
&\times[(x'(t)+\frac{1}{2}x''(\varsigma_1)\Delta t)(x'(t)+\frac{1}{2}x''(\varsigma_1)\Delta t)(y'(t)+y''(o_2)\Delta t)(y'(t)+y''(o_2)\Delta t)\notag\\&\ \ \ +(y'(t)+\frac{1}{2}y''(\varsigma_2)\Delta t)(y'(t)+\frac{1}{2}y''(\varsigma_2)\Delta t)(x'(t)+x''(o_1)\Delta t)(x'(t)+x''(o_1)\Delta t) ]\notag\\
&+[(x'(t)+\frac{1}{2}x''(\varsigma_1)\Delta t)^2+(y'(t)+\frac{1}{2}y''(\varsigma_2)\Delta t)^2]^{-\frac{1}{2}}\notag\\
&\times
[ (x'(t)+\frac{1}{2}x''(\varsigma_1)\Delta t) x''(t+\Delta t)  +(y'(t)+\frac{1}{2}y''(\varsigma_2)\Delta t) y''(t+\Delta t)]\notag\\
=& -\Delta t^{-1}[(x'(t)+\frac{1}{2}x''(\varsigma_1)\Delta t)^2+(y'(t)+\frac{1}{2}y''(\varsigma_2)\Delta t)^2]^{-\frac{3}{2}}\notag\\&\times
(2x'^2(t)y'^2(t)+x''(\varsigma_1)\Delta t x'(t)y'^2(t)+2x''(o_1)\Delta t x'(t)y'^2(t)\notag\\&\ \ \ +y''(\varsigma_2)\Delta t x'^2(t)y'(t)+2y''(o_2)\Delta tx'^2(t)y'(t)+\cdots)\notag\\
  &+\Delta t^{-1}[(x'(t)+\frac{1}{2}x''(\varsigma_1)\Delta t)^2+(y'(t)+\frac{1}{2}y''(\varsigma_2)\Delta t)^2]^{-\frac{3}{2}}\notag\\
&\times(x'^2(t)y'^2(t) +x''(\varsigma_1)\Delta tx'(t)y'^2(t) +2y''(o_2)\Delta tx'^2(t)y'(t)\notag\\&\ \ \
 +x'^2(t)y'^2(t)+y''(\varsigma_2)\Delta tx'^2(t)y'(t) +2x''(o_1)\Delta t x'(t)y'^2(t) +\cdots)\notag\\
&+[(x'(t)+\frac{1}{2}x''(\varsigma_1)\Delta t)^2+(y'(t)+\frac{1}{2}y''(\varsigma_2)\Delta t)^2]^{-\frac{1}{2}}\notag\\
&\times
[ (x'(t)+\frac{1}{2}x''(\varsigma_1)\Delta t) x''(t+\Delta t)  +(y'(t)+\frac{1}{2}y''(\varsigma_2)\Delta t) y''(t+\Delta t)],\notag
\end{align}
where each $\cdots$ represents the higher-order terms (quadratic and above) of $\Delta t$. Taking the limit of the above expression yields
\begin{align}\lim_{\Delta t\rightarrow 0^+}\frac{ \mathrm{d}^2\mathcal{D}_P}{\mathrm{d}(\Delta t)^2}
=&\lim_{\Delta t\rightarrow 0^+}[(x'(t)+\frac{1}{2}x''(\varsigma_1)\Delta t)^2+(y'(t)+\frac{1}{2}y''(\varsigma_2)\Delta t)^2]^{-\frac{1}{2}}\notag\\
&\times
[ (x'(t)+\frac{1}{2}x''(\varsigma_1)\Delta t) x''(t+\Delta t)  +(y'(t)+\frac{1}{2}y''(\varsigma_2)\Delta t) y''(t+\Delta t)]\notag\\
=&\frac{x'(t)x''(t) +y'(t)y''(t)}{\sqrt{x'^2(t)+y'^2(t)}},\label{xxsecpla21}
\end{align}
which is the desired result. It should be indicated that (\ref{xxsecpla21}) is exactly the derivative of (\ref{xxsecpla16}).

At the end of this subsection, we will discuss about the  uniqueness issue. Here is our  main result.

\begin{theorem}\label{xxsec2.7}  Given the following  two expressions associated with $\overrightarrow{r}(t)=\{x(t),y(t)\}$  on the interval $(t_0,t_1)$:   \begin{align}
\sqrt{x'^2(t)+y'^2(t)},\
\frac{1}{2}|\frac{x'(t)y''(t)-x''(t)y'(t)}{x'^2(t)+ y'^2(t)}|
.\notag\end{align}
If $\overrightarrow{r}_1(t)$\,$(t_0 <t< t_1)$ is a plane curve whose  $\phi(t)$ and $|\overrightarrow{\psi}(t)|$ are given by the above two expressions respectively, then $\overrightarrow{r}_1(t)$ and $\overrightarrow{r}(t)$ coincide on $(t_0,t_1)$ up to position in plane.

\end{theorem}
\begin{proof} Let  $\overrightarrow{r}_1  =\{x_1(t) ,y_1(t) \}$, $t_0 <t< t_1$. Then
 \begin{align}
 \sqrt{x_1'^2 +y_1'^2 }=\sqrt{x'^2 +y'^2 },\notag
\end{align}
\begin{align}
\frac{1}{2}|\frac{x'_1 y''_1 -x''_1 y'_1 }{x_1'^2  + y_1'^2 }|=
\frac{1}{2}|\frac{x'y'' -x''y' }{x'^2 + y'^2 }|
.\notag\end{align}
These imply that the curvatures of $\overrightarrow{r}_1$, $\overrightarrow{r}$ are equal. This completes the proof.

\end{proof}

The above theorem shows  that  plane curves can be classified by their local first derivative and rotational speed function.

\subsection{Plane lines } Now we will briefly discuss the  behaviour of moving points on plane lines  with rotating frame (\ref{xxsecpla2}) and  (\ref{xxsecpla12}). Give the following straight line:
\begin{equation}\label{xxsecpla22}  \overrightarrow{r} (t)=\{x_0+a t,y_0+b t\}, \ \ \ -\infty <t< +\infty.\end{equation}

\textbf{Note:} We might as well assume (\ref{xxsecpla22}) does not pass through $(0,0)$.

In view of Proposition \ref{xxsec2.1}, the following is now immediate:
\begin{corollary}\label{xxsec2.8} Let $P$ be a moving point on (\ref{xxsecpla22}), and let its parameter be $t\in (-\infty, +\infty)$.
\begin{enumerate}
  \item[$\mathrm{(i)}$]   For $\mathcal{D}=|\overrightarrow{OP}|$, we have
    \begin{align}
\frac{ \mathrm{d}\mathcal{D}}{\mathrm{d}t}=
 & \frac{a(x_0+a t)+b(y_0+b t)}{\sqrt{(x_0+a t)^2+(y_0+b t)^2} }.\notag
\end{align}
\item[$\mathrm{(ii)}$] The rotational speed of $\overrightarrow{OP}$ with respect to $t$ is $ \frac{|b(x_0+a t)-a(y_0+b t)|}{(x_0+a t)^2+(y_0+b t)^2}$.
\end{enumerate}
\end{corollary}
As a consequence of  Corollary \ref{xxsec2.8}(i), we have
 \begin{align}
\frac{ \mathrm{d}^2\mathcal{D}}{\mathrm{d}t^2}
  =& -[(x_0+a t)^2+(y_0+b t)^2  ]^{-\frac{3}{2}}[a(x_0+a t)+b(y_0+b t)]^2\notag\\
  &+[(x_0+a t)^2+(y_0+b t)^2  ]^{-\frac{1}{2}}(a^2 + b^2  ) ,\notag
\end{align}
where $-\infty <t< +\infty$.

Next let us consider the following local rotating frame:
\begin{equation}\label{xxsecpla23}
 \{P_0(x_0,y_0); \overrightarrow{e}_1, \overrightarrow{e}_2 \},
\end{equation}
where $\overrightarrow{e}_1= \frac{\{a,b\}}{\sqrt{a^2+b^2}}$, $\overrightarrow{e}_2= \frac{\{-b,a\}}{\sqrt{a^2+b^2}}$. Let the coordinates of  (\ref{xxsecpla23}) be denoted by $\xi_{P_0}$ and  $\eta_{P_0}$. Then the equation of (\ref{xxsecpla22}) near point $P_0$ under (\ref{xxsecpla23}) is
\begin{equation}
  \xi_{P_0}( t)= \sqrt{a^2  t^2+b^2   t^2}, \  \eta_{P_0}( t)=0,\notag
\end{equation}
where $t>0$. Let $Q$ be a point on (\ref{xxsecpla22}) near  $P_0$ with parameter $t$\,($>0$). Then
\begin{enumerate}
  \item[$\mathrm{i)}$] for $\mathcal{D}_{P_0}=|\overrightarrow{P_0Q}|$\,($=\xi_{P_0}(t)$), we have $\frac{ \mathrm{d}^2\mathcal{D}_{P_0} }{\mathrm{d} t^2}
= 0$;
\item[$\mathrm{ii)}$] the rotational speed of $\overrightarrow{P_0Q}$ with respect to $  t$ is $0$.
\end{enumerate}
It is not difficult to see from the above conclusion i) and ii) that  the local second-order derivative and rotational speed at $P_0$\,($t\rightarrow 0^+$) are both zero.

\textbf{Remark.} If the frame involved is (\ref{xxsecpla2}) (or (\ref{xxsecpla9})),  then any moving point on (\ref{xxsecpla22})  may have not only a non-zero second-order derivative but also a non-zero rotational speed.
Therefore, if the origin of the rotating frame is considered as the power source, $P_0$ is undoubtedly the best choice.

\textbf{Finally, we'll provide a brief summary to conclude this section:} We used rotating frames to discuss the  behaviour of moving points on plane curves. The frames used include (\ref{xxsecpla2}), (\ref{xxsecpla9}), and (\ref{xxsecpla12}). The results all indicate that the behaviour of an arbitrary moving point on a plane curve can be seen as the composition of linear motion and rotation. Furthermore, frame (\ref{xxsecpla2}), (\ref{xxsecpla9}), and (\ref{xxsecpla12}) actually correspond to three different ways to form a plane curve. The common feature among them is that
they all possess the property of binary opposition (linear motion and rotation).

\section{The behaviour of moving points on space curves}
\subsection{Space curves with general rotating frames }
Let a simple curve of class $C^2$ in $\mathbb{R}^3$ be given by
\begin{equation}\label{xxsecspa1}
  \overrightarrow{r} (t)=\{x(t),y(t),z(t)\}, \ \ \ t_0 < t<  t_1.
\end{equation}
 Set up a frame at the origin $O$ of $\mathbb{R}^3$ as follows:
\begin{equation}\label{xxsecspa2}
 \{O;\overrightarrow{e}_1(t), \overrightarrow{e}_2(t), \overrightarrow{e}_3(t) \}, \ \ \ t_0 < t<  t_1,
\end{equation}
where $\overrightarrow{e}_1(t)= \frac{\{x(t),y(t),z(t)\}}{\sqrt{x^2(t)+y^2(t)+z^2(t)}}$, $\overrightarrow{e}_i(t)$\,($i=1,2,3$) form a right-handed Cartesian frame. Frame (\ref{xxsecspa2}) is a general rotating frame of (\ref{xxsecspa1}). Let the coordinates be $\xi, \eta, \zeta$. Then the equation of (\ref{xxsecspa1}) under (\ref{xxsecspa2}) is
\begin{equation}\label{xxsecspa3}
  \xi(t)= \sqrt{x^2(t)+y^2(t)+z^2(t) },  \ \  \eta(t)=0, \ \  \zeta(t)=0.
\end{equation}

\textbf{Note:}  In this subsection, by default, $t_0 <t< t_1$ unless otherwise stated, and  (\ref{xxsecspa1}) does not pass through the origin ($O(0,0,0)$) and coordinate axes ($x$, $y$, and $z$-axes).

Recall that any moving point $P$ on (\ref{xxsecpla1}) has linear motion and rotation simultaneously.   Thanks to the establishment of  (\ref{xxsecspa2}), we can say the situation is completely similar for (\ref{xxsecspa1}), that is, the behaviour of any amoving point $P$ on (\ref{xxsecspa1}) can also be seen as the composition of linear motion and rotation. As shown in the proposition below.
\begin{proposition}\label{xxsec3.1}  Let $P$ be a moving point on (\ref{xxsecspa1}), and let its parameter be $t$.
\begin{enumerate}
  \item[$\mathrm{(i)}$]   For $\mathcal{D}=|\overrightarrow{OP}|$, we have
   \begin{align}
\frac{ \mathrm{d}\mathcal{D}}{\mathrm{d}t} =& \frac{x(t)x'(t)+y(t)y'(t)+z(t)z'(t)}{\sqrt{x^2(t)+y^2(t)+z^2(t)}} .\notag
\end{align}
\item[$\mathrm{(ii)}$]  The rotational speed of $\overrightarrow{OP}_{A}=\{x(t),y(t),0\}$ with respect to $t$ is $\frac{|x(t)y'(t)-x'(t)y(t)|}{x^2(t)+y^2(t)}$.
    \item[$\mathrm{(iii)}$] The rotational speed of $\overrightarrow{OP}_{B}=\{x(t),0,z(t)\}$ with respect to $t$ is $\frac{|x(t)z'(t)-x'(t)z(t)|}{x^2(t)+z^2(t)}$.
        \item[$\mathrm{(iv)}$] The rotational speed of  $\overrightarrow{OP}_{C}=\{0,y(t),z(t)\}$ with respect to $t$ is $\frac{|y(t)z'(t)-y'(t)z(t)|}{y^2(t)+z^2(t)}$.
\end{enumerate}
\end{proposition}
\begin{proof} (i) It suffices to know that $\mathcal{D}=\sqrt{x^2(t)+y^2(t)+z^2(t) }$\,(equals to $\xi(t)$ in (\ref{xxsecspa3})).

(ii) Clearly, we have$$\overrightarrow{e}_{\overrightarrow{OP}_{A}}= \frac{\{x(t),y(t),0\}}{\sqrt{x^2(t)+y^2(t)}}.$$
Differentiating the above vector function yields
\begin{align}\frac{\mathrm{d}\overrightarrow{e}_{\overrightarrow{OP}_{A}}}{\mathrm{d}t}
 =&-\frac{1}{2}(x^2 +y^2 )^{-\frac{3}{2}}\times 2(x x' +y y' )\{x ,y ,0\}\notag\\&+(x^2 +y^2 )^{-\frac{1}{2}}\{x' , y' ,0\}\notag\\
 =&\frac{(x y' -x'y ) \{ -y ,  x ,0 \}}{(x^2 +y^2 )^{\frac{3}{2}}}\notag.
\end{align}
The specific calculation of the above formula can refer to that of (\ref{xxsecpla4}). Thus,
\begin{align}|\frac{\mathrm{d}\overrightarrow{e}_{\overrightarrow{OP}_{A}}(t)}{\mathrm{d}t}|= \frac{|x y'-x'y|}{x^2+y^2}.\notag\end{align}

(iii)-(iv) The proof is similar to that of (ii).
\end{proof}

 \textbf{Note:} From now on, the three projective vectors of a vector, say $\overrightarrow{OP}$,  onto  $xOy$, $xOz$, and $yOz$ plane are denoted by $\overrightarrow{OP}_{A}$, $\overrightarrow{OP}_{B}$, and $\overrightarrow{OP}_{C}$, respectively. Furthermore, we always assume by default that these projective vectors of coordinate planes are non-zero.

The experience of plane curves tells us that the trajectory  of a moving point on a plane curve can be completely determined by its linear motion and rotation if  given the initial conditions. This suggests the following result.
\begin{theorem}\label{xxsec3.2} Let $O$ be the origin of $\mathbb{R}^3$.   Suppose that $P$ is a moving point in $\mathbb{R}^3$ satisfying the following conditions:
\begin{enumerate}
  \item[$\mathrm{1)}$]   for $\mathcal{D}=|\overrightarrow{OP}|$, we have
  \begin{align}
&\frac{ \mathrm{d}\mathcal{D}}{\mathrm{d}t} = \frac{x(t)x'(t)+y(t)y'(t)+z(t)z'(t)}{\sqrt{x^2(t)+y^2(t)+z^2(t)}},  \notag\\ &\mathcal{D}|_{t=t_0}=\sqrt{x^2(t_0)+y^2(t_0)+z^2(t_0)}; \notag
\end{align}
\item[$\mathrm{2)}$] $\frac{\mathrm{d}\overrightarrow{e}_{\overrightarrow{OP}_{A}}}{\mathrm{d}t}
 =(x(t)y'(t)-x'(t)y(t))\frac{\{ -y(t),  x(t),0 \}}{ (x^2(t)+y^2(t))^{\frac{3}{2}}}$, $\overrightarrow{e}_{\overrightarrow{OP}_{A}}|_{t=t_0}= \frac{\{x(t_0),y(t_0),0\}}{\sqrt{x^2(t_0)+y^2(t_0)}}$;
 \item[$\mathrm{3)}$] $\frac{\mathrm{d}\overrightarrow{e}_{\overrightarrow{OP}_{B}}}{\mathrm{d}t}
 =(x(t)z'(t)-x'(t)z(t))\frac{\{ -z(t),  0,x(t) \}}{ (x^2(t)+z^2(t))^{\frac{3}{2}}}$, $\overrightarrow{e}_{\overrightarrow{OP}_{B}}|_{t=t_0}= \frac{\{x(t_0),0, z(t_0)\}}{\sqrt{x^2(t_0)+z^2(t_0)}}$;
  \item[$\mathrm{4)}$] $\frac{\mathrm{d}\overrightarrow{e}_{\overrightarrow{OP}_{C}}}{\mathrm{d}t}
 =(y(t)z'(t)-y'(t)z(t))\frac{\{ 0, -z(t),y(t) \}}{ (y^2(t)+z^2(t))^{\frac{3}{2}}}$, $\overrightarrow{e}_{\overrightarrow{OP}_{C}}|_{t=t_0}= \frac{\{0,y(t_0), z(t_0)\}}{\sqrt{y^2(t_0)+z^2(t_0)}}$.
\end{enumerate}
Here  $t_0 \leq t\leq t_1$. Then the trajectory  of $P$ is   $$\overrightarrow{r} (t)=\{x(t),y(t),z(t)\}, \ \ \ t_0 \leq t\leq t_1.$$
\end{theorem}
\begin{proof} Let $\overrightarrow{OP}(t)=\{f_1(t),f_2(t),f_3(t)\}$, $t_0 \leq t\leq t_1$. From condition $2)$, we know that  $f_1=\lambda_1(t)x(t)$, $f_2=\lambda_1(t)y(t)$, $\lambda_1(t)>0 $; from $3)$, $f_1=\lambda_2(t)x(t)$, $f_3=\lambda_2(t)z(t)$, $\lambda_2(t)>0 $; and from
 $4)$, $f_2=\lambda_3(t)y(t)$, $f_3=\lambda_3(t)z(t)$, $\lambda_3(t)>0 $. These imply that  $\lambda_1(t)=\lambda_2(t)=\lambda_3(t)$, which can be written as $\lambda(t)$, $t_0 \leq t\leq  t_1$. But it follows from $1)$ that
 \begin{align}
{f_1}^2(t)+{f_2}^2(t)+{f_3}^2(t) = x^2(t)+y^2(t)+z^2(t)  , \ \ \ t_0 \leq t\leq  t_1. \label{xxsecspa4}
\end{align}
Taking  $f_1=\lambda(t)x(t)$, $f_2=\lambda(t)y(t)$, $f_3=\lambda(t)y(t)$ into (\ref{xxsecspa4}) we deduce that $\lambda(t)=1$.
\end{proof}

From Proposition \ref{xxsec3.1}(i), we obtain
 \begin{align}
\frac{ \mathrm{d}^2\mathcal{D}}{\mathrm{d}t^2}
  =& -(x^2(t)+y^2(t)+z^2(t) )^{-\frac{3}{2}}(x(t)x'(t)+y(t)y'(t)+z(t)z'(t))^2\notag\\
  &+(x^2(t)+y^2(t)+z^2(t) )^{-\frac{1}{2}}\notag\\&\ \ \ \times (x'^2(t)+ x(t)x''(t)+y'^2(t)+y(t)y''(t)+z'^2(t)+z(t)z''(t)) .\notag
\end{align}
Now we can carry over the proof of Theorem \ref{xxsec3.2}, virtually word for word, to get
\begin{theorem}\label{xxsec3.3}  Let $O$ be the origin of $\mathbb{R}^3$.  Suppose that $P$ is a moving point in $\mathbb{R}^3$ satisfying the following conditions:
\begin{enumerate}
  \item[$\mathrm{1)}$]   for $\mathcal{D}=|\overrightarrow{OP}|$, we have
  \begin{align}
\frac{ \mathrm{d}^2\mathcal{D}}{\mathrm{d}t^2}
   =& -(x^2(t)+y^2(t)+z^2(t) )^{-\frac{3}{2}}(x(t)x'(t)+y(t)y'(t)+z(t)z'(t))^2\notag\\
  &+(x^2(t)+y^2(t)+z^2(t) )^{-\frac{1}{2}}\notag\\&\ \ \ \times(x'^2(t)+ x(t)x''(t)+y'^2(t)+y(t)y''(t)+z'^2(t)+z(t)z''(t)), \notag\\
\frac{ \mathrm{d}\mathcal{D}}{\mathrm{d}t}|_{t=t_0}=&\frac{x(t_0)x'(t_0)+y(t_0)y'(t_0)+z(t_0)z'(t_0)}{ \sqrt{x^2(t_0)+y^2(t_0)+z^2(t_0)} }, \notag\\ \mathcal{D}|_{t=t_0}=&\sqrt{x^2(t_0)+y^2(t_0)+z^2(t_0)}; \notag
\end{align}
\item[$\mathrm{2)}$] $\frac{\mathrm{d}\overrightarrow{e}_{\overrightarrow{OP}_{A}}}{\mathrm{d}t}
 =(x(t)y'(t)-x'(t)y(t))\frac{\{ -y(t),  x(t),0 \}}{ (x^2(t)+y^2(t))^{\frac{3}{2}}}$, $\overrightarrow{e}_{\overrightarrow{OP}_{A}}|_{t=t_0}= \frac{\{x(t_0),y(t_0),0\}}{\sqrt{x^2(t_0)+y^2(t_0)}}$;
 \item[$\mathrm{3)}$] $\frac{\mathrm{d}\overrightarrow{e}_{\overrightarrow{OP}_{B}}}{\mathrm{d}t}
 =(x(t)z'(t)-x'(t)z(t))\frac{\{ -z(t),  0,x(t) \}}{ (x^2(t)+z^2(t))^{\frac{3}{2}}}$, $\overrightarrow{e}_{\overrightarrow{OP}_{B}}|_{t=t_0}= \frac{\{x(t_0),0, z(t_0)\}}{\sqrt{x^2(t_0)+z^2(t_0)}}$;
  \item[$\mathrm{4)}$] $\frac{\mathrm{d}\overrightarrow{e}_{\overrightarrow{OP}_{C}}}{\mathrm{d}t}
 =(y(t)z'(t)-y'(t)z(t))\frac{\{ 0, -z(t),y(t) \}}{ (y^2(t)+z^2(t))^{\frac{3}{2}}}$, $\overrightarrow{e}_{\overrightarrow{OP}_{C}}|_{t=t_0}= \frac{\{0,y(t_0), z(t_0)\}}{\sqrt{y^2(t_0)+z^2(t_0)}}$.
\end{enumerate}
Here $t_0 \leq t\leq t_1$. Then the trajectory of $P$ is   $$\overrightarrow{r} (t)=\{x(t),y(t),z(t)\}, \ \ \ t_0 \leq t\leq t_1.$$
\end{theorem}

\subsection{Space curves with local rotating frames}
Let $P(x(t),y(t),z(t))$ be a point on (\ref{xxsecspa1}). Set up a frame at $P$ as follows:
\begin{equation}\label{xxsecspa5}
 \{P(t);\overrightarrow{e}_1(\Delta t), \overrightarrow{e}_2(\Delta t), \overrightarrow{e}_3(\Delta t) \}, \ \ \ t_0 < t< t_1,
\end{equation}
where $\overrightarrow{e}_1(\Delta t)= \frac{\{x(t+\Delta t)-x(t),y(t+\Delta t)-y(t),z(t+\Delta t)-z(t)\}}{\sqrt{(x(t+\Delta t)-x(t))^2+(y(t+\Delta t)-y(t))^2+(z(t+\Delta t)-z(t))^2}}$ and $\overrightarrow{e}_i(\Delta t)$\,($i=1,2,3$)  form a right-handed Cartesian frame, and where $\Delta t>0$. Frame (\ref{xxsecspa5}) is a local rotating frame of (\ref{xxsecspa1}). Let the coordinates of (\ref{xxsecspa5}) be $\xi_P, \eta_P, \zeta_P$. Then the equation of (\ref{xxsecspa1}) near point $P$ under (\ref{xxsecspa5}) can be written as
\begin{equation}\label{xxsecspa6}
  \xi_P(\Delta t)= \sqrt{f_1^2(\Delta t)+f_2^2(\Delta t)+f_3^2(\Delta t) }, \  \eta_P(\Delta t)=0, \  \zeta_P(\Delta t)=0,
\end{equation}
where $f_1(\Delta t)=x(t+\Delta t)-x(t)$, $f_2(\Delta t)=y(t+\Delta t)-y(t)$, $f_3(\Delta t)=z(t+\Delta t)-z(t)$, and where $\Delta t>0$.

\textbf{Note:}
 In this subsection, we assume by default that $t_0 <t< t_1$, and $\Delta t>0$  is small enough;  we also assume that (\ref{xxsecspa1}) is of class $C^3$, and $\overrightarrow{r}'(t)\wedge\overrightarrow{r}''(t)\cdot\overrightarrow{r}'''(t)\neq 0$, $t_0 <t< t_1$. Here $\wedge$ and $\cdot$ are cross product and dot product of vectors, respectively, and will  be adopted in the following.

Let us first discuss the linear motion of the point $Q$ on (\ref{xxsecspa1}) near point $P$. The relevant result is as follows.
\begin{proposition}\label{xxsec3.4}  Let $Q(\Delta t)$ be a point on (\ref{xxsecspa1}) near  $P$, and let $\mathcal{D}_P=|\overrightarrow{PQ}|$. Then we have
 \begin{align} \frac{ \mathrm{d}\mathcal{D}_P }{\mathrm{d} \Delta t}
=&
\frac{f_1(\Delta t)x'(t+\Delta t)  +f_2(\Delta t)   y'(t+\Delta t) +f_3(\Delta t)   z'(t+\Delta t) }{\sqrt{f_1^2(\Delta t) +f_2^2(\Delta t) +f_3^2(\Delta t)}} ,\ \ \ \Delta t>0;\notag\end{align}
 \begin{align} \mathcal{D}_{P+}'(0)
 =\sqrt{x'^2(t)+y'^2(t)+z'^2(t)}. \notag\end{align}
\end{proposition}
\begin{proof} The case where  $\Delta t>0$ is immediately from $$\mathcal{D}_P= \sqrt{f_1^2(\Delta t)+f_2^2(\Delta t)+f_3^2(\Delta t) },$$ which equals to $\xi_P(\Delta t)$ in (\ref{xxsecspa6}).

Next let us consider $\Delta t\rightarrow 0^+$. Using Taylor's formula, we obtain
\begin{align}\mathcal{D}_P&=
\sqrt{(x'(t)\Delta t+\frac{1}{2}x''(\varsigma_1)\Delta t^2)^2+(y'(t)\Delta t +\frac{1}{2}y''(\varsigma_2)\Delta t^2)^2+(z'(t) \Delta t +\frac{1}{2}z''(\varsigma_3)\Delta t^2)^2}\notag\\
&=
\Delta t\sqrt{(x'(t)+\frac{1}{2}x''(\varsigma_1)\Delta t)^2+(y'(t) +\frac{1}{2}y''(\varsigma_2)\Delta t)^2+(z'(t) +\frac{1}{2}z''(\varsigma_3)\Delta t)^2}\notag, \end{align}
where $\varsigma_1, \varsigma_2, \varsigma_3$ are between $t$ and $t+\Delta t$, and where $\Delta t>0$. Thus,
\begin{align}\mathcal{D}_{P+}'(0)
 =\lim_{\Delta t\rightarrow 0^+}\frac{\mathcal{D}_P(\Delta t)-\mathcal{D}_P(0)}{\Delta t}=\sqrt{x'^2(t)+y'^2(t)+z'^2(t)},\label{xxsecspa7}
\end{align}
where $\mathcal{D}_P(0)=0$.

\end{proof}

\textbf{Note:} We call (\ref{xxsecspa7}) the local first derivative of  (\ref{xxsecspa1}) and denote it by $\phi(t)$\,(following the same convention as planar case in $\S$ 2.4).

 Let  $Q$ be as above. We now turn to discussing aspect concerning the projective rotational speed of $\overrightarrow{PQ}$. For this purpose, in addition to (\ref{xxsecspa5}), we construct another frame:
\begin{equation}\label{xxsecspa8}
 \{P;\overrightarrow{r}'(t),\overrightarrow{r}''(t),\overrightarrow{r}'''(t) \}, \ \ \ t_0 < t< t_1.
\end{equation}
Then the equation of  (\ref{xxsecspa1}) can be written as
\begin{align}&\overrightarrow{r}(t+\Delta t)-\overrightarrow{r}(t)\notag\\=&g_1(\Delta t)\overrightarrow{r}'(t)+g_2(\Delta t)\overrightarrow{r}''(t)+g_3(\Delta t)\overrightarrow{r}'''(t),\label{xxsecspa9}
\end{align}
where $\Delta t>0$.

 Now we can project $\overrightarrow{PQ}$ to the three frame planes of (\ref{xxsecspa8}) and compute the corresponding rotational speed. Let us denote by $\overrightarrow{PQ}_{12}$ the projective vector of $\overrightarrow{PQ}$ onto the plane spanned by $\overrightarrow{r}'(t),\overrightarrow{r}''(t)$, $\overrightarrow{PQ}_{13}$ the  projective vector of $\overrightarrow{PQ}$ onto the plane of  $\overrightarrow{r}'(t),\overrightarrow{r}'''(t)$, and $\overrightarrow{PQ}_{23}$ the  projective vector of $\overrightarrow{PQ}$ onto  the plane of $\overrightarrow{r}''(t),\overrightarrow{r}'''(t)$. Assuming that  $\overrightarrow{PQ}_{12}$ , $\overrightarrow{PQ}_{13}$ , and $\overrightarrow{PQ}_{23}$  are non-zero vectors, then we have
\begin{proposition}\label{xxsec3.5}  Let  $Q$ be a moving point on (\ref{xxsecspa9}) near $P$, and let its parameter be $\Delta t$.
\begin{enumerate}
 \item[$\mathrm{1)}$] The rotational speed of $\overrightarrow{PQ}_{12}$ with respect to $\Delta t$ is
 \begin{align}&[(g_1(\Delta t)\overrightarrow{r}'(t)+g_2(\Delta t)\overrightarrow{r}''(t))^2]^{-\frac{3}{2}}\notag\\
&\times|g_1^2(\Delta t)g_2'(\Delta t)\overrightarrow{r}'(t)\cdot\overrightarrow{r}''(t)\overrightarrow{r}'(t) +g_2(\Delta t)g_2'(\Delta t)g_1(\Delta t)\overrightarrow{r}''^2(t)\overrightarrow{r}'(t)\notag\\
&\ \ \  -g_1(\Delta t)g_2(\Delta t)g_1'(\Delta t)\overrightarrow{r}'(t)\cdot\overrightarrow{r}''(t)\overrightarrow{r}'(t)-g_2^2(\Delta t)g_1'(\Delta t)\overrightarrow{r}''^2(t)\overrightarrow{r}'(t)\notag\\
&\ \ \ +g_1(\Delta t)g_1'(\Delta t)g_2(\Delta t)\overrightarrow{r}'^2(t)\overrightarrow{r}''(t)+g_2^2(\Delta t)g_1'(\Delta t)\overrightarrow{r}'(t)\cdot\overrightarrow{r}''(t)\overrightarrow{r}''(t)  \notag\\
&\ \ \ -g_1^2(\Delta t)g_2'(\Delta t)\overrightarrow{r}'^2(t)\overrightarrow{r}''(t)\notag\\
&\ \ \ -g_1(\Delta t)g_2(\Delta t)g_2'(\Delta t)\overrightarrow{r}'(t)\cdot\overrightarrow{r}''(t)\overrightarrow{r}''(t)|
, \ \ \Delta t>0 ;\notag\\
 &\frac{|-
   \overrightarrow{r}'(t)\cdot\overrightarrow{r}''(t)\overrightarrow{r}'(t)
   +\overrightarrow{r}'^2(t)\overrightarrow{r}''(t)|}{2|\overrightarrow{r}'(t)|^{3}}
   ,\ \  \Delta t\rightarrow 0^+.\notag\end{align}
\item[$\mathrm{2)}$] The rotational speed of $\overrightarrow{PQ}_{13}$ with respect to $\Delta t$ is
 \begin{align}&[(g_1(\Delta t)\overrightarrow{r}'(t)+g_3(\Delta t)\overrightarrow{r}'''(t))^2]^{-\frac{3}{2}}\notag\\
&\times|g_1^2(\Delta t)g_3'(\Delta t)\overrightarrow{r}'(t)\cdot\overrightarrow{r}'''(t)\overrightarrow{r}'(t) +g_3(\Delta t)g_3'(\Delta t)g_1(\Delta t)\overrightarrow{r}'''^2(t)\overrightarrow{r}'(t)\notag\\
&\ \ \  -g_1(\Delta t)g_3(\Delta t)g_1'(\Delta t)\overrightarrow{r}'(t)\cdot\overrightarrow{r}'''(t)\overrightarrow{r}'(t)-g_3^2(\Delta t)g_1'(\Delta t)\overrightarrow{r}'''^2(t)\overrightarrow{r}'(t)\notag\\
&\ \ \ +g_1(\Delta t)g_1'(\Delta t)g_3(\Delta t)\overrightarrow{r}'^2(t)\overrightarrow{r}'''(t)+g_3^2(\Delta t)g_1'(\Delta t)\overrightarrow{r}'(t)\cdot\overrightarrow{r}'''(t)\overrightarrow{r}'''(t)  \notag\\
&\ \ \ -g_1^2(\Delta t)g_3'(\Delta t)\overrightarrow{r}'^2(t)\overrightarrow{r}'''(t) \notag\\
&\ \ \ -g_1(\Delta t)g_3(\Delta t)g_3'(\Delta t)\overrightarrow{r}'(t)\cdot\overrightarrow{r}'''(t)\overrightarrow{r}'''(t)|, \ \ \Delta t>0 ;\notag\\
 &0, \ \ \Delta t\rightarrow 0^+ .\notag\end{align}
\item[$\mathrm{3)}$] The rotational speed of $\overrightarrow{PQ}_{23}$ with respect to $\Delta t$ is
  \begin{align}&[(g_2(\Delta t)\overrightarrow{r}''(t)+g_3(\Delta t)\overrightarrow{r}'''(t))^2]^{-\frac{3}{2}}\notag\\
&\times|g_2^2(\Delta t)g_3'(\Delta t)\overrightarrow{r}''(t)\cdot\overrightarrow{r}'''(t)\overrightarrow{r}''(t) +g_3(\Delta t)g_3'(\Delta t)g_2(\Delta t)\overrightarrow{r}'''^2(t)\overrightarrow{r}''(t)\notag\\
&\ \ \  -g_2(\Delta t)g_3(\Delta t)g_2'(\Delta t)\overrightarrow{r}''(t)\cdot\overrightarrow{r}'''(t)\overrightarrow{r}''(t)-g_3^2(\Delta t)g_2'(\Delta t)\overrightarrow{r}'''^2(t)\overrightarrow{r}''(t)\notag\\
&\ \ \ +g_2(\Delta t)g_2'(\Delta t)g_3(\Delta t)\overrightarrow{r}''^2(t)\overrightarrow{r}'''(t)+g_3^2(\Delta t)g_2'(\Delta t)\overrightarrow{r}''(t)\cdot\overrightarrow{r}'''(t)\overrightarrow{r}'''(t)  \notag\\
&\ \ \ -g_2^2(\Delta t)g_3'(\Delta t)\overrightarrow{r}''^2(t)\overrightarrow{r}'''(t) \notag\\
&\ \ \ -g_2(\Delta t)g_3(\Delta t)g_3'(\Delta t)\overrightarrow{r}''(t)\cdot\overrightarrow{r}'''(t)\overrightarrow{r}'''(t)|, \ \ \Delta t>0;\notag\\
 &
 \frac{|-
   \overrightarrow{r}''(t)\cdot\overrightarrow{r}'''(t)\overrightarrow{r}''(t)
   +\overrightarrow{r}''^2(t)\overrightarrow{r}'''(t)|}{3|\overrightarrow{r}''(t)|^{3}},\ \  \Delta t\rightarrow 0^+ .\notag\end{align}
\end{enumerate}
\end{proposition}
\begin{proof}
1) In view of (\ref{xxsecspa9}), we can write $$\overrightarrow{e}_{\overrightarrow{PQ}_{12}} = \frac{g_1(\Delta t)\overrightarrow{r}'(t)+g_2(\Delta t)\overrightarrow{r}''(t)}{\sqrt{(g_1(\Delta t)\overrightarrow{r}'(t)+g_2(\Delta t)\overrightarrow{r}''(t))^2 }},$$
where $\Delta t>0$. It is sufficient for us to  compute
\begin{align}&\frac{\mathrm{d}\overrightarrow{e}_{\overrightarrow{PQ}_{12}} }{\mathrm{d}\Delta t}\notag\\
 =&-\frac{1}{2}[(g_1(\Delta t)\overrightarrow{r}'(t)+g_2(\Delta t)\overrightarrow{r}''(t))^2]^{-\frac{3}{2}}\notag\\
&\times2(g_1(\Delta t)\overrightarrow{r}'(t)+g_2(\Delta t)\overrightarrow{r}''(t) )\cdot
(g_1'(\Delta t)\overrightarrow{r}'(t)+g_2'(\Delta t)\overrightarrow{r}''(t) )\notag\\
&\times(g_1(\Delta t)\overrightarrow{r}'(t)+g_2(\Delta t)\overrightarrow{r}''(t) )\notag\\&+[(g_1(\Delta t)\overrightarrow{r}'(t)+g_2(\Delta t)\overrightarrow{r}''(t))^2]^{-\frac{1}{2}}(g_1'(\Delta t)\overrightarrow{r}'(t)+g_2'(\Delta t)\overrightarrow{r}''(t))\notag\\
  =&-[(g_1(\Delta t)\overrightarrow{r}'(t)+g_2(\Delta t)\overrightarrow{r}''(t))^2]^{-\frac{3}{2}}\notag\\
&\times[(g_1(\Delta t)\overrightarrow{r}'(t)+g_2(\Delta t)\overrightarrow{r}''(t) )\cdot
(g_1'(\Delta t)\overrightarrow{r}'(t)+g_2'(\Delta t)\overrightarrow{r}''(t) )\notag\\
&\times(g_1(\Delta t)\overrightarrow{r}'(t)+g_2(\Delta t)\overrightarrow{r}''(t) )\notag\\
&-(g_1(\Delta t)\overrightarrow{r}'(t)+g_2(\Delta t)\overrightarrow{r}''(t))^2(g_1'(\Delta t)\overrightarrow{r}'(t)+g_2'(\Delta t)\overrightarrow{r}''(t))]\notag\\
  =&-[(g_1(\Delta t)\overrightarrow{r}'(t)+g_2(\Delta t)\overrightarrow{r}''(t))^2]^{-\frac{3}{2}}\notag\\
&\times[(g_1(\Delta t)g_1'(\Delta t)\overrightarrow{r}'^2(t)+g_1(\Delta t)g_2'(\Delta t)\overrightarrow{r}'(t)\cdot\overrightarrow{r}''(t) \notag\\&\ \ \  +g_2(\Delta t)g_1'(\Delta t)\overrightarrow{r}''(t)\cdot\overrightarrow{r}'(t) +g_2(\Delta t)g_2'(\Delta t)\overrightarrow{r}''^2(t))\notag\\
&\times(g_1(\Delta t)\overrightarrow{r}'(t)+g_2(\Delta t)\overrightarrow{r}''(t)) \notag\\
&-(g_1^2(\Delta t)\overrightarrow{r}'^2(t)+2g_1(\Delta t)g_2(\Delta t)\overrightarrow{r}'(t)\cdot\overrightarrow{r}''(t)+g_2^2(\Delta t)\overrightarrow{r}''^2(t))\notag\\
&\times(g_1'(\Delta t)\overrightarrow{r}'(t)+g_2'(\Delta t)\overrightarrow{r}''(t))]\notag\\
  =&-[(g_1(\Delta t)\overrightarrow{r}'(t)+g_2(\Delta t)\overrightarrow{r}''(t))^2]^{-\frac{3}{2}}\notag\\
&\times(g_1^2(\Delta t)g_1'(\Delta t)\overrightarrow{r}'^2(t)\overrightarrow{r}'(t)+g_1^2(\Delta t)g_2'(\Delta t)\overrightarrow{r}'(t)\cdot\overrightarrow{r}''(t)\overrightarrow{r}'(t) \notag\\&\ \ \  +g_2(\Delta t)g_1'(\Delta t)g_1(\Delta t)\overrightarrow{r}''(t)\cdot\overrightarrow{r}'(t) \overrightarrow{r}'(t)+g_2(\Delta t)g_2'(\Delta t)g_1(\Delta t)\overrightarrow{r}''^2(t)\overrightarrow{r}'(t)\notag\\&\ \ \ +
g_1(\Delta t)g_1'(\Delta t)g_2(\Delta t)\overrightarrow{r}'^2(t)\overrightarrow{r}''(t)+g_1(\Delta t)g_2'(\Delta t)g_2(\Delta t)\overrightarrow{r}'(t)\cdot\overrightarrow{r}''(t)\overrightarrow{r}''(t) \notag\\&\ \ \  +g_2^2(\Delta t)g_1'(\Delta t)\overrightarrow{r}''(t)\cdot\overrightarrow{r}'(t)\overrightarrow{r}''(t) +g_2^2(\Delta t)g_2'(\Delta t)\overrightarrow{r}''^2(t)\overrightarrow{r}''(t) \notag\\
&-g_1^2(\Delta t)g_1'(\Delta t)\overrightarrow{r}'^2(t)\overrightarrow{r}'(t)-2g_1(\Delta t)g_2(\Delta t)g_1'(\Delta t)\overrightarrow{r}'(t)\cdot\overrightarrow{r}''(t)\overrightarrow{r}'(t)\notag\\
&-g_2^2(\Delta t)g_1'(\Delta t)\overrightarrow{r}''^2(t)\overrightarrow{r}'(t)-g_1^2(\Delta t)g_2'(\Delta t)\overrightarrow{r}'^2(t)\overrightarrow{r}''(t)\notag\\
&-2g_1(\Delta t)g_2(\Delta t)g_2'(\Delta t)\overrightarrow{r}'(t)\cdot\overrightarrow{r}''(t)\overrightarrow{r}''(t)-g_2^2(\Delta t)g_2'(\Delta t)\overrightarrow{r}''^2(t)\overrightarrow{r}''(t))\notag\\
  =&-[(g_1(\Delta t)\overrightarrow{r}'(t)+g_2(\Delta t)\overrightarrow{r}''(t))^2]^{-\frac{3}{2}}\notag\\
&\times(g_1^2(\Delta t)g_2'(\Delta t)\overrightarrow{r}'(t)\cdot\overrightarrow{r}''(t)\overrightarrow{r}'(t) \notag\\&\ \ \  +g_2(\Delta t)g_2'(\Delta t)g_1(\Delta t)\overrightarrow{r}''^2(t)\overrightarrow{r}'(t)\notag\\&\ \ \ +g_1(\Delta t)g_1'(\Delta t)g_2(\Delta t)\overrightarrow{r}'^2(t)\overrightarrow{r}''(t)\notag\\&\ \ \  +g_2^2(\Delta t)g_1'(\Delta t)\overrightarrow{r}'(t)\cdot\overrightarrow{r}''(t)\overrightarrow{r}''(t)  \notag\\
&-g_1(\Delta t)g_2(\Delta t)g_1'(\Delta t)\overrightarrow{r}'(t)\cdot\overrightarrow{r}''(t)\overrightarrow{r}'(t)-g_2^2(\Delta t)g_1'(\Delta t)\overrightarrow{r}''^2(t)\overrightarrow{r}'(t)\notag\\
&-g_1^2(\Delta t)g_2'(\Delta t)\overrightarrow{r}'^2(t)\overrightarrow{r}''(t)-g_1(\Delta t)g_2(\Delta t)g_2'(\Delta t)\overrightarrow{r}'(t)\cdot\overrightarrow{r}''(t)\overrightarrow{r}''(t))\notag\\
  =&-[(g_1(\Delta t)\overrightarrow{r}'(t)+g_2(\Delta t)\overrightarrow{r}''(t))^2]^{-\frac{3}{2}}\notag\\
&\times(g_1^2(\Delta t)g_2'(\Delta t)\overrightarrow{r}'(t)\cdot\overrightarrow{r}''(t)\overrightarrow{r}'(t) +g_2(\Delta t)g_2'(\Delta t)g_1(\Delta t)\overrightarrow{r}''^2(t)\overrightarrow{r}'(t)\notag\\
&\ \ \  -g_1(\Delta t)g_2(\Delta t)g_1'(\Delta t)\overrightarrow{r}'(t)\cdot\overrightarrow{r}''(t)\overrightarrow{r}'(t)-g_2^2(\Delta t)g_1'(\Delta t)\overrightarrow{r}''^2(t)\overrightarrow{r}'(t)\notag\\
&\ \ \ +g_1(\Delta t)g_1'(\Delta t)g_2(\Delta t)\overrightarrow{r}'^2(t)\overrightarrow{r}''(t)+g_2^2(\Delta t)g_1'(\Delta t)\overrightarrow{r}'(t)\cdot\overrightarrow{r}''(t)\overrightarrow{r}''(t)  \notag\\
&\ \ \ -g_1^2(\Delta t)g_2'(\Delta t)\overrightarrow{r}'^2(t)\overrightarrow{r}''(t)-g_1(\Delta t)g_2(\Delta t)g_2'(\Delta t)\overrightarrow{r}'(t)\cdot\overrightarrow{r}''(t)\overrightarrow{r}''(t))
,\label{xxsecspa10}\end{align}
where $\Delta t>0$.

Now we proceed to prove the case where $\Delta t\rightarrow 0^+$. To begin with, let us consider the following second-order Taylor's formula of vector function:
\begin{align}\overrightarrow{r}(t+\Delta t)-\overrightarrow{r}(t)=&\overrightarrow{r}'(t)\Delta t+\frac{1}{2}\overrightarrow{r}''(t)\Delta t^2+\frac{1}{6}(\overrightarrow{r}'''(t)+\overrightarrow{\varepsilon})\Delta t^3,\notag
\end{align}
where $\lim_{\Delta t\rightarrow 0^+} \overrightarrow{\varepsilon}=\overrightarrow{0}$. Let $\overrightarrow{\varepsilon}=\varepsilon_1\overrightarrow{r}'(t)+\varepsilon_2\overrightarrow{r}''(t)+\varepsilon_3\overrightarrow{r}'''(t)$. Then
\begin{align}&\overrightarrow{r}(t+\Delta t)-\overrightarrow{r}(t)\notag\\=&\overrightarrow{r}'(t)\Delta t+\frac{1}{2}\overrightarrow{r}''(t)\Delta t^2+\frac{1}{6}\overrightarrow{r}'''(t)\Delta t^3+\frac{1}{6}(\varepsilon_1\overrightarrow{r}'(t)+\varepsilon_2\overrightarrow{r}''(t)+\varepsilon_3\overrightarrow{r}'''(t))\Delta t^3\notag\\
=&\overrightarrow{r}'(t)\Delta t+\frac{1}{2}\overrightarrow{r}''(t)\Delta t^2+\frac{1}{6}\overrightarrow{r}'''(t)\Delta t^3+\frac{1}{6}\varepsilon_1\overrightarrow{r}'(t)\Delta t^3+\frac{1}{6}\varepsilon_2\overrightarrow{r}''(t)\Delta t^3+\frac{1}{6}\varepsilon_3\overrightarrow{r}'''(t) \Delta t^3\notag\\
=&(\Delta t+\frac{1}{6}\varepsilon_1\Delta t^3)\overrightarrow{r}'(t)+(\frac{1}{2}\Delta t^2+\frac{1}{6}\varepsilon_2\Delta t^3)\overrightarrow{r}''(t)+(\frac{1}{6}\Delta t^3+\frac{1}{6}\varepsilon_3\Delta t^3)\overrightarrow{r}'''(t).\label{xxsecspa11}
\end{align}
It follows that
\begin{align}\overrightarrow{e}_{\overrightarrow{PQ}_{12}}& = \frac{(\Delta t+\frac{1}{6}\varepsilon_1\Delta t^3)\overrightarrow{r}'(t)+(\frac{1}{2}\Delta t^2+\frac{1}{6}\varepsilon_2\Delta t^3)\overrightarrow{r}''(t)}{\sqrt{[(\Delta t+\frac{1}{6}\varepsilon_1\Delta t^3)\overrightarrow{r}'(t)+(\frac{1}{2}\Delta t^2+\frac{1}{6}\varepsilon_2\Delta t^3)\overrightarrow{r}''(t)]^2}}\notag\\
& =\frac{(1+\frac{1}{6}\varepsilon_1\Delta t^2)\overrightarrow{r}'(t)+(\frac{1}{2}\Delta t+\frac{1}{6}\varepsilon_2\Delta t^2)\overrightarrow{r}''(t)}{\sqrt{[(1+\frac{1}{6}\varepsilon_1\Delta t^2)\overrightarrow{r}'(t)+(\frac{1}{2}\Delta t +\frac{1}{6}\varepsilon_2\Delta t^2)\overrightarrow{r}''(t)]^2}},\notag\end{align}
where $\Delta t>0$. We just need to do the following calculation:
\begin{align}&\lim_{\Delta t\rightarrow 0^+}\frac{\overrightarrow{e}_{\overrightarrow{PQ}_{12}}(\Delta t)-\overrightarrow{e}_{\overrightarrow{PQ}_{12}}(0)}{\Delta t}\notag\\
  =&\lim_{\Delta t\rightarrow 0^+} \frac{1}{\Delta t}[\frac{(1+\frac{1}{6}\varepsilon_1\Delta t^2)\overrightarrow{r}'(t)+(\frac{1}{2}\Delta t+\frac{1}{6}\varepsilon_2\Delta t^2)\overrightarrow{r}''(t)}{|(1+\frac{1}{6}\varepsilon_1\Delta t^2)\overrightarrow{r}'(t)+(\frac{1}{2}\Delta t +\frac{1}{6}\varepsilon_2\Delta t^2)\overrightarrow{r}''(t)|}-\frac{\overrightarrow{r}'(t)}{|\overrightarrow{r}'(t)|}] \notag\\
  \notag\\
  =&\lim_{\Delta t\rightarrow 0^+} \frac{1}{\Delta t}[\frac{\overrightarrow{r}'(t)}{|(1+\frac{1}{6}\varepsilon_1\Delta t^2)\overrightarrow{r}'(t)+(\frac{1}{2}\Delta t +\frac{1}{6}\varepsilon_2\Delta t^2)\overrightarrow{r}''(t)|}-\frac{\overrightarrow{r}'(t)}{|\overrightarrow{r}'(t)|}]\notag\\&\ \ \ +
  \lim_{\Delta t\rightarrow 0^+}\frac{\frac{1}{6}\varepsilon_1\Delta t\overrightarrow{r}'(t)+(\frac{1}{2}+\frac{1}{6}\varepsilon_2\Delta t)\overrightarrow{r}''(t)}{|(1+\frac{1}{6}\varepsilon_1\Delta t^2)\overrightarrow{r}'(t)+(\frac{1}{2}\Delta t +\frac{1}{6}\varepsilon_2\Delta t^2)\overrightarrow{r}''(t)|}\notag\\
    =&\overrightarrow{r}'(t)\lim_{\Delta t\rightarrow 0^+} \frac{1}{\Delta t}|(1+\frac{1}{6}\varepsilon_1\Delta t^2)\overrightarrow{r}'(t)+(\frac{1}{2}\Delta t +\frac{1}{6}\varepsilon_2\Delta t^2)\overrightarrow{r}''(t)|^{-1}|\overrightarrow{r}'(t)|^{-1}\notag\\&\times
    [|\overrightarrow{r}'(t)|-|(1+\frac{1}{6}\varepsilon_1\Delta t^2)\overrightarrow{r}'(t)+(\frac{1}{2}\Delta t +\frac{1}{6}\varepsilon_2\Delta t^2)\overrightarrow{r}''(t)|]  +\frac{\overrightarrow{r}''(t)}{2|\overrightarrow{r}'(t)|}\notag\\
      =&\overrightarrow{r}'(t)\lim_{\Delta t\rightarrow 0^+} \frac{1}{\Delta t}|(1+\frac{1}{6}\varepsilon_1\Delta t^2)\overrightarrow{r}'(t)+(\frac{1}{2}\Delta t +\frac{1}{6}\varepsilon_2\Delta t^2)\overrightarrow{r}''(t)|^{-1}|\overrightarrow{r}'(t)|^{-1}\notag\\&\times
       [|\overrightarrow{r}'(t)|+|(1+\frac{1}{6}\varepsilon_1\Delta t^2)\overrightarrow{r}'(t)+(\frac{1}{2}\Delta t +\frac{1}{6}\varepsilon_2\Delta t^2)\overrightarrow{r}''(t)|]^{-1}\notag\\&\times
    [|\overrightarrow{r}'(t)|^2-|(1+\frac{1}{6}\varepsilon_1\Delta t^2)\overrightarrow{r}'(t)+(\frac{1}{2}\Delta t +\frac{1}{6}\varepsilon_2\Delta t^2)\overrightarrow{r}''(t)|^2]  +\frac{\overrightarrow{r}''(t)}{2|\overrightarrow{r}'(t)|}\notag\\
      =&\overrightarrow{r}'(t)\lim_{\Delta t\rightarrow 0^+} \frac{1}{\Delta t}|(1+\frac{1}{6}\varepsilon_1\Delta t^2)\overrightarrow{r}'(t)+(\frac{1}{2}\Delta t +\frac{1}{6}\varepsilon_2\Delta t^2)\overrightarrow{r}''(t)|^{-1}|\overrightarrow{r}'(t)|^{-1}\notag\\&\times
       [|\overrightarrow{r}'(t)|+|(1+\frac{1}{6}\varepsilon_1\Delta t^2)\overrightarrow{r}'(t)+(\frac{1}{2}\Delta t +\frac{1}{6}\varepsilon_2\Delta t^2)\overrightarrow{r}''(t)|]^{-1}\notag\\&\times
    [|\overrightarrow{r}'(t)|^2-|(1+\frac{1}{6}\varepsilon_1\Delta t^2)\overrightarrow{r}'(t)|^2-
    2(1+\frac{1}{6}\varepsilon_1\Delta t^2)(\frac{1}{2}\Delta t +\frac{1}{6}\varepsilon_2\Delta t^2)\overrightarrow{r}'(t)\cdot\overrightarrow{r}''(t)
    \notag\\&\ \ \ -|(\frac{1}{2}\Delta t +\frac{1}{6}\varepsilon_2\Delta t^2)\overrightarrow{r}''(t)|^2] +\frac{\overrightarrow{r}''(t)}{2|\overrightarrow{r}'(t)|}\notag\\
      =&\overrightarrow{r}'(t)\lim_{\Delta t\rightarrow 0^+} \frac{1}{\Delta t}|(1+\frac{1}{6}\varepsilon_1\Delta t^2)\overrightarrow{r}'(t)+(\frac{1}{2}\Delta t +\frac{1}{6}\varepsilon_2\Delta t^2)\overrightarrow{r}''(t)|^{-1}|\overrightarrow{r}'(t)|^{-1}\notag\\&\times
       [|\overrightarrow{r}'(t)|+|(1+\frac{1}{6}\varepsilon_1\Delta t^2)\overrightarrow{r}'(t)+(\frac{1}{2}\Delta t +\frac{1}{6}\varepsilon_2\Delta t^2)\overrightarrow{r}''(t)|]^{-1}\notag\\&\times
    [-(\frac{1}{3}\varepsilon_1\Delta t^2+\frac{1}{36}\varepsilon_1^2\Delta t^4)|\overrightarrow{r}'(t)|^2-
    2(1+\frac{1}{6}\varepsilon_1\Delta t^2)(\frac{1}{2}\Delta t +\frac{1}{6}\varepsilon_2\Delta t^2)\overrightarrow{r}'(t)\cdot\overrightarrow{r}''(t)
    \notag\\&\ \ \ -(\frac{1}{2}\Delta t +\frac{1}{6}\varepsilon_2\Delta t^2)^2|\overrightarrow{r}''(t)|^2] +\frac{\overrightarrow{r}''(t)}{2|\overrightarrow{r}'(t)|}\notag\\
       =&\overrightarrow{r}'(t)\lim_{\Delta t\rightarrow 0^+} |(1+\frac{1}{6}\varepsilon_1\Delta t^2)\overrightarrow{r}'(t)+(\frac{1}{2}\Delta t +\frac{1}{6}\varepsilon_2\Delta t^2)\overrightarrow{r}''(t)|^{-1}|\overrightarrow{r}'(t)|^{-1}\notag\\&\times
       [|\overrightarrow{r}'(t)|+|(1+\frac{1}{6}\varepsilon_1\Delta t^2)\overrightarrow{r}'(t)+(\frac{1}{2}\Delta t +\frac{1}{6}\varepsilon_2\Delta t^2)\overrightarrow{r}''(t)|]^{-1}\notag\\&\times
    [-(\frac{1}{3}\varepsilon_1\Delta t+\frac{1}{36}\varepsilon_1^2\Delta t^3)|\overrightarrow{r}'(t)|^2-
    2(1+\frac{1}{6}\varepsilon_1\Delta t^2)(\frac{1}{2} +\frac{1}{6}\varepsilon_2\Delta t)\overrightarrow{r}'(t)\cdot\overrightarrow{r}''(t)
    \notag\\&\ \ \ -\Delta t(\frac{1}{2} +\frac{1}{6}\varepsilon_2\Delta t)^2|\overrightarrow{r}''(t)|^2] +\frac{\overrightarrow{r}''(t)}{2|\overrightarrow{r}'(t)|}\notag\\
        =&\overrightarrow{r}'(t)|\overrightarrow{r}'(t)|^{-2}[2|\overrightarrow{r}'(t)|]^{-1}
    (-\overrightarrow{r}'(t)\cdot\overrightarrow{r}''(t)
    ) +\frac{\overrightarrow{r}''(t)}{2|\overrightarrow{r}'(t)|}\notag\\
       =&-\frac{1}{2}|\overrightarrow{r}'(t)|^{-3}
   \overrightarrow{r}'(t)\cdot\overrightarrow{r}''(t)\overrightarrow{r}'(t)
   +\frac{\overrightarrow{r}''(t)}{2|\overrightarrow{r}'(t)|}
  .\label{xxsecspa12}
\end{align}

2) and 3) The proof is analogous to that of 1). Below is  a list of simple calculations for  $\Delta t\rightarrow 0^+$.
\begin{align}\overrightarrow{e}_{\overrightarrow{PQ}_{13}}& = \frac{(\Delta t+\frac{1}{6}\varepsilon_1\Delta t^3)\overrightarrow{r}'(t)+(\frac{1}{6}\Delta t^3+\frac{1}{6}\varepsilon_3\Delta t^3)\overrightarrow{r}'''(t)}{\sqrt{[(\Delta t+\frac{1}{6}\varepsilon_1\Delta t^3)\overrightarrow{r}'(t)+(\frac{1}{6}\Delta t^3+\frac{1}{6}\varepsilon_3\Delta t^3)\overrightarrow{r}'''(t)]^2}}\notag\\
& =\frac{(1+\frac{1}{6}\varepsilon_1\Delta t^2)\overrightarrow{r}'(t)+(\frac{1}{6}\Delta t^2+\frac{1}{6}\varepsilon_3\Delta t^2)\overrightarrow{r}'''(t)}{\sqrt{[(1+\frac{1}{6}\varepsilon_1\Delta t^2)\overrightarrow{r}'(t)+(\frac{1}{6}\Delta t^2+\frac{1}{6}\varepsilon_3\Delta t^2)\overrightarrow{r}'''(t)]^2}},\notag\end{align}

\begin{align}&\lim_{\Delta t\rightarrow 0^+}\frac{\overrightarrow{e}_{\overrightarrow{PQ}_{13}}(\Delta t)-\overrightarrow{e}_{\overrightarrow{PQ}_{13}}(0)}{\Delta t}\notag\\
  =&\lim_{\Delta t\rightarrow 0^+} \frac{1}{\Delta t}[\frac{(1+\frac{1}{6}\varepsilon_1\Delta t^2)\overrightarrow{r}'(t)+(\frac{1}{6}\Delta t^2+\frac{1}{6}\varepsilon_3\Delta t^2)\overrightarrow{r}'''(t)}{|(1+\frac{1}{6}\varepsilon_1\Delta t^2)\overrightarrow{r}'(t)+(\frac{1}{6}\Delta t^2+\frac{1}{6}\varepsilon_3\Delta t^2)\overrightarrow{r}'''(t)|}-\frac{\overrightarrow{r}'(t)}{|\overrightarrow{r}'(t)|}] \notag\\
  \notag\\
  =&\lim_{\Delta t\rightarrow 0^+} \frac{1}{\Delta t}[\frac{\overrightarrow{r}'(t)}{|(1+\frac{1}{6}\varepsilon_1\Delta t^2)\overrightarrow{r}'(t)+(\frac{1}{6}\Delta t^2+\frac{1}{6}\varepsilon_3\Delta t^2)\overrightarrow{r}'''(t)|}-\frac{\overrightarrow{r}'(t)}{|\overrightarrow{r}'(t)|}]\notag\\&\ \ \ +
  \lim_{\Delta t\rightarrow 0^+}\frac{\frac{1}{6}\varepsilon_1\Delta t\overrightarrow{r}'(t)+(\frac{1}{6}\Delta t+\frac{1}{6}\varepsilon_3\Delta t)\overrightarrow{r}'''(t)}{|(1+\frac{1}{6}\varepsilon_1\Delta t^2)\overrightarrow{r}'(t)+(\frac{1}{6}\Delta t^2+\frac{1}{6}\varepsilon_3\Delta t^2)\overrightarrow{r}'''(t)|}\notag\\
    =&\overrightarrow{r}'(t)\lim_{\Delta t\rightarrow 0^+} \frac{1}{\Delta t}|(1+\frac{1}{6}\varepsilon_1\Delta t^2)\overrightarrow{r}'(t)+(\frac{1}{6}\Delta t^2+\frac{1}{6}\varepsilon_3\Delta t^2)\overrightarrow{r}'''(t)|^{-1}|\overrightarrow{r}'(t)|^{-1}\notag\\&\times
    [|\overrightarrow{r}'(t)|-|(1+\frac{1}{6}\varepsilon_1\Delta t^2)\overrightarrow{r}'(t)+(\frac{1}{6}\Delta t^2+\frac{1}{6}\varepsilon_3\Delta t^2)\overrightarrow{r}'''(t)|] \notag\\
      =&\overrightarrow{r}'(t)\lim_{\Delta t\rightarrow 0^+} \frac{1}{\Delta t}|(1+\frac{1}{6}\varepsilon_1\Delta t^2)\overrightarrow{r}'(t)+(\frac{1}{6}\Delta t^2+\frac{1}{6}\varepsilon_3\Delta t^2)\overrightarrow{r}'''(t)|^{-1}|\overrightarrow{r}'(t)|^{-1}\notag\\&\times
       [|\overrightarrow{r}'(t)|+|(1+\frac{1}{6}\varepsilon_1\Delta t^2)\overrightarrow{r}'(t)+(\frac{1}{6}\Delta t^2+\frac{1}{6}\varepsilon_3\Delta t^2)\overrightarrow{r}'''(t)|]^{-1}\notag\\&\times
    [|\overrightarrow{r}'(t)|^2-|(1+\frac{1}{6}\varepsilon_1\Delta t^2)\overrightarrow{r}'(t)+(\frac{1}{6}\Delta t^2+\frac{1}{6}\varepsilon_3\Delta t^2)\overrightarrow{r}'''(t)|^2]  \notag\\
      =&\overrightarrow{r}'(t)\lim_{\Delta t\rightarrow 0^+} \frac{1}{\Delta t}|(1+\frac{1}{6}\varepsilon_1\Delta t^2)\overrightarrow{r}'(t)+(\frac{1}{6}\Delta t^2+\frac{1}{6}\varepsilon_3\Delta t^2)\overrightarrow{r}'''(t)|^{-1}|\overrightarrow{r}'(t)|^{-1}\notag\\&\times
       [|\overrightarrow{r}'(t)|+|(1+\frac{1}{6}\varepsilon_1\Delta t^2)\overrightarrow{r}'(t)+(\frac{1}{6}\Delta t^2+\frac{1}{6}\varepsilon_3\Delta t^2)\overrightarrow{r}'''(t)|]^{-1}\notag\\&\times
    [|\overrightarrow{r}'(t)|^2-|(1+\frac{1}{6}\varepsilon_1\Delta t^2)\overrightarrow{r}'(t)|^2-
    2(1+\frac{1}{6}\varepsilon_1\Delta t^2)(\frac{1}{6}\Delta t^2+\frac{1}{6}\varepsilon_3\Delta t^2)\overrightarrow{r}'(t)\cdot\overrightarrow{r}'''(t)
    \notag\\&\ \ \ -|(\frac{1}{6}\Delta t^2+\frac{1}{6}\varepsilon_3\Delta t^2)\overrightarrow{r}'''(t)|^2] \notag\\
      =&\overrightarrow{r}'(t)\lim_{\Delta t\rightarrow 0^+} \frac{1}{\Delta t}|(1+\frac{1}{6}\varepsilon_1\Delta t^2)\overrightarrow{r}'(t)+(\frac{1}{6}\Delta t^2+\frac{1}{6}\varepsilon_3\Delta t^2)\overrightarrow{r}'''(t)|^{-1}|\overrightarrow{r}'(t)|^{-1}\notag\\&\times
       [|\overrightarrow{r}'(t)|+|(1+\frac{1}{6}\varepsilon_1\Delta t^2)\overrightarrow{r}'(t)+(\frac{1}{6}\Delta t^2+\frac{1}{6}\varepsilon_3\Delta t^2)\overrightarrow{r}'''(t)|]^{-1}\notag\\&\times
    [-(\frac{1}{3}\varepsilon_1\Delta t^2+\frac{1}{36}\varepsilon_1^2\Delta t^4)|\overrightarrow{r}'(t)|^2-
    2(1+\frac{1}{6}\varepsilon_1\Delta t^2)(\frac{1}{6}\Delta t^2+\frac{1}{6}\varepsilon_3\Delta t^2)\overrightarrow{r}'(t)\cdot\overrightarrow{r}'''(t)
    \notag\\&\ \ \ -(\frac{1}{6}\Delta t^2+\frac{1}{6}\varepsilon_3\Delta t^2)^2|\overrightarrow{r}'''(t)|^2]\notag\\
       =&\overrightarrow{r}'(t)\lim_{\Delta t\rightarrow 0^+} |(1+\frac{1}{6}\varepsilon_1\Delta t^2)\overrightarrow{r}'(t)+(\frac{1}{6}\Delta t^2+\frac{1}{6}\varepsilon_3\Delta t^2)\overrightarrow{r}'''(t)|^{-1}|\overrightarrow{r}'(t)|^{-1}\notag\\&\times
       [|\overrightarrow{r}'(t)|+|(1+\frac{1}{6}\varepsilon_1\Delta t^2)\overrightarrow{r}'(t)+(\frac{1}{6}\Delta t^2+\frac{1}{6}\varepsilon_3\Delta t^2)\overrightarrow{r}'''(t)|]^{-1}\notag\\&\times
    [-(\frac{1}{3}\varepsilon_1\Delta t+\frac{1}{36}\varepsilon_1^2\Delta t^3)|\overrightarrow{r}'(t)|^2-
    2(1+\frac{1}{6}\varepsilon_1\Delta t^2)(\frac{1}{6}\Delta t+\frac{1}{6}\varepsilon_3\Delta t)\overrightarrow{r}'(t)\cdot\overrightarrow{r}'''(t)
    \notag\\&\ \ \ -\Delta t(\frac{1}{6}\Delta t+\frac{1}{6}\varepsilon_3\Delta t)^2|\overrightarrow{r}'''(t)|^2]\notag\\
       =&\overrightarrow{0}
  ,\label{xxsecspa13}
\end{align}
$$|\lim_{\Delta t\rightarrow 0^+}\frac{\overrightarrow{e}_{\overrightarrow{PQ}_{13}}(\Delta t)-\overrightarrow{e}_{\overrightarrow{PQ}_{13}}(0)}{\Delta t}|=0;$$

\begin{align}\overrightarrow{e}_{\overrightarrow{PQ}_{23}}& = \frac{(\frac{1}{2}\Delta t^2+\frac{1}{6}\varepsilon_2\Delta t^3)\overrightarrow{r}''(t)+(\frac{1}{6}\Delta t^3+\frac{1}{6}\varepsilon_3\Delta t^3)\overrightarrow{r}'''(t)}{\sqrt{[(\frac{1}{2}\Delta t^2+\frac{1}{6}\varepsilon_2\Delta t^3)\overrightarrow{r}''(t)+(\frac{1}{6}\Delta t^3+\frac{1}{6}\varepsilon_3\Delta t^3)\overrightarrow{r}'''(t)]^2}}\notag\\
& =\frac{(1+\frac{1}{3}\varepsilon_2\Delta t)\overrightarrow{r}''(t)+(\frac{1}{3}\Delta t+\frac{1}{3}\varepsilon_3\Delta t)\overrightarrow{r}'''(t)}{\sqrt{[(1+\frac{1}{3}\varepsilon_2\Delta t)\overrightarrow{r}''(t)+(\frac{1}{3}\Delta t+\frac{1}{3}\varepsilon_3\Delta t)\overrightarrow{r}'''(t)]^2}},\notag\end{align}

\begin{align}&\lim_{\Delta t\rightarrow 0^+}\frac{\overrightarrow{e}_{\overrightarrow{PQ}_{23}}(\Delta t)-\overrightarrow{e}_{\overrightarrow{PQ}_{23}}(0)}{\Delta t}\notag\\
  =&\lim_{\Delta t\rightarrow 0^+} \frac{1}{\Delta t}[\frac{(1+\frac{1}{3}\varepsilon_2\Delta t)\overrightarrow{r}''(t)+(\frac{1}{3}\Delta t+\frac{1}{3}\varepsilon_3\Delta t)\overrightarrow{r}'''(t)}{|(1+\frac{1}{3}\varepsilon_2\Delta t)\overrightarrow{r}''(t)+(\frac{1}{3}\Delta t+\frac{1}{3}\varepsilon_3\Delta t)\overrightarrow{r}'''(t)|}-\frac{\overrightarrow{r}''(t)}{|\overrightarrow{r}''(t)|}] \notag\\
  \notag\\
  =&\lim_{\Delta t\rightarrow 0^+} \frac{1}{\Delta t}[\frac{\overrightarrow{r}''(t)}{|(1+\frac{1}{3}\varepsilon_2\Delta t)\overrightarrow{r}''(t)+(\frac{1}{3}\Delta t+\frac{1}{3}\varepsilon_3\Delta t)\overrightarrow{r}'''(t)|}-\frac{\overrightarrow{r}''(t)}{|\overrightarrow{r}''(t)|}]\notag\\&\ \ \ +
  \lim_{\Delta t\rightarrow 0^+}\frac{\frac{1}{3}\varepsilon_2\overrightarrow{r}''(t)+(\frac{1}{3}+\frac{1}{3}\varepsilon_3)\overrightarrow{r}'''(t)}{|(1+\frac{1}{3}\varepsilon_2\Delta t)\overrightarrow{r}''(t)+(\frac{1}{3}\Delta t+\frac{1}{3}\varepsilon_3\Delta t)\overrightarrow{r}'''(t)|}\notag\\
    =&\overrightarrow{r}''(t)\lim_{\Delta t\rightarrow 0^+} \frac{1}{\Delta t}|(1+\frac{1}{3}\varepsilon_2\Delta t)\overrightarrow{r}''(t)+(\frac{1}{3}\Delta t+\frac{1}{3}\varepsilon_3\Delta t)\overrightarrow{r}'''(t)|^{-1}|\overrightarrow{r}''(t)|^{-1}\notag\\&\times
    [|\overrightarrow{r}''(t)|-|(1+\frac{1}{3}\varepsilon_2\Delta t)\overrightarrow{r}''(t)+(\frac{1}{3}\Delta t+\frac{1}{3}\varepsilon_3\Delta t)\overrightarrow{r}'''(t)|] +\frac{\overrightarrow{r}'''(t)}{3|\overrightarrow{r}''(t)|}\notag\\
      =&\overrightarrow{r}''(t)\lim_{\Delta t\rightarrow 0^+} \frac{1}{\Delta t}|(1+\frac{1}{3}\varepsilon_2\Delta t)\overrightarrow{r}''(t)+(\frac{1}{3}\Delta t+\frac{1}{3}\varepsilon_3\Delta t)\overrightarrow{r}'''(t)|^{-1}|\overrightarrow{r}''(t)|^{-1}\notag\\&\times
       [|\overrightarrow{r}''(t)|+|(1+\frac{1}{3}\varepsilon_2\Delta t)\overrightarrow{r}''(t)+(\frac{1}{3}\Delta t+\frac{1}{3}\varepsilon_3\Delta t)\overrightarrow{r}'''(t)|]^{-1}\notag\\&\times
    [|\overrightarrow{r}''(t)|^2-|(1+\frac{1}{3}\varepsilon_2\Delta t)\overrightarrow{r}''(t)+(\frac{1}{3}\Delta t+\frac{1}{3}\varepsilon_3\Delta t)\overrightarrow{r}'''(t)|^2]+\frac{\overrightarrow{r}'''(t)}{3|\overrightarrow{r}''(t)|} \notag\\
      =&\overrightarrow{r}''(t)\lim_{\Delta t\rightarrow 0^+} \frac{1}{\Delta t}|(1+\frac{1}{3}\varepsilon_2\Delta t)\overrightarrow{r}''(t)+(\frac{1}{3}\Delta t+\frac{1}{3}\varepsilon_3\Delta t)\overrightarrow{r}'''(t)|^{-1}|\overrightarrow{r}''(t)|^{-1}\notag\\&\times
       [|\overrightarrow{r}''(t)|+|(1+\frac{1}{3}\varepsilon_2\Delta t)\overrightarrow{r}''(t)+(\frac{1}{3}\Delta t+\frac{1}{3}\varepsilon_3\Delta t)\overrightarrow{r}'''(t)|]^{-1}\notag\\&\times
    [|\overrightarrow{r}''(t)|^2-|(1+\frac{1}{3}\varepsilon_2\Delta t)\overrightarrow{r}''(t)|^2-
    2(1+\frac{1}{3}\varepsilon_2\Delta t)(\frac{1}{3}\Delta t+\frac{1}{3}\varepsilon_3\Delta t)\overrightarrow{r}''(t)\cdot\overrightarrow{r}'''(t)
    \notag\\&\ \ \ -|(\frac{1}{3}\Delta t+\frac{1}{3}\varepsilon_3\Delta t)\overrightarrow{r}'''(t)|^2] +\frac{\overrightarrow{r}'''(t)}{3|\overrightarrow{r}''(t)|}\notag\\
      =&\overrightarrow{r}''(t)\lim_{\Delta t\rightarrow 0^+} \frac{1}{\Delta t}|(1+\frac{1}{3}\varepsilon_2\Delta t)\overrightarrow{r}''(t)+(\frac{1}{3}\Delta t+\frac{1}{3}\varepsilon_3\Delta t)\overrightarrow{r}'''(t)|^{-1}|\overrightarrow{r}''(t)|^{-1}\notag\\&\times
       [|\overrightarrow{r}''(t)|+|(1+\frac{1}{3}\varepsilon_2\Delta t)\overrightarrow{r}''(t)+(\frac{1}{3}\Delta t+\frac{1}{3}\varepsilon_3\Delta t)\overrightarrow{r}'''(t)|]^{-1}\notag\\&\times
    [-(\frac{2}{3}\varepsilon_2\Delta t+\frac{1}{9}\varepsilon_2^2\Delta t^2)|\overrightarrow{r}''(t)|^2-
    2(1+\frac{1}{3}\varepsilon_2\Delta t)(\frac{1}{3}\Delta t+\frac{1}{3}\varepsilon_3\Delta t)\overrightarrow{r}''(t)\cdot\overrightarrow{r}'''(t)
    \notag\\&\ \ \ -(\frac{1}{3}\Delta t+\frac{1}{3}\varepsilon_3\Delta t)^2|\overrightarrow{r}'''(t)|^2]+\frac{\overrightarrow{r}'''(t)}{3|\overrightarrow{r}''(t)|}\notag\\
           =&\overrightarrow{r}''(t)\lim_{\Delta t\rightarrow 0^+}|(1+\frac{1}{3}\varepsilon_2\Delta t)\overrightarrow{r}''(t)+(\frac{1}{3}\Delta t+\frac{1}{3}\varepsilon_3\Delta t)\overrightarrow{r}'''(t)|^{-1}|\overrightarrow{r}''(t)|^{-1}\notag\\&\times
       [|\overrightarrow{r}''(t)|+|(1+\frac{1}{3}\varepsilon_2\Delta t)\overrightarrow{r}''(t)+(\frac{1}{3}\Delta t+\frac{1}{3}\varepsilon_3\Delta t)\overrightarrow{r}'''(t)|]^{-1}\notag\\&\times
    [-(\frac{2}{3}\varepsilon_2+\frac{1}{9}\varepsilon_2^2\Delta t)|\overrightarrow{r}''(t)|^2-
    2(1+\frac{1}{3}\varepsilon_2\Delta t)(\frac{1}{3}+\frac{1}{3}\varepsilon_3)\overrightarrow{r}''(t)\cdot\overrightarrow{r}'''(t)
    \notag\\&\ \ \ -\Delta t(\frac{1}{3}+\frac{1}{3}\varepsilon_3)^2|\overrightarrow{r}'''(t)|^2]+\frac{\overrightarrow{r}'''(t)}{3|\overrightarrow{r}''(t)|}\notag\\
       =&|\overrightarrow{r}''(t)|^{-2}(2|\overrightarrow{r}''(t)|^{-1})
  (-\frac{2}{3} \overrightarrow{r}''(t)\cdot\overrightarrow{r}'''(t))\overrightarrow{r}''(t)
   +\frac{\overrightarrow{r}'''(t)}{3|\overrightarrow{r}''(t)|}\notag\\
       =&-\frac{1}{3}|\overrightarrow{r}''(t)|^{-3}
   \overrightarrow{r}''(t)\cdot\overrightarrow{r}'''(t)\overrightarrow{r}''(t)
   +\frac{\overrightarrow{r}'''(t)}{3|\overrightarrow{r}''(t)|}
  ,\label{xxsecspa14}
\end{align}
$$|\lim_{\Delta t\rightarrow 0^+}\frac{\overrightarrow{e}_{\overrightarrow{PQ}_{23}}(\Delta t)-\overrightarrow{e}_{\overrightarrow{PQ}_{23}}(0)}{\Delta t}|= \frac{|-
   \overrightarrow{r}''(t)\cdot\overrightarrow{r}'''(t)\overrightarrow{r}''(t)
   +\overrightarrow{r}''^2(t)\overrightarrow{r}'''(t)|}{3|\overrightarrow{r}''(t)|^{3}}.$$

\end{proof}

\textbf{Note:} 1. We call (\ref{xxsecspa12})-(\ref{xxsecspa14})  the local projective rotational velocity functions of (\ref{xxsecspa1}). They will be denoted by $\overrightarrow{\psi}_{12}(t)$, $\overrightarrow{\psi}_{13}(t)$, and $\overrightarrow{\psi}_{23}(t)$ respectively.

2. In the following discussion, we consistently assume that the projective vectors  of a vector  onto the three frame planes  of any frame system are non-zero.
For example, the aforementioned $\overrightarrow{PQ}_{12}$ , $\overrightarrow{PQ}_{13}$ , and $\overrightarrow{PQ}_{23}$.

Before proceeding further, we must define the sign function for (\ref{xxsecspa1}) as follows:
 $$\epsilon(t)=\frac{\overrightarrow{r}'(t)\wedge\overrightarrow{r}''(t)\cdot\overrightarrow{r}'''(t)}{|\overrightarrow{r}'(t)\wedge\overrightarrow{r}''(t)\cdot\overrightarrow{r}'''(t)|}.$$
Now we are in a position to consider the following uniqueness theorem.
\begin{theorem}\label{xxsec3.6}  Given the following   expressions associated with a space curve $\overrightarrow{r}(t)$ on the interval $(t_0,t_1)$:  \begin{align}
&|\overrightarrow{r}'(t)|,\notag\\\
&\frac{|-
   \overrightarrow{r}'(t)\cdot\overrightarrow{r}''(t)\overrightarrow{r}'(t)
   +\overrightarrow{r}'^2(t)\overrightarrow{r}''(t)|}{2|\overrightarrow{r}'(t)|^{3}}
   ,\notag\\
&
0,\notag\\
& \frac{|-
   \overrightarrow{r}''(t)\cdot\overrightarrow{r}'''(t)\overrightarrow{r}''(t)
   +\overrightarrow{r}''^2(t)\overrightarrow{r}'''(t)|}{3|\overrightarrow{r}''(t)|^{3}},\notag\\
&\frac{\overrightarrow{r}'(t)\wedge\overrightarrow{r}''(t)\cdot\overrightarrow{r}'''(t)}{|\overrightarrow{r}'(t)\wedge\overrightarrow{r}''(t)\cdot\overrightarrow{r}'''(t)|}.\notag
\end{align}
If  $\overrightarrow{r}_1(t)$\,($t_0 <t< t_1$) is a space curve whose $\phi(t)$,  $|\overrightarrow{\psi}_{12}(t)|$, $|\overrightarrow{\psi}_{13}(t)|$, $|\overrightarrow{\psi}_{23}(t)|$, and $\epsilon(t)$  are given by the above expressions respectively, then $\overrightarrow{r}_1(t)$ and $\overrightarrow{r}(t)$ coincide on $(t_0,t_1)$ up to position in space.
 \end{theorem}
\begin{proof} Note first that
 \begin{align}
{\overrightarrow{r}'_1}^2= \overrightarrow{r}'^2 ,\label{xxsecspa15}
\end{align}
so differentiating the two sides of the above relation yields
\begin{align} \overrightarrow{r}'_1 \cdot\overrightarrow{r}''_1 =\overrightarrow{r}'\cdot\overrightarrow{r}'' .\label{xxsecspa16}\end{align} The hypotheses force
 \begin{align}
&\frac{|-\overrightarrow{r}'_1 \cdot\overrightarrow{r}_1'' \overrightarrow{r}_1'
+\overrightarrow{r}'^2_1 \overrightarrow{r}_1'' | }{2|\overrightarrow{r}_1' |^{3}}
=\frac{|-\overrightarrow{r}' \cdot\overrightarrow{r}'' \overrightarrow{r}'
+\overrightarrow{r}'^2 \overrightarrow{r}'' |}{2|\overrightarrow{r}' |^{3}},\label{xxsecspa17}\\
&\frac{|-\overrightarrow{r}''_1 \cdot\overrightarrow{r}'''_1
\overrightarrow{r}''_1 +\overrightarrow{r}_1''^2\overrightarrow{r}'''_1 |}{3|\overrightarrow{r}''_1 |^{3}}
=\frac{|-\overrightarrow{r}'' \cdot\overrightarrow{r}'''
\overrightarrow{r}'' +\overrightarrow{r}''^2 \overrightarrow{r}''' |}{3|\overrightarrow{r}'' |^{3}}.\label{xxsecspa18}
\end{align}
Taking into account (\ref{xxsecspa15}) and comparing the two sides of (\ref{xxsecspa17}), we infer that
 \begin{align}
|-\overrightarrow{r}'_1 \cdot\overrightarrow{r}''_1 \overrightarrow{r}'_1 +{\overrightarrow{r}'_1}^2 \overrightarrow{r}''_1 |
 = |-\overrightarrow{r}' \cdot\overrightarrow{r}'' \overrightarrow{r}' +\overrightarrow{r}'^2 \overrightarrow{r}'' |
,\notag\end{align}
or equivalently
\begin{align}
&(\overrightarrow{r}'_1 \cdot\overrightarrow{r}''_1 )^2{\overrightarrow{r}'_1}^2
-2(\overrightarrow{r}'_1 \cdot\overrightarrow{r}''_1 )^2{\overrightarrow{r}'_1}^2
+{\overrightarrow{r}'_1}^4 {\overrightarrow{r}''_1}^2
\notag\\
=& (\overrightarrow{r}' \cdot\overrightarrow{r}'' )^2\overrightarrow{r}'^2
-2(\overrightarrow{r}' \cdot\overrightarrow{r}'' )^2\overrightarrow{r}'^2
+\overrightarrow{r}'^4 \overrightarrow{r}''^2,\notag
\end{align}
which degenerates into
 \begin{align}
{\overrightarrow{r}''_1}^2 =\overrightarrow{r}''^2\label{xxsecspa19}
\end{align}
by (\ref{xxsecspa15}) and (\ref{xxsecspa16}). Differentiating  (\ref{xxsecspa19}), we obtain
 \begin{align}
{\overrightarrow{r}''_1} \cdot{\overrightarrow{r}'''_1} =  \overrightarrow{r}'' \cdot  \overrightarrow{r}''' \label{xxsecspa20}
.\end{align}
Analogously, considering (\ref{xxsecspa19}) and comparing the two sides of (\ref{xxsecspa18}) gives \begin{align}
|-\overrightarrow{r}''_1 \cdot\overrightarrow{r}'''_1
\overrightarrow{r}''_1 +{\overrightarrow{r}''_1}^2 \overrightarrow{r}'''_1 |
=|-\overrightarrow{r}'' \cdot\overrightarrow{r}'''
\overrightarrow{r}'' +\overrightarrow{r}''^2 \overrightarrow{r}''' |,\notag
\end{align}
that is
\begin{align}
&(\overrightarrow{r}''_1 \cdot\overrightarrow{r}'''_1 )^2{\overrightarrow{r}''_1}^2
-2(\overrightarrow{r}''_1 \cdot\overrightarrow{r}'''_1 )^2{\overrightarrow{r}''_1}^2
+{\overrightarrow{r}''_1}^4 {\overrightarrow{r}'''_1}^2
\notag\\
=&(\overrightarrow{r}'' \cdot\overrightarrow{r}''' )^2\overrightarrow{r}''^2
-2(\overrightarrow{r}'' \cdot\overrightarrow{r}''' )^2\overrightarrow{r}''^2
+\overrightarrow{r}''^4 \overrightarrow{r}'''^2.\notag
\end{align}
By applying (\ref{xxsecspa19}) and (\ref{xxsecspa20}) to the above relation, we get
 \begin{align}
{\overrightarrow{r}'''_1}^2 =\overrightarrow{r}'''^2 .\label{xxsecspa21}
\end{align}
 Differentiating (\ref{xxsecspa16}) leads to
\begin{align} {\overrightarrow{r}''_1}^2
+\overrightarrow{r}'_1 \cdot\overrightarrow{r}'''_1 =\overrightarrow{r}''^2 +\overrightarrow{r}' \cdot\overrightarrow{r}'''  .\notag\end{align}
Taking (\ref{xxsecspa19}) into the above relation yields
\begin{align} \overrightarrow{r}'_1 \cdot\overrightarrow{r}'''_1=\overrightarrow{r}' \cdot\overrightarrow{r}'''
 .\label{xxsecspa22}\end{align}

Before continuing our proof, we need the following Lagrange's identity:
\begin{align}( \overrightarrow{r}'  \wedge\overrightarrow{r}'' )^2=&\overrightarrow{r}'^2 \overrightarrow{r}''^2
-(\overrightarrow{r}' \cdot\overrightarrow{r}'' )^2.\notag
\end{align}
This together with (\ref{xxsecspa15})-(\ref{xxsecspa16}), (\ref{xxsecspa19}) tells us \begin{align}|\overrightarrow{r}'_1 \wedge\overrightarrow{r}''_1|=|\overrightarrow{r}' \wedge\overrightarrow{r}''| .\label{xxsecspa23}
\end{align}
Using Lagrange's identity  again, we obtain
\begin{align}
&(\overrightarrow{r}' \wedge\overrightarrow{r}'' \cdot\overrightarrow{r}''' )^2\notag\\
=&( \overrightarrow{r}'  \wedge\overrightarrow{r}'' )^2\overrightarrow{r}'''^2
-[(\overrightarrow{r}' \wedge\overrightarrow{r}'') \wedge\overrightarrow{r}''' ]^2\notag\\
=&( \overrightarrow{r}'  \wedge\overrightarrow{r}'' )^2\overrightarrow{r}'''^2
-(\overrightarrow{r}' \cdot\overrightarrow{r}''' \overrightarrow{r}''
-\overrightarrow{r}'' \cdot\overrightarrow{r}''' \overrightarrow{r}' )^2\notag\\
=&( \overrightarrow{r}'  \wedge\overrightarrow{r}'' )^2\overrightarrow{r}'''^2
-(\overrightarrow{r}' \cdot\overrightarrow{r}''' )^2\overrightarrow{r}''^2 \notag\\
&\ \ \ +2(\overrightarrow{r}' \cdot\overrightarrow{r}''' )( \overrightarrow{r}''  \cdot\overrightarrow{r}''' )(\overrightarrow{r}' \cdot\overrightarrow{r}'' )
-(\overrightarrow{r}'' \cdot\overrightarrow{r}''' )^2\overrightarrow{r}'^2 .\label{xxsecspa24}
\end{align}
This together with (\ref{xxsecspa15})-(\ref{xxsecspa16}), (\ref{xxsecspa19})-(\ref{xxsecspa23}) gives rise to
\begin{align}
&(\overrightarrow{r}'_1 \wedge\overrightarrow{r}''_1 \cdot\overrightarrow{r}'''_1 )^2=(\overrightarrow{r}' \wedge\overrightarrow{r}'' \cdot\overrightarrow{r}''' )^2
.\label{xxsecspa25}
\end{align}
So  the curvatures of $\overrightarrow{r}_1$, $\overrightarrow{r}$ are equal by (\ref{xxsecspa15}) and (\ref{xxsecspa23}), as are their torsions by (\ref{xxsecspa23}) and  (\ref{xxsecspa25}) (remember to use the sign function of $\overrightarrow{r}_1$). This completes the proof.
\end{proof}

Now we can say that space curves can be classified by $\phi(t)$,  $|\overrightarrow{\psi}_{12}(t)|$, $|\overrightarrow{\psi}_{13}(t)|$, $|\overrightarrow{\psi}_{23}(t)|$, and $\epsilon(t)$.

\subsection{Two space curves with symmetric rotating frames}

Consider the following pair of simple ($C^2$-) curves:
\begin{align}
  \overrightarrow{r}_1 (t)=&\{x_1(t),y_1(t),z_1(t)\}, \label{xxsecspa26}\\
   \overrightarrow{r}_2 (t)=&\{x_2(t),y_2(t),z_2(t)\}, \label{xxsecspa27}
\end{align}
where $t_0 < t<  t_1$. Assume $P(x_1(t),y_1(t),z_1(t))$ is an arbitrary point on (\ref{xxsecspa26}). Then we can set up the following frames:
\begin{align}
 &\{P(t); \overrightarrow{i}(t), \overrightarrow{j}(t), \overrightarrow{k}(t)\}, \label{xxsecspa28}\\
& \{P(t);\overrightarrow{e}_1(t), \overrightarrow{e}_2(t), \overrightarrow{e}_3(t) \}. \label{xxsecspa29}
\end{align}
Here $\overrightarrow{i}(t), \overrightarrow{j}(t), \overrightarrow{k}(t)$ are unit vectors in the same direction as the standard basis vector; $$\overrightarrow{e}_1(t)= \frac{\{x_2(t)-x_1(t),y_2(t)-y_1(t),z_2(t)-z_1(t)\}}{\sqrt{(x_2(t)-x_1(t))^2+(y_2(t)-y_1(t))^2+(z_2(t)-z_1(t))^2}},$$  $\overrightarrow{e}_i(t)$\,($i=1,2,3$)  constitute a right-handed Cartesian frame. Let the coordinates of (\ref{xxsecspa29}) be denoted by $\xi, \eta, \zeta$.  Then the equation of (\ref{xxsecspa27}) under (\ref{xxsecspa29}) will  be
\begin{equation}\label{xxsecspa30}
  \xi(t)= \sqrt{(x_2(t)-x_1(t))^2+(y_2(t)-y_1(t))^2+(z_2(t)-z_1(t))^2 },  \ \  \eta(t)=0, \ \  \zeta(t)=0.
\end{equation}

\textbf{Note:} In this subsection, the default parameter range is $t_0 <t< t_1$, (\ref{xxsecspa26}) and (\ref{xxsecspa27}) do not intersect within this interval.

 In the setting of (\ref{xxsecspa28}) and  (\ref{xxsecspa29}),  the  moving point $S(x_2(t),y_2(t),z_2(t))$ on (\ref{xxsecspa27}) undergoes  moving along the line and rotating at the same time.
Let us still use $\overrightarrow{PS}_{A}$ to denote the projective vector of $\overrightarrow{PS}$ onto the  plane spanned by $\overrightarrow{i}(t),\overrightarrow{j}(t)$, $\overrightarrow{PS}_{B}$ the projective vector of $\overrightarrow{PS}$ onto the plane of $\overrightarrow{i}(t),\overrightarrow{k}(t)$, and $\overrightarrow{PS}_{C}$ the projective vector of $\overrightarrow{PS}$ onto  the plane of $\overrightarrow{j}(t),\overrightarrow{k}(t)$.  Remember that these  are non-zero vectors. Then we can describe the behaviour of the moving points on  (\ref{xxsecspa27}) by the following proposition.
\begin{proposition}\label{xxsec3.7} Let $S$ be a moving point on (\ref{xxsecspa27}), and let its parameter be $t$.
\begin{enumerate}
  \item[$\mathrm{(i)}$]  For $\mathcal{D}=|\overrightarrow{PS}|$, we have
   \begin{align}
\frac{ \mathrm{d}\mathcal{D}}{\mathrm{d}t} =&[(x_2(t)-x_1(t))^2+(y_2(t)-y_1(t))^2+(z_2(t)-z_1(t))^2]^{-\frac{1}{2}}\notag\\
&\times[(x_2(t)-x_1(t))(x_2'(t)-x_1'(t))+(y_2(t)-y_1(t))(y_2'(t)-y_1'(t))\notag\\&
+(z_2(t)-z_1(t))(z_2'(t)-z_1'(t))].\notag
\end{align}
\item[$\mathrm{(ii)}$] The rotational speed of $\overrightarrow{PS}_{A}=\{x_2(t)-x_1(t),y_2(t)-y_1(t),0\}$  is $$\frac{|(x_2(t)-x_1(t))(y_2'(t)-y_1'(t))-(x_2'(t)-x_1'(t))(y_2(t)-y_1(t))|}{(x_2(t)-x_1(t))^2+(y_2(t)-y_1(t))^2}.$$
    \item[$\mathrm{(iii)}$] The rotational speed of $\overrightarrow{PS}_{B}=\{x_2(t)-x_1(t),0,z_2(t)-z_1(t)\}\}$ is $$\frac{|(x_2(t)-x_1(t))(z_2'(t)-z_1'(t))-(x_2'(t)-x_1'(t))(z_2(t)-z_1(t))|}{(x_2(t)-x_1(t))^2+(z_2(t)-z_1(t))^2}.$$
        \item[$\mathrm{(iv)}$] The rotational speed of $\overrightarrow{PS}_{C}=\{0,y_2(t)-y_1(t),z_2(t)-z_1(t)\}$ is $$\frac{|(y_2(t)-y_1(t))(z_2'(t)-z_1'(t))-(y_2'(t)-y_1'(t))(z_2(t)-z_1(t))|}{(y_2(t)-y_1(t))^2+(z_2(t)-z_1(t))^2}.$$
\end{enumerate}
\end{proposition}
\begin{proof}

(i) This follows immediately from  $$ \mathcal{D}= \sqrt{(x_2(t)-x_1(t))^2+(y_2(t)-y_1(t))^2+(z_2(t)-z_1(t))^2 },$$
which is equal to $\xi(t)$ in (\ref{xxsecspa30}).

(ii) Note that $$\overrightarrow{e}_{\overrightarrow{PS}_{A}}= \frac{\{x_2(t)-x_1(t),y_2(t)-y_1(t),0\}}{\sqrt{(x_2(t)-x_1(t))^2+(y_2(t)-y_1(t))^2}},$$
so it can be seen from the calculation process of Proposition \ref{xxsec3.1}(ii) that the conclusion holds.

(iii)-(iv) The proof is similar to that of (ii).
\end{proof}

\section{The behaviour of moving points on curves lying on surfaces}

\subsection{Curves on  surfaces with general rotating frames}

Let
\begin{equation}\label{xxsecsur1}
  \overrightarrow{r}(u,v)=\{x(u,v),y(u,v),z(u,v)\}, \ \ \ \ (u,v)\in D
\end{equation}
be a  surface of class $C^2$, and let
\begin{equation}\label{xxsecsur2} u=u(t), v=v(t), \ \ \ t_0 < t<  t_1\end{equation}
be the equation of a curve  lying on (\ref{xxsecsur1}).  Set up a rotating frame at $O(0,0,0)$:
\begin{equation}\label{xxsecsur3}
 \{O;\overrightarrow{e}_1(u(t),v(t)), \overrightarrow{e}_2(u(t),v(t)), \overrightarrow{e}_3(u(t),v(t)) \}, \ \ \ t_0 < t<  t_1,
\end{equation}
where $\overrightarrow{e}_1(u(t),v(t)) =\frac{\overrightarrow{r}(u(t),v(t))}{\sqrt{\overrightarrow{r}^2(u(t),v(t))}}$, $\overrightarrow{e}_i(u(t),v(t))$\,($i=1,2,3$) form a  a right-handed Cartesian frame. If the coordinates of (\ref{xxsecsur3}) are denoted by $\xi, \eta, \zeta$, then the equation of (\ref{xxsecsur2}) under (\ref{xxsecsur3}) is
\begin{equation}\label{xxsecsur4}
  \xi(t)= \sqrt{\overrightarrow{r}^2(u(t),v(t)) },  \ \  \eta(t)=0, \ \  \zeta(t)=0.
\end{equation}

\textbf{Note:} In this subsection, by default, $t_0 <t< t_1$, and (\ref{xxsecsur1}) does not pass through $(0,0,0)$.

Now we are in a position to describe the behaviour of the moving points on (\ref{xxsecsur2}).
\begin{proposition}\label{xxsec4.1} Let $P$ be a moving point on (\ref{xxsecsur2}), and let its parameter be $t$.
\begin{enumerate}
  \item[$\mathrm{(i)}$]  For $\mathcal{D}=|\overrightarrow{OP}|$, we have
  \begin{align}
\frac{ \mathrm{d}\mathcal{D}}{\mathrm{d}t}  =&\frac{\overrightarrow{r} (\overrightarrow{r}_uu'+
 \overrightarrow{r}_vv')}{ \sqrt{\overrightarrow{r}^2 } }
.\notag
\end{align}

\item[$\mathrm{(ii)}$] The rotational speed of $\overrightarrow{OP}_{A}=\{x(u(t),v(t)),y(u(t),v(t)),0\}$ with respect to $t$ is $$| \frac{-\overrightarrow{r}_{A}\cdot(\frac{\partial \overrightarrow{r}_{A}}{\partial u}u'+
 \frac{\partial \overrightarrow{r}_{A}}{\partial v}v')\overrightarrow{r}_{A}+(\frac{\partial \overrightarrow{r}_{A}}{\partial u}u'+
 \frac{\partial \overrightarrow{r}_{A}}{\partial v}v')\overrightarrow{r}_{A}^2}{(\overrightarrow{r}_{A}^2)^{3/2}}|.$$
    \item[$\mathrm{(iii)}$] The rotational speed of  $\overrightarrow{OP}_{B}=\{x(u(t),v(t)),0,z(u(t),v(t))\}$ with respect to $t$ is $$| \frac{-\overrightarrow{r}_{B}\cdot(\frac{\partial \overrightarrow{r}_{B}}{\partial u}u'+
 \frac{\partial \overrightarrow{r}_{B}}{\partial v}v')\overrightarrow{r}_{B}+(\frac{\partial \overrightarrow{r}_{B}}{\partial u}u'+
 \frac{\partial \overrightarrow{r}_{B}}{\partial v}v')\overrightarrow{r}_{B}^2}{(\overrightarrow{r}_{B}^2)^{3/2}}|.$$
        \item[$\mathrm{(iv)}$] The rotational speed of  $\overrightarrow{OP}_{C}=\{0,y(u(t),v(t)),z(u(t),v(t))\}$ with respect to $t$ is $$| \frac{-\overrightarrow{r}_{C}\cdot(\frac{\partial \overrightarrow{r}_{C}}{\partial u}u'+
 \frac{\partial \overrightarrow{r}_{C}}{\partial v}v')\overrightarrow{r}_{C}+(\frac{\partial \overrightarrow{r}_{C}}{\partial u}u'+
 \frac{\partial \overrightarrow{r}_{C}}{\partial v}v')\overrightarrow{r}_{C}^2}{(\overrightarrow{r}_{C}^2)^{3/2}}|.$$
\end{enumerate}
\end{proposition}
\begin{proof}

(i) This follows easily from the fact $\mathcal{D}=\sqrt{\overrightarrow{r}^2(u(t),v(t)) }$, which is equal to $\xi(t)$ in (\ref{xxsecsur4}).

(ii) The following is the proof process:$$\overrightarrow{e}_{\overrightarrow{OP}_{A}}= \frac{\overrightarrow{OP}_{A}}{\sqrt{\overrightarrow{OP}_{A}^2}}= \frac{\overrightarrow{r}_{A}}{\sqrt{\overrightarrow{r}_{A}^2}},$$
\begin{align}&\frac{\mathrm{d} \overrightarrow{e}_{\overrightarrow{OP}_{A}}}{\mathrm{d} t}\notag\\
 =&-\frac{1}{2}(\overrightarrow{r}_{A}^2)^{-\frac{3}{2}}[ 2\overrightarrow{r}_{A} \cdot(\frac{\partial \overrightarrow{r}_{A}}{\partial u}u'+
 \frac{\partial \overrightarrow{r}_{A}}{\partial v}v')]\overrightarrow{r}_{A}+ \frac{\frac{\partial \overrightarrow{r}_{A}}{\partial u}u'+
 \frac{\partial \overrightarrow{r}_{A}}{\partial v}v'}{\sqrt{\overrightarrow{r}_{A}^2}}\notag\\
 =&-(\overrightarrow{r}_{A}^2)^{-\frac{3}{2}}[\overrightarrow{r}_{A}\cdot(\frac{\partial \overrightarrow{r}_{A}}{\partial u}u'+
 \frac{\partial \overrightarrow{r}_{A}}{\partial v}v')]\overrightarrow{r}_{A}\notag\\&+ \frac{(\frac{\partial \overrightarrow{r}_{A}}{\partial u}u'+
 \frac{\partial \overrightarrow{r}_{A}}{\partial v}v')\overrightarrow{r}_{A}^2}{(\overrightarrow{r}_{A}^2)^{3/2}}\notag\\
  =&\frac{[-\overrightarrow{r}_{A} \cdot(\frac{\partial \overrightarrow{r}_{A}}{\partial u}u'+
 \frac{\partial \overrightarrow{r}_{A}}{\partial v}v')]\overrightarrow{r}_{A}+(\frac{\partial \overrightarrow{r}_{A}}{\partial u}u'+
 \frac{\partial \overrightarrow{r}_{A}}{\partial v}v')\overrightarrow{r}_{A}^2}{(\overrightarrow{r}_{A}^2)^{3/2}}, \notag
\end{align}
$$|\frac{\mathrm{d} \overrightarrow{e}_{\overrightarrow{OP}_{A}}}{\mathrm{d} t}|=| \frac{-\overrightarrow{r}_{A}\cdot(\frac{\partial \overrightarrow{r}_{A}}{\partial u}u'+
 \frac{\partial \overrightarrow{r}_{A}}{\partial v}v')\overrightarrow{r}_{A}+(\frac{\partial \overrightarrow{r}_{A}}{\partial u}u'+
 \frac{\partial \overrightarrow{r}_{A}}{\partial v}v')\overrightarrow{r}_{A}^2}{(\overrightarrow{r}_{A}^2)^{3/2}}|.$$

(iii)-(iv) The proof is similar to that of (ii).

\end{proof}

\subsection{Curves on  surfaces with local rotating frames}

Let $$P(x(u(t),v(t)),y(u(t),v(t)),z(u(t),v(t)))$$ be a point on (\ref{xxsecsur2}). We can set up a rotating frame at $P$ as follows:
\begin{equation}\label{xxsecsur5}
 \{P;\overrightarrow{e}_1(\Delta t), \overrightarrow{e}_2(\Delta t), \overrightarrow{e}_3(\Delta t) \},
\end{equation}
where \begin{align}\overrightarrow{e}_1(\Delta t)
= \frac{\overrightarrow{r}((t+\Delta t),v(t+\Delta t))-\overrightarrow{r}(u(t),v(t))}{|\overrightarrow{r}((t+\Delta t),v(t+\Delta t))-\overrightarrow{r}(u(t),v(t))|}, \notag\end{align}
 $\overrightarrow{e}_i(\Delta t)$\,($i=1,2,3$)   form a right-handed Cartesian frame, and where $\Delta t> 0$. If the coordinates of (\ref{xxsecsur5}) are $\xi_P, \eta_P, \zeta_P$, then the equation of (\ref{xxsecsur2}) near point $P$ under (\ref{xxsecsur5}) can be written as
\begin{equation}\label{xxsecsur6}
  \xi_P(\Delta t)= \sqrt{[\overrightarrow{r}((t+\Delta t),v(t+\Delta t))-\overrightarrow{r}(u(t),v(t))]^2 },  \ \  \eta_P(\Delta t)=0, \ \  \zeta_P(\Delta t)=0,
\end{equation}
where $\Delta t>0$.

\textbf{Note:} In this subsection, we assume by default that $t_0 <t< t_1$, $\Delta t> 0$  is small enough, and $\overrightarrow{r}'(u(t),v(t))\wedge\overrightarrow{r}''(u(t),v(t))\cdot\overrightarrow{r}'''(u(t),v(t))\neq 0$, $t_0 <t< t_1$. Moreover, (\ref{xxsecsur1}) is of class $C^3$, and the coordinate net  is regular.

Let $Q$ be a point on (\ref{xxsecsur2}) near $P$. We begin by investigating  its linear motion.
\begin{proposition}\label{xxsec4.2} Let $Q(\Delta t)$ be a moving point on (\ref{xxsecsur2}) near  $P$, and let $\mathcal{D}_P=|\overrightarrow{PQ}|$. Then we have
\begin{align} \frac{ \mathrm{d}\mathcal{D}_P}{\mathrm{d}\Delta t}  =&|\overrightarrow{r}(u(t+\Delta t),v(t+\Delta t))-\overrightarrow{r}(u(t),v(t))|^{-1}\notag\\
&\times[\overrightarrow{r}(u(t+\Delta t),v(t+\Delta t))-\overrightarrow{r}(u(t),v(t))]\notag\\
&\cdot [\overrightarrow{r}_u(u(t+\Delta t),v(t+\Delta t))u'(t+\Delta t)\notag\\&\ \ \ +
\overrightarrow{r}_v(u(t+\Delta t),v(t+\Delta t))v'(t+\Delta t)] ,\ \ \ \Delta t>0 ;\notag
\end{align}
 \begin{align}
\mathcal{D}_{P+}'(0)=&|\overrightarrow{r}_u(u(t),v(t))u'(t)+\overrightarrow{r}_v(u(t),v(t))v'(t)|.\notag
\end{align}

\end{proposition}
\begin{proof} Since $\mathcal{D}_P=\sqrt{[\overrightarrow{r}((t+\Delta t),v(t+\Delta t))-\overrightarrow{r}(u(t),v(t))]^2 }$\,(equals to $\xi_P(\Delta t)$ in (\ref{xxsecsur6})), we have
\begin{align}\frac{ \mathrm{d}\mathcal{D}_P}{\mathrm{d}\Delta t}=&\frac{1}{2}[|\overrightarrow{r}(u(t+\Delta t),v(t+\Delta t))-\overrightarrow{r}(u(t),v(t))|^2]^{-\frac{1}{2}}\notag\\
&\times2[\overrightarrow{r}(u(t+\Delta t),v(t+\Delta t))-\overrightarrow{r}(u(t),v(t))]\notag\\
&\cdot [\overrightarrow{r}_u(u(t+\Delta t),v(t+\Delta t))u'(t+\Delta t)+
\overrightarrow{r}_v(u(t+\Delta t),v(t+\Delta t))v'(t+\Delta t)]\notag\\
=&|\overrightarrow{r}(u(t+\Delta t),v(t+\Delta t))-\overrightarrow{r}(u(t),v(t))|^{-1}\notag\\
&\times[\overrightarrow{r}(u(t+\Delta t),v(t+\Delta t))-\overrightarrow{r}(u(t),v(t))]\notag\\
&\cdot [\overrightarrow{r}_u(u(t+\Delta t),v(t+\Delta t))u'(t+\Delta t)+
\overrightarrow{r}_v(u(t+\Delta t),v(t+\Delta t))v'(t+\Delta t)],\notag\end{align}
where $\Delta t>0$.

Next let us consider $\Delta t\rightarrow 0^+$.  Before further calculation, we first provide the following Taylor formula:
\begin{align}&\overrightarrow{r}u((t+\Delta t),v(t+\Delta t))-\overrightarrow{r}(u(t),v(t))\notag\\
=&\overrightarrow{r}'(u(t),v(t))\Delta t+\frac{1}{2}\overrightarrow{r}''(u(t),v(t))\Delta t^2+\frac{1}{6}(\overrightarrow{r}'''(u(t),v(t))+\overrightarrow{\varepsilon})\Delta t^3,\label{xxsecsur7}
\end{align}
where $\lim_{\Delta t\rightarrow 0^+}\overrightarrow{\varepsilon}=\overrightarrow{0}$.
Now one can compute
\begin{align}\mathcal{D}_{P+}'(0)
 &=\lim_{\Delta t\rightarrow 0^+}\frac{\mathcal{D}_P(\Delta t)-\mathcal{D}_P(0)}{\Delta t}\notag\\
  &=\lim_{\Delta t\rightarrow 0^+}\frac{\sqrt{[\overrightarrow{r}'(u(t),v(t))\Delta t+\frac{1}{2}\overrightarrow{r}''(u(t),v(t))\Delta t^2+\frac{1}{6}(\overrightarrow{r}'''(u(t),v(t))+\overrightarrow{\varepsilon})\Delta t^3]^2}}{\Delta t}\notag\\
 &=\sqrt{\overrightarrow{r}'^2(u(t),v(t))} \notag\\
 &=\sqrt{[\overrightarrow{r}_u(u(t),v(t))u'(t)+\overrightarrow{r}_v(u(t),v(t))v'(t)]^2},\notag
\end{align}
where $\mathcal{D}_P(0)=0$.
\end{proof}

\textbf{Note:} We need to provide necessary explanations for some symbols that will be used subsequently. $u^1,u^2$ denote $u,v$, respectively; $\overrightarrow{r}_1=\overrightarrow{r}_{u^1}$, $\overrightarrow{r}_2=\overrightarrow{r}_{u^2}$; $g_{ij}$ the first fundamental form coefficients, $(g^{ij})$ the inverse matrix of $g_{ij}$\,($i, j=1, 2$); $L_{ij}$\,($i, j=1, 2$) the second fundamental form coefficients; $\Gamma_{ij}^k$\,($i, j, k=1, 2$) the second kind Christoffel symbols. These are consistent with those in standard differential geometry textbooks (cf. \cite{MeiHuang}).

We now proceed to analyze $\mathcal{D}_{P+}'(0)$ further.
\begin{align}(\mathcal{D}_{P+}'(0))^2=
 (\overrightarrow{r}_uu'(t)+\overrightarrow{r}_vv'(t))^2=\overrightarrow{r}_u^2u'^2(t)+2\overrightarrow{r}_u\cdot\overrightarrow{r}_vu'(t)v'(t)+\overrightarrow{r}_v^2v'^2(t).\notag
\end{align}
The above formula can be written in the form
\begin{align} \mathrm{d}\overrightarrow{r}^2=\overrightarrow{r}_u^2\mathrm{d}u^2+
2\overrightarrow{r}_u\cdot\overrightarrow{r}_v\mathrm{d}u\mathrm{d}v+\overrightarrow{r}_v^2\mathrm{d}v^2,\notag
\end{align}
which is just the first fundamental form of (\ref{xxsecsur1}),  and can also be rewritten as
\begin{align} \mathrm{I}=\sum_{i,j=1}^2g_{ij}\mathrm{d}u^i\mathrm{d}u^j.\notag
\end{align}
This means we can get the  first fundamental form of a surface through a way different from calculating the arc length.

Let $Q$ be as above, and let us now turn to the discussion of the projective rotational speed of $\overrightarrow{PQ}$.  Therefore, besides  (\ref{xxsecsur5}), we also need to construct another transitional frame:
\begin{equation}\label{xxsecsur8}
 \{P;\overrightarrow{r}'(u(t),v(t)),\overrightarrow{r}''(u(t),v(t)),\overrightarrow{r}'''(u(t),v(t)) \}, \ \ \ t_0 < t< t_1.
\end{equation}
Under frame (\ref{xxsecsur8}), the equation of (\ref{xxsecsur2}) near point $P$ can take the following form:
\begin{align}&\overrightarrow{r}((u(t+\Delta t),v(t+\Delta t)))-\overrightarrow{r}(u(t),v(t))\notag\\=&g_1(\Delta t)\overrightarrow{r}'(u(t),v(t))+g_2(\Delta t)\overrightarrow{r}''(u(t),v(t))+g_3(\Delta t)\overrightarrow{r}'''(u(t),v(t)).\label{xxsecsur9}
\end{align}

Before proceeding, let us give the following fundamental equations of the natural frame (cf. \cite{MeiHuang} etc.):
\begin{align}\overrightarrow{r}_{ij}=&\sum_{k}\Gamma_{ij}^k\overrightarrow{r}_{k}+L_{ij}\overrightarrow{n},\
\overrightarrow{n}_{i}=-\sum_{j,k}L_{ik}g^{kj}\overrightarrow{r}_{j}.\notag\end{align}
Next we can do the following computation:
\begin{align}\overrightarrow{r}'(u(t),v(t))
=&\sum_i\frac{\mathrm{d}u^i}{\mathrm{d}t}\overrightarrow{r}_i,\label{xxsecsur10}\\
\overrightarrow{r}''(u(t),v(t))=
&\sum_{i,j}\frac{\mathrm{d}u^i}{\mathrm{d}t}\frac{\mathrm{d}u^j}{\mathrm{d}t}\overrightarrow{r}_{ij}+\sum_i\frac{\mathrm{d}^2u^i}{\mathrm{d}t^2}\overrightarrow{r}_i\notag\\
=&\sum_{i,j,k}\Gamma_{ij}^k\frac{\mathrm{d}u^i}{\mathrm{d}t}\frac{\mathrm{d}u^j}{\mathrm{d}t}\overrightarrow{r}_k
+\sum_{i,j}L_{ij}\frac{\mathrm{d}u^i}{\mathrm{d}t}\frac{\mathrm{d}u^j}{\mathrm{d}t}\overrightarrow{n}+\sum_i\frac{\mathrm{d}^2u^i}{\mathrm{d}t^2}\overrightarrow{r}_i
, \label{xxsecsur11}\\
\overrightarrow{r}'''(u(t),v(t))
=&
\sum_{i,j,k,l}(\Gamma_{ij}^k)_l\frac{\mathrm{d}u^i}{\mathrm{d}t}\frac{\mathrm{d}u^j}{\mathrm{d}t}\frac{\mathrm{d}u^l}{\mathrm{d}t}\overrightarrow{r}_k
+2\sum_{i,j,k}\Gamma_{ij}^k\frac{\mathrm{d}^2u^i}{\mathrm{d}t^2}\frac{\mathrm{d}u^j}{\mathrm{d}t}\overrightarrow{r}_k
+\sum_{i,j,k}\Gamma_{ij}^k\frac{\mathrm{d}u^i}{\mathrm{d}t}\frac{\mathrm{d}^2u^j}{\mathrm{d}t^2}\overrightarrow{r}_k\notag\\
&+\sum_{i,j,k,l,m}\Gamma_{ij}^m\Gamma_{ml}^k\frac{\mathrm{d}u^i}{\mathrm{d}t}\frac{\mathrm{d}u^j}{\mathrm{d}t}\frac{\mathrm{d}u^l}{\mathrm{d}t}\overrightarrow{r}_{k}
-\sum_{i,j,k,l,m}L_{ij}L_{lm}g^{mk}\frac{\mathrm{d}u^i}{\mathrm{d}t}\frac{\mathrm{d}u^j}{\mathrm{d}t}
\frac{\mathrm{d}u^l}{\mathrm{d}t}\overrightarrow{r}_{k}\notag\\
&
+\sum_k\frac{\mathrm{d}^3u^k}{\mathrm{d}t^3}\overrightarrow{r}_k+\sum_{i,j,k,l}\Gamma_{ij}^kL_{kl}\frac{\mathrm{d}u^i}{\mathrm{d}t}\frac{\mathrm{d}u^j}{\mathrm{d}t}\frac{\mathrm{d}u^l}{\mathrm{d}t}\overrightarrow{n}+\sum_{i,j,k}(L_{ij})_k\frac{\mathrm{d}u^i}{\mathrm{d}t}\frac{\mathrm{d}u^j}{\mathrm{d}t}\frac{\mathrm{d}u^k}{\mathrm{d}t}\overrightarrow{n}\notag\\
&+2\sum_{i,j}L_{ij}\frac{\mathrm{d}^2u^i}{\mathrm{d}t^2}\frac{\mathrm{d}u^j}{\mathrm{d}t}\overrightarrow{n}+\sum_{i,j}L_{ij}\frac{\mathrm{d}u^i}{\mathrm{d}t}\frac{\mathrm{d}^2u^j}{\mathrm{d}t^2}\overrightarrow{n}
.\label{xxsecsur12}
\end{align}
Since the computational process of (\ref{xxsecsur12}) is lengthy, we present it in Appendix A. Using (\ref{xxsecsur10})-(\ref{xxsecsur12}), we can rewrite (\ref{xxsecsur9}) as
\begin{align}&\overrightarrow{r}((u(t+\Delta t),v(t+\Delta t)))-\overrightarrow{r}(u(t),v(t))\notag\\
=&g_1(\Delta t)\sum_k\frac{\mathrm{d}u^k}{\mathrm{d}t}\overrightarrow{r}_k+g_2(\Delta t)\sum_{i,j,k}\Gamma_{ij}^k\frac{\mathrm{d}u^i}{\mathrm{d}t}\frac{\mathrm{d}u^j}{\mathrm{d}t}\overrightarrow{r}_k
\notag\\&
+g_2(\Delta t)\sum_k\frac{\mathrm{d}^2u^k}{\mathrm{d}t^2}\overrightarrow{r}_k
+g_3(\Delta t)\sum_{i,j,k,l}(\Gamma_{ij}^k)_l\frac{\mathrm{d}u^i}{\mathrm{d}t}\frac{\mathrm{d}u^j}{\mathrm{d}t}\frac{\mathrm{d}u^l}{\mathrm{d}t}\overrightarrow{r}_k\notag\\
&
+2g_3(\Delta t)\sum_{i,j,k}\Gamma_{ij}^k\frac{\mathrm{d}^2u^i}{\mathrm{d}t^2}\frac{\mathrm{d}u^j}{\mathrm{d}t}\overrightarrow{r}_k
+g_3(\Delta t)\sum_{i,j,k}\Gamma_{ij}^k\frac{\mathrm{d}u^i}{\mathrm{d}t}\frac{\mathrm{d}^2u^j}{\mathrm{d}t^2}\overrightarrow{r}_k\notag\\
&+g_3(\Delta t)\sum_{i,j,k,l,m}\Gamma_{ij}^m\Gamma_{ml}^k\frac{\mathrm{d}u^i}{\mathrm{d}t}\frac{\mathrm{d}u^j}{\mathrm{d}t}\frac{\mathrm{d}u^l}{\mathrm{d}t}\overrightarrow{r}_{k}
-g_3(\Delta t)\sum_{i,j,k,l,m}L_{ij}L_{lm}g^{mk}\frac{\mathrm{d}u^i}{\mathrm{d}t}\frac{\mathrm{d}u^j}{\mathrm{d}t}
\frac{\mathrm{d}u^l}{\mathrm{d}t}\overrightarrow{r}_{k}\notag\\
&
+g_3(\Delta t)\sum_k\frac{\mathrm{d}^3u^k}{\mathrm{d}t^3}\overrightarrow{r}_k
+g_2(\Delta t)\sum_{i,j}L_{ij}\frac{\mathrm{d}u^i}{\mathrm{d}t}\frac{\mathrm{d}u^j}{\mathrm{d}t}\overrightarrow{n}\notag\\
&+g_3(\Delta t)\sum_{i,j,k,l}\Gamma_{ij}^kL_{kl}\frac{\mathrm{d}u^i}{\mathrm{d}t}\frac{\mathrm{d}u^j}{\mathrm{d}t}\frac{\mathrm{d}u^l}{\mathrm{d}t}\overrightarrow{n}
+g_3(\Delta t)\sum_{i,j,k}(L_{ij})_k\frac{\mathrm{d}u^i}{\mathrm{d}t}\frac{\mathrm{d}u^j}{\mathrm{d}t}\frac{\mathrm{d}u^k}{\mathrm{d}t}\overrightarrow{n}\notag\\
&+2g_3(\Delta t)\sum_{i,j}L_{ij}\frac{\mathrm{d}^2u^i}{\mathrm{d}t^2}\frac{\mathrm{d}u^j}{\mathrm{d}t}\overrightarrow{n}
+g_3(\Delta t)\sum_{i,j}L_{ij}\frac{\mathrm{d}u^i}{\mathrm{d}t}\frac{\mathrm{d}^2u^j}{\mathrm{d}t^2}\overrightarrow{n}\notag\\
&=\chi_1(\Delta t)\overrightarrow{r}_1+\chi_2(\Delta t)\overrightarrow{r}_2+\chi_3(\Delta t)\overrightarrow{n}
,\label{xxsecsur13}
\end{align}
where
\begin{align}
&\chi_k(\Delta t)\notag\\=&g_1(\Delta t)\frac{\mathrm{d}u^k}{\mathrm{d}t}+g_2(\Delta t)\sum_{i,j}\Gamma_{ij}^k\frac{\mathrm{d}u^i}{\mathrm{d}t}\frac{\mathrm{d}u^j}{\mathrm{d}t}
\notag\\&
+g_2(\Delta t)\frac{\mathrm{d}^2u^k}{\mathrm{d}t^2}
+g_3(\Delta t)\sum_{i,j,l}(\Gamma_{ij}^k)_l\frac{\mathrm{d}u^i}{\mathrm{d}t}\frac{\mathrm{d}u^j}{\mathrm{d}t}\frac{\mathrm{d}u^l}{\mathrm{d}t}\notag\\
&
+2g_3(\Delta t)\sum_{i,j}\Gamma_{ij}^k\frac{\mathrm{d}^2u^i}{\mathrm{d}t^2}\frac{\mathrm{d}u^j}{\mathrm{d}t}
+g_3(\Delta t)\sum_{i,j}\Gamma_{ij}^k\frac{\mathrm{d}u^i}{\mathrm{d}t}\frac{\mathrm{d}^2u^j}{\mathrm{d}t^2}\notag\\
&+g_3(\Delta t)\sum_{i,j,l,m}\Gamma_{ij}^m\Gamma_{ml}^k\frac{\mathrm{d}u^i}{\mathrm{d}t}\frac{\mathrm{d}u^j}{\mathrm{d}t}\frac{\mathrm{d}u^l}{\mathrm{d}t}
-g_3(\Delta t)\sum_{i,j,l,m}L_{ij}L_{lm}g^{mk}\frac{\mathrm{d}u^i}{\mathrm{d}t}\frac{\mathrm{d}u^j}{\mathrm{d}t}
\frac{\mathrm{d}u^l}{\mathrm{d}t} \notag\\
&
+g_3(\Delta t)\frac{\mathrm{d}^3u^k}{\mathrm{d}t^3},  k=1,2,\notag\\
&\chi_3(\Delta t)\notag\\=&g_2(\Delta t)\sum_{i,j}L_{ij}\frac{\mathrm{d}u^i}{\mathrm{d}t}\frac{\mathrm{d}u^j}{\mathrm{d}t} +g_3(\Delta t)\sum_{i,j,k,l}\Gamma_{ij}^kL_{kl}\frac{\mathrm{d}u^i}{\mathrm{d}t}\frac{\mathrm{d}u^j}{\mathrm{d}t}\frac{\mathrm{d}u^l}{\mathrm{d}t}
+g_3(\Delta t)\sum_{i,j,k}(L_{ij})_k\frac{\mathrm{d}u^i}{\mathrm{d}t}\frac{\mathrm{d}u^j}{\mathrm{d}t}\frac{\mathrm{d}u^k}{\mathrm{d}t} \notag\\
&+2g_3(\Delta t)\sum_{i,j}L_{ij}\frac{\mathrm{d}^2u^i}{\mathrm{d}t^2}\frac{\mathrm{d}u^j}{\mathrm{d}t}
+g_3(\Delta t)\sum_{i,j}L_{ij}\frac{\mathrm{d}u^i}{\mathrm{d}t}\frac{\mathrm{d}^2u^j}{\mathrm{d}t^2} .\notag
\end{align}

Next we will project $\overrightarrow{PQ}$ to the three frame planes of
\begin{equation}
 \{P;\overrightarrow{r}_1(u(t),v(t)),\overrightarrow{r}_2(u(t),v(t)),\overrightarrow{n}(u(t),v(t))\}.\notag
\end{equation}
 Let $\overrightarrow{PQ}_{\mathcal{A}}$ be the projective vector of $\overrightarrow{PQ}$ onto the  plane spanned by $\overrightarrow{r}_1,\overrightarrow{r}_2$, $\overrightarrow{PQ}_{\mathcal{B}}$ the projective vector of $\overrightarrow{PQ}$  onto the plane of $\overrightarrow{r}_1,\overrightarrow{n}$, and $\overrightarrow{PQ}_{\mathcal{C}}$ the projective vector  of $\overrightarrow{PQ}$  onto the  plane of $\overrightarrow{r}_2,\overrightarrow{r}_n$. Then we have

\begin{proposition}\label{xxsec4.3}  Let  $Q$ be a moving point on (\ref{xxsecsur2}) near $P$, and let its parameter be $\Delta t$.
\begin{enumerate}
\item[$\mathrm{1)}$]  The rotational speed of $\overrightarrow{PQ}_{\mathcal{A}}$ with respect to $\Delta t$ is
\begin{align}
&|\chi_1(\Delta t)\overrightarrow{r}_1+\chi_2(\Delta t)\overrightarrow{r}_2|^{-3}\notag\\
&\times|\chi_1^2(\Delta t)\chi_2'(\Delta t)g_{12}\overrightarrow{r}_1 +\chi_2(\Delta t)\chi_2'(\Delta t)\chi_1(\Delta t)g_{22}\overrightarrow{r}_1\notag\\
&\ \ \  -\chi_1(\Delta t)\chi_2(\Delta t)\chi_1'(\Delta t)g_{12}\overrightarrow{r}_1-\chi_2^2(\Delta t)\chi_1'(\Delta t)g_{22}\overrightarrow{r}_1\notag\\
&\ \ \ +\chi_1(\Delta t)\chi_1'(\Delta t)\chi_2(\Delta t)g_{11}\overrightarrow{r}_2+\chi_2^2(\Delta t)\chi_1'(\Delta t)g_{12}\overrightarrow{r}_2  \notag\\
&\ \ \ -\chi_1^2(\Delta t)\chi_2'(\Delta t)g_{11}\overrightarrow{r}_2-\chi_1(\Delta t)\chi_2(\Delta t)\chi_2'(\Delta t)g_{12}\overrightarrow{r}_2|
,\ \ \  \Delta t>0;  \notag
 \end{align}
 \begin{align}
&\frac{1}{2}|\sum_k\frac{\mathrm{d}u^k}{\mathrm{d}t}\overrightarrow{r}_k|^{-3}\notag\\&\times
[\sum_{p,q}\sum_{k,l}(\sum_{i,j}\Gamma_{ij}^l\frac{\mathrm{d}u^i}{\mathrm{d}t}\frac{\mathrm{d}u^j}{\mathrm{d}t}\frac{\mathrm{d}u^k}{\mathrm{d}t}\frac{\mathrm{d}u^p}{\mathrm{d}t}
+\frac{\mathrm{d}^2u^l}{\mathrm{d}t^2}\frac{\mathrm{d}u^k}{\mathrm{d}t}\frac{\mathrm{d}u^p}{\mathrm{d}t} \notag\\& -
\sum_{i,j}\Gamma_{ij}^p\frac{\mathrm{d}u^i}{\mathrm{d}t}\frac{\mathrm{d}u^j}{\mathrm{d}t}\frac{\mathrm{d}u^k}{\mathrm{d}t}\frac{\mathrm{d}u^l}{\mathrm{d}t}
-\frac{\mathrm{d}^2u^p}{\mathrm{d}t^2}\frac{\mathrm{d}u^k}{\mathrm{d}t}\frac{\mathrm{d}u^l}{\mathrm{d}t})g_{kl} \notag\\
&\times\sum_{k,l}(\sum_{i,j}\Gamma_{ij}^l\frac{\mathrm{d}u^i}{\mathrm{d}t}\frac{\mathrm{d}u^j}{\mathrm{d}t}\frac{\mathrm{d}u^k}{\mathrm{d}t}\frac{\mathrm{d}u^q}{\mathrm{d}t}
+\frac{\mathrm{d}^2u^l}{\mathrm{d}t^2}\frac{\mathrm{d}u^k}{\mathrm{d}t}\frac{\mathrm{d}u^q}{\mathrm{d}t}\notag\\& -
\sum_{i,j}\Gamma_{ij}^q\frac{\mathrm{d}u^i}{\mathrm{d}t}\frac{\mathrm{d}u^j}{\mathrm{d}t}\frac{\mathrm{d}u^k}{\mathrm{d}t}\frac{\mathrm{d}u^l}{\mathrm{d}t}
-\frac{\mathrm{d}^2u^q}{\mathrm{d}t^2}\frac{\mathrm{d}u^k}{\mathrm{d}t}\frac{\mathrm{d}u^l}{\mathrm{d}t})g_{kl} g_{pq}]^{-\frac{1}{2}},\ \ \ \Delta t\rightarrow 0^+. \notag\end{align}
    \item[$\mathrm{2)}$]  The rotational speed of $\overrightarrow{PQ}_{\mathcal{B}}$ with respect to $\Delta t$ is
\begin{align}
&|\chi_1(\Delta t)\overrightarrow{r}_1+\chi_3(\Delta t)\overrightarrow{n}|^{-3}\notag\\
&\times|\chi_3(\Delta t)\chi_3'(\Delta t)\chi_1(\Delta t)\overrightarrow{r}_1 -\chi_3^2(\Delta t)\chi_1'(\Delta t)\overrightarrow{r}_1\notag\\
&\ \ \ +\chi_1(\Delta t)\chi_1'(\Delta t)\chi_3(\Delta t)g_{11}\overrightarrow{n}-\chi_1^2(\Delta t)\chi_3'(\Delta t)g_{11}\overrightarrow{n}|
,\ \ \  \Delta t>0;  \notag
 \end{align}
 \begin{align}
&\frac{1}{2}|\frac{\mathrm{d}u^1}{\mathrm{d}t}\overrightarrow{r}_1|^{-1}
  | \sum_{i,j}L_{ij}\frac{\mathrm{d}u^i}{\mathrm{d}t}\frac{\mathrm{d}u^j}{\mathrm{d}t}|,\ \ \ \Delta t\rightarrow 0^+.  \notag\end{align}
        \item[$\mathrm{3)}$]  The rotational speed of $\overrightarrow{PQ}_{\mathcal{C}}$ with respect to $\Delta t$ is
\begin{align}
&|\chi_2(\Delta t)\overrightarrow{r}_2+\chi_3(\Delta t)\overrightarrow{n}|^{-3}\notag\\
&\times| \chi_3(\Delta t)\chi_3'(\Delta t)\chi_2(\Delta t)\overrightarrow{r}_2-\chi_3^2(\Delta t)\chi_2'(\Delta t)\overrightarrow{r}_2\notag\\
&\ \ \ +\chi_2(\Delta t)\chi_2'(\Delta t)\chi_3(\Delta t)g_{22}\overrightarrow{n}-\chi_2^2(\Delta t)\chi_3'(\Delta t)g_{22}\overrightarrow{n}|
,\ \ \  \Delta t>0;  \notag
 \end{align}
 \begin{align}
&\frac{1}{2}|\frac{\mathrm{d}u^2}{\mathrm{d}t}\overrightarrow{r}_2|^{-1}
|\sum_{i,j}L_{ij}\frac{\mathrm{d}u^i}{\mathrm{d}t}\frac{\mathrm{d}u^j}{\mathrm{d}t}|,\ \ \ \Delta t\rightarrow 0^+.  \notag\end{align}
\end{enumerate}
\end{proposition}
\begin{proof}

By (\ref{xxsecsur13}), we can write
$$\overrightarrow{e}_{\overrightarrow{PQ}_{\mathcal{A}}} = \frac{\chi_1(\Delta t)\overrightarrow{r}_1+\chi_2(\Delta t)\overrightarrow{r}_2}{\sqrt{(\chi_1(\Delta t)\overrightarrow{r}_1+\chi_2(\Delta t)\overrightarrow{r}_2)^2 }},$$
so
\begin{align}\frac{\mathrm{d}\overrightarrow{e}_{\overrightarrow{PQ}_{\mathcal{A}}} }{\mathrm{d}\Delta t}
 =&-[(\chi_1(\Delta t)\overrightarrow{r}_1+\chi_2(\Delta t)\overrightarrow{r}_2)^2]^{-\frac{3}{2}}\notag\\
&\times(\chi_1^2(\Delta t)\chi_2'(\Delta t)g_{12}\overrightarrow{r}_1 +\chi_2(\Delta t)\chi_2'(\Delta t)\chi_1(\Delta t)g_{22}\overrightarrow{r}_1\notag\\
&\ \ \  -\chi_1(\Delta t)\chi_2(\Delta t)\chi_1'(\Delta t)g_{12}\overrightarrow{r}_1-\chi_2^2(\Delta t)\chi_1'(\Delta t)g_{22}\overrightarrow{r}_1\notag\\
&\ \ \ +\chi_1(\Delta t)\chi_1'(\Delta t)\chi_2(\Delta t)g_{11}\overrightarrow{r}_2+\chi_2^2(\Delta t)\chi_1'(\Delta t)g_{12}\overrightarrow{r}_2  \notag\\
&\ \ \ -\chi_1^2(\Delta t)\chi_2'(\Delta t)g_{11}\overrightarrow{r}_2-\chi_1(\Delta t)\chi_2(\Delta t)\chi_2'(\Delta t)g_{12}\overrightarrow{r}_2)
,\notag\end{align}
where $\Delta t>0$\,(cf. (\ref{xxsecspa10})), and so the first part of 1) can be obtained.

Next let us prove the case where $\Delta t\rightarrow0^+$. Formula (\ref{xxsecsur7}) can be rewritten as (cf. (\ref{xxsecspa11}))
\begin{align}&\overrightarrow{r}((u(t+\Delta t),v(t+\Delta t)))-\overrightarrow{r}(u(t),v(t))\notag\\=&(\Delta t+\frac{1}{6}\varepsilon_1\Delta t^3)\overrightarrow{r}'(u(t),v(t))+(\frac{1}{2}\Delta t^2+\frac{1}{6}\varepsilon_2\Delta t^3)\overrightarrow{r}''(u(t),v(t))\notag\\&+(\frac{1}{6}\Delta t^3+\frac{1}{6}\varepsilon_3\Delta t^3)\overrightarrow{r}'''(u(t),v(t)),\notag
\end{align}
where $\lim_{\Delta t\rightarrow 0^+} \{\varepsilon_1,\varepsilon_2,\varepsilon_3\}=\overrightarrow{0}$.
Taking into account (\ref{xxsecsur10})-(\ref{xxsecsur12}), we obtain
\begin{align}&\overrightarrow{r}((u(t+\Delta t),v(t+\Delta t)))-\overrightarrow{r}(u(t),v(t))\notag\\
=&(\Delta t+\frac{1}{6}\varepsilon_1\Delta t^3)\sum_k\frac{\mathrm{d}u^k}{\mathrm{d}t}\overrightarrow{r}_k
+(\frac{1}{2}\Delta t^2+\frac{1}{6}\varepsilon_2\Delta t^3)\sum_{i,j,k}\Gamma_{ij}^k\frac{\mathrm{d}u^i}{\mathrm{d}t}\frac{\mathrm{d}u^j}{\mathrm{d}t}\overrightarrow{r}_k
\notag\\
&+(\frac{1}{2}\Delta t^2+\frac{1}{6}\varepsilon_2\Delta t^3)\sum_k\frac{\mathrm{d}^2u^k}{\mathrm{d}t^2}\overrightarrow{r}_k
+(\frac{1}{6}\Delta t^3+\frac{1}{6}\varepsilon_3\Delta t^3)\sum_{i,j,k,l}(\Gamma_{ij}^k)_l\frac{\mathrm{d}u^i}{\mathrm{d}t}\frac{\mathrm{d}u^j}{\mathrm{d}t}\frac{\mathrm{d}u^l}{\mathrm{d}t}\overrightarrow{r}_k
\notag\\&+2(\frac{1}{6}\Delta t^3+\frac{1}{6}\varepsilon_3\Delta t^3)\sum_{i,j,k}\Gamma_{ij}^k\frac{\mathrm{d}^2u^i}{\mathrm{d}t^2}\frac{\mathrm{d}u^j}{\mathrm{d}t}\overrightarrow{r}_k
+(\frac{1}{6}\Delta t^3+\frac{1}{6}\varepsilon_3\Delta t^3)\sum_{i,j,k}\Gamma_{ij}^k\frac{\mathrm{d}u^i}{\mathrm{d}t}\frac{\mathrm{d}^2u^j}{\mathrm{d}t^2}\overrightarrow{r}_k\notag\\
&+(\frac{1}{6}\Delta t^3+\frac{1}{6}\varepsilon_3\Delta t^3)\sum_{i,j,k,l,m}\Gamma_{ij}^m\Gamma_{ml}^k\frac{\mathrm{d}u^i}{\mathrm{d}t}\frac{\mathrm{d}u^j}{\mathrm{d}t}\frac{\mathrm{d}u^l}{\mathrm{d}t}\overrightarrow{r}_{k}
\notag\\&-(\frac{1}{6}\Delta t^3+\frac{1}{6}\varepsilon_3\Delta t^3)\sum_{i,j,k,l,m}L_{ij}L_{lm}g^{mk}\frac{\mathrm{d}u^i}{\mathrm{d}t}\frac{\mathrm{d}u^j}{\mathrm{d}t}
\frac{\mathrm{d}u^l}{\mathrm{d}t}\overrightarrow{r}_{k}\notag\\
&
+(\frac{1}{6}\Delta t^3+\frac{1}{6}\varepsilon_3\Delta t^3)\sum_k\frac{\mathrm{d}^3u^k}{\mathrm{d}t^3}\overrightarrow{r}_k
+(\frac{1}{2}\Delta t^2+\frac{1}{6}\varepsilon_2\Delta t^3)\sum_{i,j}L_{ij}\frac{\mathrm{d}u^i}{\mathrm{d}t}\frac{\mathrm{d}u^j}{\mathrm{d}t}\overrightarrow{n}\notag\\
&+(\frac{1}{6}\Delta t^3+\frac{1}{6}\varepsilon_3\Delta t^3)\sum_{i,j,k,l}\Gamma_{ij}^kL_{kl}\frac{\mathrm{d}u^i}{\mathrm{d}t}\frac{\mathrm{d}u^j}{\mathrm{d}t}\frac{\mathrm{d}u^l}{\mathrm{d}t}\overrightarrow{n}
\notag\\&+(\frac{1}{6}\Delta t^3+\frac{1}{6}\varepsilon_3\Delta t^3)\sum_{i,j,k}(L_{ij})_k\frac{\mathrm{d}u^i}{\mathrm{d}t}\frac{\mathrm{d}u^j}{\mathrm{d}t}\frac{\mathrm{d}u^k}{\mathrm{d}t}\overrightarrow{n}\notag\\
&+2(\frac{1}{6}\Delta t^3+\frac{1}{6}\varepsilon_3\Delta t^3)\sum_{i,j}L_{ij}\frac{\mathrm{d}^2u^i}{\mathrm{d}t^2}\frac{\mathrm{d}u^j}{\mathrm{d}t}\overrightarrow{n}
+(\frac{1}{6}\Delta t^3+\frac{1}{6}\varepsilon_3\Delta t^3)\sum_{i,j}L_{ij}\frac{\mathrm{d}u^i}{\mathrm{d}t}\frac{\mathrm{d}^2u^j}{\mathrm{d}t^2}\overrightarrow{n},
\label{xxsecsur14}
\end{align}
where $\lim_{\Delta t\rightarrow 0^+} \{\varepsilon_1,\varepsilon_2,\varepsilon_3\}=\overrightarrow{0}$.

From (\ref{xxsecsur14}) we obtain
\begin{align}\overrightarrow{e}_{\overrightarrow{PQ}_{\mathcal{A}}}
& = [(1+\frac{1}{6}\varepsilon_1\Delta t^2)\sum_k\frac{\mathrm{d}u^k}{\mathrm{d}t}\overrightarrow{r}_k
+(\frac{1}{2}\Delta t+\frac{1}{6}\varepsilon_2\Delta t^2)\sum_{i,j,k}\Gamma_{ij}^k\frac{\mathrm{d}u^i}{\mathrm{d}t}\frac{\mathrm{d}u^j}{\mathrm{d}t}\overrightarrow{r}_k
\notag\\
&+(\frac{1}{2}\Delta t+\frac{1}{6}\varepsilon_2\Delta t^2)\sum_k\frac{\mathrm{d}^2u^k}{\mathrm{d}t^2}\overrightarrow{r}_k
+(\frac{1}{6}\Delta t^2+\frac{1}{6}\varepsilon_3\Delta t^2)\sum_{i,j,k,l}(\Gamma_{ij}^k)_l\frac{\mathrm{d}u^i}{\mathrm{d}t}\frac{\mathrm{d}u^j}{\mathrm{d}t}\frac{\mathrm{d}u^l}{\mathrm{d}t}\overrightarrow{r}_k
\notag\\&+2(\frac{1}{6}\Delta t^2+\frac{1}{6}\varepsilon_3\Delta t^2)\sum_{i,j,k}\Gamma_{ij}^k\frac{\mathrm{d}^2u^i}{\mathrm{d}t^2}\frac{\mathrm{d}u^j}{\mathrm{d}t}\overrightarrow{r}_k
+(\frac{1}{6}\Delta t^2+\frac{1}{6}\varepsilon_3\Delta t^2)\sum_{i,j,k}\Gamma_{ij}^k\frac{\mathrm{d}u^i}{\mathrm{d}t}\frac{\mathrm{d}^2u^j}{\mathrm{d}t^2}\overrightarrow{r}_k\notag\\
&+(\frac{1}{6}\Delta t^2+\frac{1}{6}\varepsilon_3\Delta t^2)\sum_{i,j,k,l,m}\Gamma_{ij}^m\Gamma_{ml}^k\frac{\mathrm{d}u^i}{\mathrm{d}t}\frac{\mathrm{d}u^j}{\mathrm{d}t}\frac{\mathrm{d}u^l}{\mathrm{d}t}\overrightarrow{r}_{k}
\notag\\&-(\frac{1}{6}\Delta t^2+\frac{1}{6}\varepsilon_3\Delta t^2)\sum_{i,j,k,l,m}L_{ij}L_{lm}g^{mk}\frac{\mathrm{d}u^i}{\mathrm{d}t}\frac{\mathrm{d}u^j}{\mathrm{d}t}
\frac{\mathrm{d}u^l}{\mathrm{d}t}\overrightarrow{r}_{k}\notag\\
&
+(\frac{1}{6}\Delta t^2+\frac{1}{6}\varepsilon_3\Delta t^2)\sum_k\frac{\mathrm{d}^3u^k}{\mathrm{d}t^3}\overrightarrow{r}_k]
\notag\\
&\times |(1+\frac{1}{6}\varepsilon_1\Delta t^2)\sum_k\frac{\mathrm{d}u^k}{\mathrm{d}t}\overrightarrow{r}_k
+(\frac{1}{2}\Delta t+\frac{1}{6}\varepsilon_2\Delta t^2)\sum_{i,j,k}\Gamma_{ij}^k\frac{\mathrm{d}u^i}{\mathrm{d}t}\frac{\mathrm{d}u^j}{\mathrm{d}t}\overrightarrow{r}_k
\notag\\
&+(\frac{1}{2}\Delta t+\frac{1}{6}\varepsilon_2\Delta t^2)\sum_k\frac{\mathrm{d}^2u^k}{\mathrm{d}t^2}\overrightarrow{r}_k
+(\frac{1}{6}\Delta t^2+\frac{1}{6}\varepsilon_3\Delta t^2)\sum_{i,j,k,l}(\Gamma_{ij}^k)_l\frac{\mathrm{d}u^i}{\mathrm{d}t}\frac{\mathrm{d}u^j}{\mathrm{d}t}\frac{\mathrm{d}u^l}{\mathrm{d}t}\overrightarrow{r}_k
\notag\\&+2(\frac{1}{6}\Delta t^2+\frac{1}{6}\varepsilon_3\Delta t^2)\sum_{i,j,k}\Gamma_{ij}^k\frac{\mathrm{d}^2u^i}{\mathrm{d}t^2}\frac{\mathrm{d}u^j}{\mathrm{d}t}\overrightarrow{r}_k
+(\frac{1}{6}\Delta t^2+\frac{1}{6}\varepsilon_3\Delta t^2)\sum_{i,j,k}\Gamma_{ij}^k\frac{\mathrm{d}u^i}{\mathrm{d}t}\frac{\mathrm{d}^2u^j}{\mathrm{d}t^2}\overrightarrow{r}_k\notag\\
&+(\frac{1}{6}\Delta t^2+\frac{1}{6}\varepsilon_3\Delta t^2)\sum_{i,j,k,l,m}\Gamma_{ij}^m\Gamma_{ml}^k\frac{\mathrm{d}u^i}{\mathrm{d}t}\frac{\mathrm{d}u^j}{\mathrm{d}t}\frac{\mathrm{d}u^l}{\mathrm{d}t}\overrightarrow{r}_{k}
\notag\\&-(\frac{1}{6}\Delta t^2+\frac{1}{6}\varepsilon_3\Delta t^2)\sum_{i,j,k,l,m}L_{ij}L_{lm}g^{mk}\frac{\mathrm{d}u^i}{\mathrm{d}t}\frac{\mathrm{d}u^j}{\mathrm{d}t}
\frac{\mathrm{d}u^l}{\mathrm{d}t}\overrightarrow{r}_{k}\notag\\
&
+(\frac{1}{6}\Delta t^2+\frac{1}{6}\varepsilon_3\Delta t^2)\sum_k\frac{\mathrm{d}^3u^k}{\mathrm{d}t^3}\overrightarrow{r}_k|^{-1}
.\notag\end{align}
 One can compute that
\begin{align}&\lim_{\Delta t\rightarrow 0^+}\frac{\overrightarrow{e}_{\overrightarrow{PQ}_{\mathcal{A}}}(\Delta t)-\overrightarrow{e}_{\overrightarrow{PQ}_{\mathcal{A}}}(0)}{\Delta t}\notag\\
=&-\frac{1}{2}|\sum_k\frac{\mathrm{d}u^k}{\mathrm{d}t}\overrightarrow{r}_k|^{-3}\notag\\&\times[ (
\sum_{i,j,k,l,m}\Gamma_{ij}^l\frac{\mathrm{d}u^i}{\mathrm{d}t}\frac{\mathrm{d}u^j}{\mathrm{d}t}\frac{\mathrm{d}u^k}{\mathrm{d}t}\frac{\mathrm{d}u^m}{\mathrm{d}t}g_{kl}\overrightarrow{r}_m
+\sum_{k,l,m}\frac{\mathrm{d}^2u^l}{\mathrm{d}t^2}\frac{\mathrm{d}u^k}{\mathrm{d}t}\frac{\mathrm{d}u^m}{\mathrm{d}t}g_{kl}\overrightarrow{r}_m) \notag\\& -(
\sum_{i,j,k,l,m}\Gamma_{ij}^m\frac{\mathrm{d}u^i}{\mathrm{d}t}\frac{\mathrm{d}u^j}{\mathrm{d}t}\frac{\mathrm{d}u^k}{\mathrm{d}t}\frac{\mathrm{d}u^l}{\mathrm{d}t}g_{kl}\overrightarrow{r}_m
+\sum_{k,l,m}\frac{\mathrm{d}^2u^m}{\mathrm{d}t^2}\frac{\mathrm{d}u^k}{\mathrm{d}t}\frac{\mathrm{d}u^l}{\mathrm{d}t}g_{kl}\overrightarrow{r}_m
)]
.\label{xxsecsur15} \end{align}
For brevity, we place the detailed calculation of (\ref{xxsecsur15}) in Appendix A. Now we have
 \begin{align}&|\lim_{\Delta t\rightarrow 0^+}\frac{\overrightarrow{e}_{\overrightarrow{PQ}_{\mathcal{A}}}(\Delta t)-\overrightarrow{e}_{\overrightarrow{PQ}_{\mathcal{A}}}(0)}{\Delta t}|\notag\\
=&\frac{1}{2}|\sum_k\frac{\mathrm{d}u^k}{\mathrm{d}t}\overrightarrow{r}_k|^{-3}
|\sum_{m}(\sum_{i,j,k,l}\Gamma_{ij}^l\frac{\mathrm{d}u^i}{\mathrm{d}t}\frac{\mathrm{d}u^j}{\mathrm{d}t}\frac{\mathrm{d}u^k}{\mathrm{d}t}\frac{\mathrm{d}u^m}{\mathrm{d}t}
g_{kl}
+\sum_{k,l}\frac{\mathrm{d}^2u^l}{\mathrm{d}t^2}\frac{\mathrm{d}u^k}{\mathrm{d}t}\frac{\mathrm{d}u^m}{\mathrm{d}t}g_{kl}\notag\\& -
\sum_{i,j,k,l}\Gamma_{ij}^m\frac{\mathrm{d}u^i}{\mathrm{d}t}\frac{\mathrm{d}u^j}{\mathrm{d}t}\frac{\mathrm{d}u^k}{\mathrm{d}t}\frac{\mathrm{d}u^l}{\mathrm{d}t}g_{kl}
-\sum_{k,l}\frac{\mathrm{d}^2u^m}{\mathrm{d}t^2}\frac{\mathrm{d}u^k}{\mathrm{d}t}\frac{\mathrm{d}u^l}{\mathrm{d}t}g_{kl})  \overrightarrow{r}_m|\notag\\
=&\frac{1}{2}|\sum_k\frac{\mathrm{d}u^k}{\mathrm{d}t}\overrightarrow{r}_k|^{-3}\notag\\
&\times\{[\sum_{p}(\sum_{i,j,k,l}\Gamma_{ij}^l\frac{\mathrm{d}u^i}{\mathrm{d}t}\frac{\mathrm{d}u^j}{\mathrm{d}t}\frac{\mathrm{d}u^k}{\mathrm{d}t}\frac{\mathrm{d}u^p}{\mathrm{d}t}g_{kl}
+\sum_{k,l}\frac{\mathrm{d}^2u^l}{\mathrm{d}t^2}\frac{\mathrm{d}u^k}{\mathrm{d}t}\frac{\mathrm{d}u^p}{\mathrm{d}t}g_{kl}\notag\\& -
\sum_{i,j,k,l}\Gamma_{ij}^p\frac{\mathrm{d}u^i}{\mathrm{d}t}\frac{\mathrm{d}u^j}{\mathrm{d}t}\frac{\mathrm{d}u^k}{\mathrm{d}t}\frac{\mathrm{d}u^l}{\mathrm{d}t}g_{kl}
-\sum_{k,l}\frac{\mathrm{d}^2u^p}{\mathrm{d}t^2}\frac{\mathrm{d}u^k}{\mathrm{d}t}\frac{\mathrm{d}u^l}{\mathrm{d}t}g_{kl}) \overrightarrow{r}_p]\notag\\
&\cdot[\sum_{q}(\sum_{i,j,k,l}\Gamma_{ij}^l\frac{\mathrm{d}u^i}{\mathrm{d}t}\frac{\mathrm{d}u^j}{\mathrm{d}t}\frac{\mathrm{d}u^k}{\mathrm{d}t}\frac{\mathrm{d}u^q}{\mathrm{d}t}g_{kl}
+\sum_{k,l}\frac{\mathrm{d}^2u^l}{\mathrm{d}t^2}\frac{\mathrm{d}u^k}{\mathrm{d}t}\frac{\mathrm{d}u^q}{\mathrm{d}t}g_{kl}\notag\\& -
\sum_{i,j,k,l}\Gamma_{ij}^q\frac{\mathrm{d}u^i}{\mathrm{d}t}\frac{\mathrm{d}u^j}{\mathrm{d}t}\frac{\mathrm{d}u^k}{\mathrm{d}t}\frac{\mathrm{d}u^l}{\mathrm{d}t}g_{kl}
-\sum_{k,l}\frac{\mathrm{d}^2u^q}{\mathrm{d}t^2}\frac{\mathrm{d}u^k}{\mathrm{d}t}\frac{\mathrm{d}u^l}{\mathrm{d}t}g_{kl}) \overrightarrow{r}_q]\}^{\frac{1}{2}}\notag\\
=&\frac{1}{2}|\sum_k\frac{\mathrm{d}u^k}{\mathrm{d}t}\overrightarrow{r}_k|^{-3}\notag\\&\times
[\sum_{p,q}(\sum_{i,j,k,l}\Gamma_{ij}^l\frac{\mathrm{d}u^i}{\mathrm{d}t}\frac{\mathrm{d}u^j}{\mathrm{d}t}\frac{\mathrm{d}u^k}{\mathrm{d}t}\frac{\mathrm{d}u^p}{\mathrm{d}t}g_{kl}
+\sum_{k,l}\frac{\mathrm{d}^2u^l}{\mathrm{d}t^2}\frac{\mathrm{d}u^k}{\mathrm{d}t}\frac{\mathrm{d}u^p}{\mathrm{d}t}g_{kl} \notag\\& -
\sum_{i,j,k,l}\Gamma_{ij}^p\frac{\mathrm{d}u^i}{\mathrm{d}t}\frac{\mathrm{d}u^j}{\mathrm{d}t}\frac{\mathrm{d}u^k}{\mathrm{d}t}\frac{\mathrm{d}u^l}{\mathrm{d}t}g_{kl}
-\sum_{k,l}\frac{\mathrm{d}^2u^p}{\mathrm{d}t^2}\frac{\mathrm{d}u^k}{\mathrm{d}t}\frac{\mathrm{d}u^l}{\mathrm{d}t}g_{kl}) \notag\\
&\times(\sum_{i,j,k,l}\Gamma_{ij}^l\frac{\mathrm{d}u^i}{\mathrm{d}t}\frac{\mathrm{d}u^j}{\mathrm{d}t}\frac{\mathrm{d}u^k}{\mathrm{d}t}\frac{\mathrm{d}u^q}{\mathrm{d}t}g_{kl}
+\sum_{k,l}\frac{\mathrm{d}^2u^l}{\mathrm{d}t^2}\frac{\mathrm{d}u^k}{\mathrm{d}t}\frac{\mathrm{d}u^q}{\mathrm{d}t}g_{kl}\notag\\& -
\sum_{i,j,k,l}\Gamma_{ij}^q\frac{\mathrm{d}u^i}{\mathrm{d}t}\frac{\mathrm{d}u^j}{\mathrm{d}t}\frac{\mathrm{d}u^k}{\mathrm{d}t}\frac{\mathrm{d}u^l}{\mathrm{d}t}g_{kl}
-\sum_{k,l}\frac{\mathrm{d}^2u^q}{\mathrm{d}t^2}\frac{\mathrm{d}u^k}{\mathrm{d}t}\frac{\mathrm{d}u^l}{\mathrm{d}t}g_{kl}) g_{pq}]^{-\frac{1}{2}}\notag\\
=&\frac{1}{2}|\sum_k\frac{\mathrm{d}u^k}{\mathrm{d}t}\overrightarrow{r}_k|^{-3}\notag\\&\times
[\sum_{p,q}\sum_{k,l}(\sum_{i,j}\Gamma_{ij}^l\frac{\mathrm{d}u^i}{\mathrm{d}t}\frac{\mathrm{d}u^j}{\mathrm{d}t}\frac{\mathrm{d}u^k}{\mathrm{d}t}\frac{\mathrm{d}u^p}{\mathrm{d}t}
+\frac{\mathrm{d}^2u^l}{\mathrm{d}t^2}\frac{\mathrm{d}u^k}{\mathrm{d}t}\frac{\mathrm{d}u^p}{\mathrm{d}t} \notag\\& -
\sum_{i,j}\Gamma_{ij}^p\frac{\mathrm{d}u^i}{\mathrm{d}t}\frac{\mathrm{d}u^j}{\mathrm{d}t}\frac{\mathrm{d}u^k}{\mathrm{d}t}\frac{\mathrm{d}u^l}{\mathrm{d}t}
-\frac{\mathrm{d}^2u^p}{\mathrm{d}t^2}\frac{\mathrm{d}u^k}{\mathrm{d}t}\frac{\mathrm{d}u^l}{\mathrm{d}t})g_{kl} \notag\\
&\times\sum_{k,l}(\sum_{i,j}\Gamma_{ij}^l\frac{\mathrm{d}u^i}{\mathrm{d}t}\frac{\mathrm{d}u^j}{\mathrm{d}t}\frac{\mathrm{d}u^k}{\mathrm{d}t}\frac{\mathrm{d}u^q}{\mathrm{d}t}
+\frac{\mathrm{d}^2u^l}{\mathrm{d}t^2}\frac{\mathrm{d}u^k}{\mathrm{d}t}\frac{\mathrm{d}u^q}{\mathrm{d}t}\notag\\& -
\sum_{i,j}\Gamma_{ij}^q\frac{\mathrm{d}u^i}{\mathrm{d}t}\frac{\mathrm{d}u^j}{\mathrm{d}t}\frac{\mathrm{d}u^k}{\mathrm{d}t}\frac{\mathrm{d}u^l}{\mathrm{d}t}
-\frac{\mathrm{d}^2u^q}{\mathrm{d}t^2}\frac{\mathrm{d}u^k}{\mathrm{d}t}\frac{\mathrm{d}u^l}{\mathrm{d}t})g_{kl} g_{pq}]^{-\frac{1}{2}}
.\notag
\end{align}
2) and 3)  The proof for $\Delta t>0$ is similar to that of 1).  So we only consider the case where $\Delta t\rightarrow 0^+$. Below is a simple process:
\begin{align}\overrightarrow{e}_{\overrightarrow{PQ}_{\mathcal{B}}}& =[(1+\frac{1}{6}\varepsilon_1\Delta t^2)\frac{\mathrm{d}u^1}{\mathrm{d}t}\overrightarrow{r}_1
+(\frac{1}{2}\Delta t+\frac{1}{6}\varepsilon_2\Delta t^2)\sum_{i,j}\Gamma_{ij}^1\frac{\mathrm{d}u^i}{\mathrm{d}t}\frac{\mathrm{d}u^j}{\mathrm{d}t}\overrightarrow{r}_1
\notag\\
&+(\frac{1}{2}\Delta t+\frac{1}{6}\varepsilon_2\Delta t^2)\frac{\mathrm{d}^2u^1}{\mathrm{d}t^2}\overrightarrow{r}_1
+(\frac{1}{6}\Delta t^2+\frac{1}{6}\varepsilon_3\Delta t^2)\sum_{i,j,l}(\Gamma_{ij}^1)_l\frac{\mathrm{d}u^i}{\mathrm{d}t}\frac{\mathrm{d}u^j}{\mathrm{d}t}\frac{\mathrm{d}u^l}{\mathrm{d}t}\overrightarrow{r}_1
\notag\\&+2(\frac{1}{6}\Delta t^2+\frac{1}{6}\varepsilon_3\Delta t^2)\sum_{i,j}\Gamma_{ij}^1\frac{\mathrm{d}^2u^i}{\mathrm{d}t^2}\frac{\mathrm{d}u^j}{\mathrm{d}t}\overrightarrow{r}_1
+(\frac{1}{6}\Delta t^2+\frac{1}{6}\varepsilon_3\Delta t^2)\sum_{i,j}\Gamma_{ij}^1\frac{\mathrm{d}u^i}{\mathrm{d}t}\frac{\mathrm{d}^2u^j}{\mathrm{d}t^2}\overrightarrow{r}_1\notag\\
&+(\frac{1}{6}\Delta t^2+\frac{1}{6}\varepsilon_3\Delta t^2)\sum_{i,j,l,m}\Gamma_{ij}^m\Gamma_{ml}^1\frac{\mathrm{d}u^i}{\mathrm{d}t}\frac{\mathrm{d}u^j}{\mathrm{d}t}\frac{\mathrm{d}u^l}{\mathrm{d}t}\overrightarrow{r}_{1}
\notag\\&-(\frac{1}{6}\Delta t^2+\frac{1}{6}\varepsilon_3\Delta t^2)\sum_{i,j,l,m}L_{ij}L_{lm}g^{m1}\frac{\mathrm{d}u^i}{\mathrm{d}t}\frac{\mathrm{d}u^j}{\mathrm{d}t}
\frac{\mathrm{d}u^l}{\mathrm{d}t}\overrightarrow{r}_{1}\notag\\
&
+(\frac{1}{6}\Delta t^2+\frac{1}{6}\varepsilon_3\Delta t^2)\frac{\mathrm{d}^3u^1}{\mathrm{d}t^3}\overrightarrow{r}_1
+(\frac{1}{2}\Delta t+\frac{1}{6}\varepsilon_2\Delta t^2)\sum_{i,j}L_{ij}\frac{\mathrm{d}u^i}{\mathrm{d}t}\frac{\mathrm{d}u^j}{\mathrm{d}t}\overrightarrow{n}\notag\\
&+(\frac{1}{6}\Delta t^2+\frac{1}{6}\varepsilon_3\Delta t^2)\sum_{i,j,k,l}\Gamma_{ij}^kL_{kl}\frac{\mathrm{d}u^i}{\mathrm{d}t}\frac{\mathrm{d}u^j}{\mathrm{d}t}\frac{\mathrm{d}u^l}{\mathrm{d}t}\overrightarrow{n}
\notag\\&+(\frac{1}{6}\Delta t^2+\frac{1}{6}\varepsilon_3\Delta t^2)\sum_{i,j,k}(L_{ij})_k\frac{\mathrm{d}u^i}{\mathrm{d}t}\frac{\mathrm{d}u^j}{\mathrm{d}t}\frac{\mathrm{d}u^k}{\mathrm{d}t}\overrightarrow{n}\notag\\
&+2(\frac{1}{6}\Delta t^2+\frac{1}{6}\varepsilon_3\Delta t^2)\sum_{i,j}L_{ij}\frac{\mathrm{d}^2u^i}{\mathrm{d}t^2}\frac{\mathrm{d}u^j}{\mathrm{d}t}\overrightarrow{n}
+(\frac{1}{6}\Delta t^2+\frac{1}{6}\varepsilon_3\Delta t^2)\sum_{i,j}L_{ij}\frac{\mathrm{d}u^i}{\mathrm{d}t}\frac{\mathrm{d}^2u^j}{\mathrm{d}t^2}\overrightarrow{n}]\notag\\
&\times
|(1+\frac{1}{6}\varepsilon_1\Delta t^2)\frac{\mathrm{d}u^1}{\mathrm{d}t}\overrightarrow{r}_1
+(\frac{1}{2}\Delta t+\frac{1}{6}\varepsilon_2\Delta t^2)\sum_{i,j}\Gamma_{ij}^1\frac{\mathrm{d}u^i}{\mathrm{d}t}\frac{\mathrm{d}u^j}{\mathrm{d}t}\overrightarrow{r}_1
\notag\\
&+(\frac{1}{2}\Delta t+\frac{1}{6}\varepsilon_2\Delta t^2)\frac{\mathrm{d}^2u^1}{\mathrm{d}t^2}\overrightarrow{r}_1
+(\frac{1}{6}\Delta t^2+\frac{1}{6}\varepsilon_3\Delta t^2)\sum_{i,j,l}(\Gamma_{ij}^1)_l\frac{\mathrm{d}u^i}{\mathrm{d}t}\frac{\mathrm{d}u^j}{\mathrm{d}t}\frac{\mathrm{d}u^l}{\mathrm{d}t}\overrightarrow{r}_1
\notag\\&+2(\frac{1}{6}\Delta t^2+\frac{1}{6}\varepsilon_3\Delta t^2)\sum_{i,j}\Gamma_{ij}^1\frac{\mathrm{d}^2u^i}{\mathrm{d}t^2}\frac{\mathrm{d}u^j}{\mathrm{d}t}\overrightarrow{r}_1
+(\frac{1}{6}\Delta t^2+\frac{1}{6}\varepsilon_3\Delta t^2)\sum_{i,j}\Gamma_{ij}^1\frac{\mathrm{d}u^i}{\mathrm{d}t}\frac{\mathrm{d}^2u^j}{\mathrm{d}t^2}\overrightarrow{r}_1\notag\\
&+(\frac{1}{6}\Delta t^2+\frac{1}{6}\varepsilon_3\Delta t^2)\sum_{i,j,l,m}\Gamma_{ij}^m\Gamma_{ml}^1\frac{\mathrm{d}u^i}{\mathrm{d}t}\frac{\mathrm{d}u^j}{\mathrm{d}t}\frac{\mathrm{d}u^l}{\mathrm{d}t}\overrightarrow{r}_{1}
\notag\\&-(\frac{1}{6}\Delta t^2+\frac{1}{6}\varepsilon_3\Delta t^2)\sum_{i,j,l,m}L_{ij}L_{lm}g^{m1}\frac{\mathrm{d}u^i}{\mathrm{d}t}\frac{\mathrm{d}u^j}{\mathrm{d}t}
\frac{\mathrm{d}u^l}{\mathrm{d}t}\overrightarrow{r}_{1}\notag\\
&
+(\frac{1}{6}\Delta t^2+\frac{1}{6}\varepsilon_3\Delta t^2)\frac{\mathrm{d}^3u^1}{\mathrm{d}t^3}\overrightarrow{r}_1
+(\frac{1}{2}\Delta t+\frac{1}{6}\varepsilon_2\Delta t^2)\sum_{i,j}L_{ij}\frac{\mathrm{d}u^i}{\mathrm{d}t}\frac{\mathrm{d}u^j}{\mathrm{d}t}\overrightarrow{n}\notag\\
&+(\frac{1}{6}\Delta t^2+\frac{1}{6}\varepsilon_3\Delta t^2)\sum_{i,j,k,l}\Gamma_{ij}^kL_{kl}\frac{\mathrm{d}u^i}{\mathrm{d}t}\frac{\mathrm{d}u^j}{\mathrm{d}t}\frac{\mathrm{d}u^l}{\mathrm{d}t}\overrightarrow{n}
\notag\\&+(\frac{1}{6}\Delta t^2+\frac{1}{6}\varepsilon_3\Delta t^2)\sum_{i,j,k}(L_{ij})_k\frac{\mathrm{d}u^i}{\mathrm{d}t}\frac{\mathrm{d}u^j}{\mathrm{d}t}\frac{\mathrm{d}u^k}{\mathrm{d}t}\overrightarrow{n}\notag\\
&+2(\frac{1}{6}\Delta t^2+\frac{1}{6}\varepsilon_3\Delta t^2)\sum_{i,j}L_{ij}\frac{\mathrm{d}^2u^i}{\mathrm{d}t^2}\frac{\mathrm{d}u^j}{\mathrm{d}t}\overrightarrow{n}
+(\frac{1}{6}\Delta t^2+\frac{1}{6}\varepsilon_3\Delta t^2)\sum_{i,j}L_{ij}\frac{\mathrm{d}u^i}{\mathrm{d}t}\frac{\mathrm{d}^2u^j}{\mathrm{d}t^2}\overrightarrow{n}|^{-1},\notag\end{align}
\begin{align}&\lim_{\Delta t\rightarrow 0^+}\frac{\overrightarrow{e}_{\overrightarrow{PQ}_{\mathcal{B}}}(\Delta t)-\overrightarrow{e}_{\overrightarrow{PQ}_{\mathcal{B}}}(0)}{\Delta t}
\notag\\=&\frac{1}{2}|\frac{\mathrm{d}u^1}{\mathrm{d}t}\overrightarrow{r}_1|^{-1}
  ( \sum_{i,j}L_{ij}\frac{\mathrm{d}u^i}{\mathrm{d}t}\frac{\mathrm{d}u^j}{\mathrm{d}t}\overrightarrow{n})\ (\text{see Appendix A})
,\label{xxsecsur16} \end{align}
\begin{align}&|\lim_{\Delta t\rightarrow 0^+}\frac{\overrightarrow{e}_{\overrightarrow{PQ}_{\mathcal{B}}}(\Delta t)-\overrightarrow{e}_{\overrightarrow{PQ}_{\mathcal{B}}}(0)}{\Delta t}
|=\frac{1}{2}|\frac{\mathrm{d}u^1}{\mathrm{d}t}\overrightarrow{r}_1|^{-1}
  | \sum_{i,j}L_{ij}\frac{\mathrm{d}u^i}{\mathrm{d}t}\frac{\mathrm{d}u^j}{\mathrm{d}t}|
;\notag \end{align}
\begin{align}\overrightarrow{e}_{\overrightarrow{PQ}_{\mathcal{C}}}& = [(1+\frac{1}{6}\varepsilon_1\Delta t^2)\frac{\mathrm{d}u^2}{\mathrm{d}t}\overrightarrow{r}_2
+(\frac{1}{2}\Delta t+\frac{1}{6}\varepsilon_2\Delta t^2)\sum_{i,j}\Gamma_{ij}^2\frac{\mathrm{d}u^i}{\mathrm{d}t}\frac{\mathrm{d}u^j}{\mathrm{d}t}\overrightarrow{r}_2
\notag\\
&+(\frac{1}{2}\Delta t+\frac{1}{6}\varepsilon_2\Delta t^2)\frac{\mathrm{d}^2u^2}{\mathrm{d}t^2}\overrightarrow{r}_2
+(\frac{1}{6}\Delta t^2+\frac{1}{6}\varepsilon_3\Delta t^2)\sum_{i,j,l}(\Gamma_{ij}^2)_l\frac{\mathrm{d}u^i}{\mathrm{d}t}\frac{\mathrm{d}u^j}{\mathrm{d}t}\frac{\mathrm{d}u^l}{\mathrm{d}t}\overrightarrow{r}_2
\notag\\&+2(\frac{1}{6}\Delta t^2+\frac{1}{6}\varepsilon_3\Delta t^2)\sum_{i,j}\Gamma_{ij}^2\frac{\mathrm{d}^2u^i}{\mathrm{d}t^2}\frac{\mathrm{d}u^j}{\mathrm{d}t}\overrightarrow{r}_2
+(\frac{1}{6}\Delta t^2+\frac{1}{6}\varepsilon_3\Delta t^2)\sum_{i,j}\Gamma_{ij}^2\frac{\mathrm{d}u^i}{\mathrm{d}t}\frac{\mathrm{d}^2u^j}{\mathrm{d}t^2}\overrightarrow{r}_2\notag\\
&+(\frac{1}{6}\Delta t^2+\frac{1}{6}\varepsilon_3\Delta t^2)\sum_{i,j,l,m}\Gamma_{ij}^m\Gamma_{ml}^2\frac{\mathrm{d}u^i}{\mathrm{d}t}\frac{\mathrm{d}u^j}{\mathrm{d}t}\frac{\mathrm{d}u^l}{\mathrm{d}t}\overrightarrow{r}_{2}
\notag\\&-(\frac{1}{6}\Delta t^2+\frac{1}{6}\varepsilon_3\Delta t^2)\sum_{i,j,l,m}L_{ij}L_{lm}g^{m2}\frac{\mathrm{d}u^i}{\mathrm{d}t}\frac{\mathrm{d}u^j}{\mathrm{d}t}
\frac{\mathrm{d}u^l}{\mathrm{d}t}\overrightarrow{r}_{2}\notag\\
&
+(\frac{1}{6}\Delta t^2+\frac{1}{6}\varepsilon_3\Delta t^2)\frac{\mathrm{d}^3u^2}{\mathrm{d}t^3}\overrightarrow{r}_2
+(\frac{1}{2}\Delta t+\frac{1}{6}\varepsilon_2\Delta t^2)\sum_{i,j}L_{ij}\frac{\mathrm{d}u^i}{\mathrm{d}t}\frac{\mathrm{d}u^j}{\mathrm{d}t}\overrightarrow{n}\notag\\
&+(\frac{1}{6}\Delta t^2+\frac{1}{6}\varepsilon_3\Delta t^2)\sum_{i,j,k,l}\Gamma_{ij}^kL_{kl}\frac{\mathrm{d}u^i}{\mathrm{d}t}\frac{\mathrm{d}u^j}{\mathrm{d}t}\frac{\mathrm{d}u^l}{\mathrm{d}t}\overrightarrow{n}
\notag\\&+(\frac{1}{6}\Delta t^2+\frac{1}{6}\varepsilon_3\Delta t^2)\sum_{i,j,k}(L_{ij})_k\frac{\mathrm{d}u^i}{\mathrm{d}t}\frac{\mathrm{d}u^j}{\mathrm{d}t}\frac{\mathrm{d}u^k}{\mathrm{d}t}\overrightarrow{n}\notag\\
&+2(\frac{1}{6}\Delta t^2+\frac{1}{6}\varepsilon_3\Delta t^2)\sum_{i,j}L_{ij}\frac{\mathrm{d}^2u^i}{\mathrm{d}t^2}\frac{\mathrm{d}u^j}{\mathrm{d}t}\overrightarrow{n}
+(\frac{1}{6}\Delta t^2+\frac{1}{6}\varepsilon_3\Delta t^2)\sum_{i,j}L_{ij}\frac{\mathrm{d}u^i}{\mathrm{d}t}\frac{\mathrm{d}^2u^j}{\mathrm{d}t^2}\overrightarrow{n}] \notag\\
&\times|(1+\frac{1}{6}\varepsilon_1\Delta t^2)\frac{\mathrm{d}u^2}{\mathrm{d}t}\overrightarrow{r}_2
+(\frac{1}{2}\Delta t+\frac{1}{6}\varepsilon_2\Delta t^2)\sum_{i,j}\Gamma_{ij}^2\frac{\mathrm{d}u^i}{\mathrm{d}t}\frac{\mathrm{d}u^j}{\mathrm{d}t}\overrightarrow{r}_2
\notag\\
&+(\frac{1}{2}\Delta t+\frac{1}{6}\varepsilon_2\Delta t^2)\frac{\mathrm{d}^2u^2}{\mathrm{d}t^2}\overrightarrow{r}_2
+(\frac{1}{6}\Delta t^2+\frac{1}{6}\varepsilon_3\Delta t^2)\sum_{i,j,l}(\Gamma_{ij}^2)_l\frac{\mathrm{d}u^i}{\mathrm{d}t}\frac{\mathrm{d}u^j}{\mathrm{d}t}\frac{\mathrm{d}u^l}{\mathrm{d}t}\overrightarrow{r}_2
\notag\\&+2(\frac{1}{6}\Delta t^2+\frac{1}{6}\varepsilon_3\Delta t^2)\sum_{i,j}\Gamma_{ij}^2\frac{\mathrm{d}^2u^i}{\mathrm{d}t^2}\frac{\mathrm{d}u^j}{\mathrm{d}t}\overrightarrow{r}_2
+(\frac{1}{6}\Delta t^2+\frac{1}{6}\varepsilon_3\Delta t^2)\sum_{i,j}\Gamma_{ij}^2\frac{\mathrm{d}u^i}{\mathrm{d}t}\frac{\mathrm{d}^2u^j}{\mathrm{d}t^2}\overrightarrow{r}_2\notag\\
&+(\frac{1}{6}\Delta t^2+\frac{1}{6}\varepsilon_3\Delta t^2)\sum_{i,j,l,m}\Gamma_{ij}^m\Gamma_{ml}^2\frac{\mathrm{d}u^i}{\mathrm{d}t}\frac{\mathrm{d}u^j}{\mathrm{d}t}\frac{\mathrm{d}u^l}{\mathrm{d}t}\overrightarrow{r}_{2}
\notag\\&-(\frac{1}{6}\Delta t^2+\frac{1}{6}\varepsilon_3\Delta t^2)\sum_{i,j,l,m}L_{ij}L_{lm}g^{m2}\frac{\mathrm{d}u^i}{\mathrm{d}t}\frac{\mathrm{d}u^j}{\mathrm{d}t}
\frac{\mathrm{d}u^l}{\mathrm{d}t}\overrightarrow{r}_{2}\notag\\
&
+(\frac{1}{6}\Delta t^2+\frac{1}{6}\varepsilon_3\Delta t^2)\frac{\mathrm{d}^3u^2}{\mathrm{d}t^3}\overrightarrow{r}_2
+(\frac{1}{2}\Delta t+\frac{1}{6}\varepsilon_2\Delta t^2)\sum_{i,j}L_{ij}\frac{\mathrm{d}u^i}{\mathrm{d}t}\frac{\mathrm{d}u^j}{\mathrm{d}t}\overrightarrow{n}\notag\\
&+(\frac{1}{6}\Delta t^2+\frac{1}{6}\varepsilon_3\Delta t^2)\sum_{i,j,k,l}\Gamma_{ij}^kL_{kl}\frac{\mathrm{d}u^i}{\mathrm{d}t}\frac{\mathrm{d}u^j}{\mathrm{d}t}\frac{\mathrm{d}u^l}{\mathrm{d}t}\overrightarrow{n}
\notag\\&+(\frac{1}{6}\Delta t^2+\frac{1}{6}\varepsilon_3\Delta t^2)\sum_{i,j,k}(L_{ij})_k\frac{\mathrm{d}u^i}{\mathrm{d}t}\frac{\mathrm{d}u^j}{\mathrm{d}t}\frac{\mathrm{d}u^k}{\mathrm{d}t}\overrightarrow{n}\notag\\
&+2(\frac{1}{6}\Delta t^2+\frac{1}{6}\varepsilon_3\Delta t^2)\sum_{i,j}L_{ij}\frac{\mathrm{d}^2u^i}{\mathrm{d}t^2}\frac{\mathrm{d}u^j}{\mathrm{d}t}\overrightarrow{n}
+(\frac{1}{6}\Delta t^2+\frac{1}{6}\varepsilon_3\Delta t^2)\sum_{i,j}L_{ij}\frac{\mathrm{d}u^i}{\mathrm{d}t}\frac{\mathrm{d}^2u^j}{\mathrm{d}t^2}\overrightarrow{n}|^{-1} ,\notag\end{align}
\begin{align}&\lim_{\Delta t\rightarrow 0^+}\frac{\overrightarrow{e}_{\overrightarrow{PQ}_{\mathcal{C}}}(\Delta t)-\overrightarrow{e}_{\overrightarrow{PQ}_{\mathcal{C}}}(0)}{\Delta t}\notag\\
 =&\frac{1}{2}|\frac{\mathrm{d}u^2}{\mathrm{d}t}\overrightarrow{r}_2|^{-1}
(\sum_{i,j}L_{ij}\frac{\mathrm{d}u^i}{\mathrm{d}t}\frac{\mathrm{d}u^j}{\mathrm{d}t}\overrightarrow{n} )\ (\text{see Appendix A})
,\label{xxsecsur17} \end{align}
\begin{align}|\lim_{\Delta t\rightarrow 0^+}\frac{\overrightarrow{e}_{\overrightarrow{PQ}_{\mathcal{C}}}(\Delta t)-\overrightarrow{e}_{\overrightarrow{PQ}_{\mathcal{C}}}(0)}{\Delta t} |
 =\frac{1}{2}|\frac{\mathrm{d}u^2}{\mathrm{d}t}\overrightarrow{r}_2|^{-1}
|\sum_{i,j}L_{ij}\frac{\mathrm{d}u^i}{\mathrm{d}t}\frac{\mathrm{d}u^j}{\mathrm{d}t}|
.\notag\end{align}

\end{proof}

\section{Application: The behaviour of moving points on ellipses}
In this section, we will use the method of rotating frames to discuss the  behaviour of moving points on ellipses in detail.
We mainly use frame (\ref{xxsecpla2}) and  (\ref{xxsecpla9}). Let us start with (\ref{xxsecpla2}).

\subsection{Ellipses with general rotating frames 1}

Let the vector function of an ellipse be given by
\begin{equation}\label{xxsecapl1}
  \vec{r}(\theta)=\{a\cos \theta, b\sin \theta \}, \ \ \ 0\leq\theta\leq2\pi,
\end{equation}
where $a>b>0$. Frame (\ref{xxsecpla2}) now becomes
\begin{equation}\label{xxsecapl2}
 \{O; \vec{e}_1(\theta), \vec{e}_2(\theta) \}, \ \ \ 0\leq\theta\leq2\pi,
\end{equation}
where $\vec{e}_1(\theta)=\frac{\{a\cos\theta, b\sin\theta\}}{\sqrt{a^2\cos^2\theta+b^2\sin^2\theta}}$, $\vec{e}_2(\theta)=\frac{\{ -b\sin\theta,a\cos\theta\}}{\sqrt{a^2\cos^2\theta+b^2\sin^2\theta}}$.  Let the coordinates of (\ref{xxsecapl2}) be $\xi,\eta$.
Then the equation of (\ref{xxsecapl1}) under (\ref{xxsecapl2}) is
\begin{equation}\label{xxsecapl3}
  \xi(\theta)= \sqrt{a^2\cos^2\theta+b^2\sin^2\theta }, \  \eta(\theta)=0,\ \ \ 0\leq\theta\leq2\pi.
\end{equation}

\textbf{Note:} Unless otherwise specified, the range of $\theta$ is $[0,2\pi]$, though it may also be interpreted as $(-\infty, +\infty)$.

It is precisely due to the establishment of (\ref{xxsecapl2}), we can observe two kinds of underlying motions in the behaviour of an arbitrary moving point on (\ref{xxsecapl1}), namely vibration and rotation. The details are as follows.
\begin{proposition}\label{xxsec5.1} Let $P$ be a moving point on (\ref{xxsecapl1}), and let its parameter be $\theta$.
 \begin{enumerate}
  \item[$\mathrm{(i)}$] For $\mathcal{D}=|\overrightarrow{OP}|$, we have
  \begin{align}
\frac{ \rm{d}^2\mathcal{D}}{\rm{d}\theta^2}
   =&\frac{ c^2(-a^2\cos^4\theta +b^2\sin^4\theta)}{(a^2\cos^2\theta+b^2\sin^2\theta)^{\frac{3}{2}}} , \notag\\
  \frac{ \rm{d}\mathcal{D}}{\rm{d}\theta}|_{\theta=0}&=0,
   \mathcal{D}|_{\theta=0}=a.\notag
\end{align}
\item[$\mathrm{(ii)}$] The rotational speed of $\overrightarrow{OP}$ with respect to $\theta$
 is $$\frac{ab}{a^2\cos^2\theta+b^2\sin^2\theta}.$$
\end{enumerate}
\end{proposition}
\begin{proof} (i)   Noting that $\mathcal{D}=\sqrt{a^2\cos^2\theta+b^2\sin^2\theta }$\,(equals to $\xi(\theta)$ in (\ref{xxsecapl3})), we find that
\begin{align}
 \frac{ \rm{d}\mathcal{D}}{\rm{d}\theta}&= \frac{-2a^2\cos\theta\sin\theta+2b^2\sin\theta\cos\theta}{2\sqrt{a^2\cos^2\theta+b^2\sin^2\theta }}\notag\\
 &=  \frac{  -c^2\sin\theta\cos\theta }{ \sqrt{a^2\cos^2\theta+b^2\sin^2\theta }} ,\notag
\end{align}
\begin{align}
\frac{ \rm{d}^2\mathcal{D}}{\rm{d}\theta^2}=& -\frac{1}{2}(a^2\cos^2\theta+b^2\sin^2\theta)^{-\frac{3}{2}}(  -c^2\sin\theta\cos\theta)\times 2(  -c^2\sin\theta\cos\theta)\notag\\
  &+(a^2\cos^2\theta+b^2\sin^2\theta)^{-\frac{1}{2}}( -c^2\cos^2\theta  +c^2\sin^2\theta) \notag\\
   =& - (a^2\cos^2\theta+b^2\sin^2\theta)^{-\frac{3}{2}}(   c^2\sin \theta\cos\theta )^2\notag\\
  &+(a^2\cos^2\theta+b^2\sin^2\theta)^{-\frac{1}{2}}(  -c^2\cos^2\theta  +c^2\sin^2\theta ) \notag\\
  =& c^2(a^2\cos^2\theta+b^2\sin^2\theta)^{-\frac{3}{2}}[-   c^2\sin ^2\theta\cos^2\theta
  +(a^2\cos^2\theta+b^2\sin^2\theta)(  -\cos^2\theta  +\sin^2\theta )]  \notag\\
   =& c^2(a^2\cos^2\theta+b^2\sin^2\theta)^{-\frac{3}{2}}(-   c^2\sin ^2\theta\cos^2\theta
\notag\\&\ \ \  -a^2\cos^4\theta+a^2\cos^2\theta\sin^2\theta-b^2\sin^2\theta\cos^2\theta +b^2\sin^4\theta) \notag\\
  =&\frac{c^2(-a^2\cos^4\theta +b^2\sin^4\theta)}{(a^2\cos^2\theta+b^2\sin^2\theta)^{\frac{3}{2}}}.\notag
\end{align}
The values of $\frac{ \rm{d}\mathcal{D}}{\rm{d}\theta}$ and $ \mathcal{D}$ at $\theta=0$ obviously hold.

(ii) Note that \begin{align}\vec{e}_{\overrightarrow{OP}}=\frac{\{a\cos\theta, b\sin\theta\}}{\sqrt{a^2\cos^2\theta+b^2\sin^2\theta}},\notag\end{align}
so
\begin{align}\frac{\mathrm{d}\vec{e}_{\overrightarrow{OP}}}{\mathrm{d}\theta}
 =&-\frac{1}{2}(a^2\cos^2\theta+b^2\sin^2\theta)^{-\frac{3}{2}}(-2a^2\cos\theta\sin\theta+2b^2\sin\theta \cos\theta)\{a\cos\theta, b\sin\theta\}\notag\\&+(a^2\cos^2\theta+b^2\sin^2\theta)^{-\frac{1}{2}}\{-a\sin\theta, b\cos\theta\}\notag\\
=&(a^2\cos^2\theta+b^2\sin^2\theta)^{-\frac{3}{2}}(a^2\cos\theta\sin\theta-b^2\sin\theta \cos\theta)\{a\cos\theta, b\sin\theta\}\notag\\
 &+(a^2\cos^2\theta+b^2\sin^2\theta)^{-\frac{3}{2}}(a^2\cos^2\theta+b^2\sin^2\theta)\{-a\sin\theta, b\cos\theta\}\notag\\
= & (a^2\cos^2\theta+b^2\sin^2\theta)^{-\frac{3}{2}}\notag\\
 &\ \ \ \times \{a^3\cos^2\theta\sin\theta-ab^2\sin\theta \cos^2\theta, a^2b\cos\theta\sin^2\theta-b^3\sin^2\theta \cos\theta\}\notag\\
&+ (a^2\cos^2\theta+b^2\sin^2\theta)^{-\frac{3}{2}}\notag\\
 &\ \ \ \times\{-a^3\cos^2\theta\sin\theta-ab^2\sin^3\theta, a^2 b\cos^3\theta+b^3\sin^2\theta\cos\theta\} \notag\\
=&\frac{ \{- ab^2\sin\theta \cos^2\theta- ab^2\sin^3\theta, a^2b\cos\theta\sin^2\theta +a^2 b\cos^3\theta\}}{(a^2\cos^2\theta+b^2\sin^2\theta)^{\frac{3}{2}}}\notag\\
=&\frac{ \{ -ab^2\sin\theta ,  a^2b\cos\theta \}}{(a^2\cos^2\theta+b^2\sin^2\theta)^{\frac{3}{2}}},\notag\\
 =&\frac{ab\{ -b\sin\theta ,  a\cos\theta \}}{ (a^2\cos^2\theta+b^2\sin^2\theta)^{\frac{3}{2}}},\notag\end{align}
 \begin{align}|\frac{\mathrm{d}\vec{e}_{\overrightarrow{OP}}}{\mathrm{d}\theta}|= \frac{ab}{a^2\cos^2\theta+b^2\sin^2\theta}.\notag\end{align}
This proves the assertion (ii).
 \end{proof}

We can summarize the conclusion of Proposition \ref{xxsec5.1} as follows:
\begin{enumerate}
  \item[$\textcircled{1}$]  the second-order derivative of $|\overrightarrow{OP}|$ with respect to  $\theta$ has an average value of $0$ over $[0,2\pi]$, and it takes the value $0$ at $\theta=\arctan\sqrt{\frac{a}{b}}, \pi\pm \arctan\sqrt{\frac{a}{b}}, 2\pi-\arctan\sqrt{\frac{a}{b}}$;
\item[$\textcircled{2}$]  the average rotational speed of $\overrightarrow{OP}$ with respect to $\theta$ is 1 in  any quadrant.
\end{enumerate}

We can now state the following unique existence theorem originating from frame (\ref{xxsecapl2}).
\begin{theorem}\label{xxsec5.2} Let $O$ be the origin of $\mathbb{R}^2$, and let $a>c>0, b=\sqrt{a^2-c^2}$. Suppose that $P$ is a moving point in $\mathbb{R}^2$ satisfying the following conditions:
\begin{enumerate}
  \item[$\mathrm{1)}$] for $\mathcal{D}=|\overrightarrow{OP}|$, we have
  \begin{align}
&\frac{ \rm{d}^2\mathcal{D}}{\rm{d}\theta^2}
   =\frac{c^2(-a^2\cos^4\theta +b^2\sin^4\theta)}{(a^2\cos^2\theta+b^2\sin^2\theta)^{\frac{3}{2}}},  \notag\\
&\frac{ \rm{d}\mathcal{D}}{\rm{d}\theta}|_{\theta=0}=0, \mathcal{D}|_{\theta=0}=a; \notag
\end{align}
\item[$\mathrm{2)}$]
  $\frac{\mathrm{d}\vec{e}_{\overrightarrow{OP}}}{\mathrm{d}\theta}=\frac{ab\{ -b\sin\theta ,  a\cos\theta \}}{ (a^2\cos^2\theta+b^2\sin^2\theta)^{\frac{3}{2}}},
 \vec{e}_{\overrightarrow{OP}}|_{\theta=0}=\{1,0\}
 $.
\end{enumerate}
Here $0\leq\theta\leq2\pi$. Then the trajectory of $P$ is (\ref{xxsecapl1}).
\end{theorem}
\begin{proof} Assume point $P$ is at $(x,y)$. We can set $x=a\lambda(\theta)\cos\theta$ and $y=b\lambda(\theta)\sin\theta$ by 2), where $\lambda(\theta)> 0 $, $0\leq\theta\leq2\pi$. It follows from  1) that
 \begin{equation}\label{xxsecapl4}
x^2+y^2 =   a^2\cos^2\theta^2+b^2\sin^2\theta  , \ \ \ 0\leq\theta\leq2\pi.
\end{equation}
Taking $x=a\lambda(\theta)\cos\theta$, $y=b\lambda(\theta)\sin\theta$ into (\ref{xxsecapl4}) yields
$$ \lambda^2(\theta)(a^2 \cos^2 \theta+b^2\sin^2\theta) = a^2 \cos^2 \theta+b^2\sin^2\theta  , \ \ \ 0\leq\theta\leq2\pi.$$
Thus $\lambda^2= 1$, and we must have $\lambda=1$ by the hypothesis. This completes the proof.
\end{proof}

\subsection{Ellipses with general rotating frames 2}

The origin of frame (\ref{xxsecapl2}) is  $O(0,0)$. For an ellipse, the foci have special significance. In order to further investigate the behaviour of
 moving points on ellipses, it is necessary to introduce another rotating frame with the focus of (\ref{xxsecapl1}) as its origin.

Let us first choose  point $A$ in (\ref{xxsecpla9}) to be the focus $(c,0)$ and rename it by $A_1$. Then (\ref{xxsecpla9}) becomes
\begin{equation}\label{xxsecapl5}
 \{A_1; \vec{e}_1(\theta), \vec{e}_2(\theta) \}, \ \ \ 0\leq\theta\leq2\pi,
\end{equation}
where $\vec{e}_1(\theta)=\frac{\{a\cos\theta-c, b\sin\theta\}}{\sqrt{(a\cos\theta-c)^2+b^2\sin^2\theta}}$, $\vec{e}_2(\theta)=\frac{\{ -b\sin\theta,a\cos\theta-c\}}{\sqrt{(a\cos\theta-c)^2+b^2\sin^2\theta}}$. If the coordinates (\ref{xxsecapl5}) are $\xi_1$, $\eta_1$, then the equation of (\ref{xxsecapl1}) under (\ref{xxsecapl5}) becomes
\begin{equation}\label{xxsecapl6}
  \xi_1(\theta)= \sqrt{(a\cos\theta-c)^2+b^2\sin^2\theta },\   \eta_1(\theta)=0,\ \ \ 0\leq\theta\leq2\pi.
\end{equation}

Let us discuss $\xi_1(\theta)$  firstly in detail.
\begin{proposition}\label{xxsec5.3}  Let $\xi_1(\theta)$ be as in (\ref{xxsecapl6}).
\begin{enumerate}
  \item[$\mathrm{(1)}$] $\frac{ \rm{d}\xi_1(\theta)}{\rm{d}\theta}>0$, $\theta\in (0,\pi)$; $\frac{ \rm{d}\xi_1(\theta)}{\rm{d}\theta}<0$, $\theta\in (\pi, 2\pi)$.
\item[$\mathrm{(2)}$] $\frac{\rm{d}^2\xi_1(\theta)}{\rm{d}\theta^2}>0$, $\theta\in (0,\frac{\pi}{2})\cup (\frac{3\pi}{2},2\pi) $; $\frac{\rm{d}^2\xi_1(\theta)}{\rm{d}\theta^2}<0$, $\theta\in (\frac{\pi}{2},\pi)\cup (\pi,\frac{3\pi}{2}) $.
\item[$\mathrm{(3)}$]     $\frac{\rm{d}^3\xi_1(\theta)}{\rm{d}\theta^3}<0$, $\theta\in(0,\pi)$; $\frac{\rm{d}^3\xi_1(\theta)}{\rm{d}\theta^3}>0$, $\theta\in(\pi,2\pi)$.
\end{enumerate}
\end{proposition}
\begin{proof} (1) This follows immediately from the following calculation:
\begin{align}
 \frac{ \rm{d}\xi_1}{\rm{d}\theta}&= \frac{2(a\cos\theta-c)(-a\sin\theta)+2b^2\sin\theta\cos\theta}{2\sqrt{(a\cos\theta-c)^2+b^2\sin^2\theta }}\notag\\
 &= \frac{ -2a^2\sin\theta\cos\theta+2ac\sin\theta + 2b^2\sin\theta\cos\theta }{2\sqrt{(a\cos\theta-c)^2+b^2\sin^2\theta }}\notag\\
&=  \frac{ ac\sin\theta -c^2\sin\theta\cos\theta }{ \sqrt{(a\cos\theta-c)^2+b^2\sin^2\theta }}. \label{xxsecapl7}
\end{align}

(2) Taking the derivative of (\ref{xxsecapl7}), we obtain
\begin{align}
\frac{ \rm{d}^2\xi_1}{\rm{d}\theta^2}=& -\frac{1}{2}[(a\cos\theta-c)^2+b^2\sin^2\theta ]^{-\frac{3}{2}}( ac\sin\theta -c^2\sin\theta\cos\theta)\times 2( ac\sin\theta -c^2\sin\theta\cos\theta)\notag\\
  &+[(a\cos\theta-c)^2+b^2\sin^2\theta ]^{-\frac{1}{2}}(ac\cos\theta-c^2\cos^2\theta  +c^2\sin^2\theta) \notag\\
   =& - [(a\cos\theta-c)^2+b^2\sin^2\theta ]^{-\frac{3}{2}}(   ac\sin\theta -c^2\sin \theta\cos\theta )^2\notag\\
  &+[(a\cos\theta-c)^2+b^2\sin^2\theta ]^{-\frac{1}{2}}(   ac\cos\theta -c^2\cos2\theta ). \label{xxsecapl8}
\end{align}
Let $\tilde{\xi}_1(\theta)= \sqrt{(a\cos\theta+c)^2+b^2\sin^2\theta }$. Then we have
 \begin{align}
 \frac{ \rm{d}\tilde{\xi}_1}{\rm{d}\theta}&= \frac{2(a\cos\theta+c)(-a\sin\theta)+2b^2\sin\theta\cos\theta}{2\sqrt{(a\cos\theta+c)^2+b^2\sin^2\theta }}\notag\\
 &= \frac{ -2a^2\sin\theta\cos\theta-2ac\sin\theta + 2b^2\sin\theta\cos\theta }{2\sqrt{(a\cos\theta+c)^2+b^2\sin^2\theta }}
 \notag\\
 &= \frac{-ac\sin\theta -c^2\sin\theta\cos\theta }{\sqrt{(a\cos\theta+c)^2+b^2\sin^2\theta }}
, \notag
\end{align}
\begin{align}
\frac{ \rm{d}^2\tilde{\xi}_1}{\rm{d}\theta^2}=& -\frac{1}{2}[(a\cos\theta+c)^2+b^2\sin^2\theta ]^{-\frac{3}{2}}(-ac\sin\theta -c^2\sin\theta\cos\theta )\times 2(-ac\sin\theta -c^2\sin\theta\cos\theta )\notag\\
  &+[(a\cos\theta+c)^2+b^2\sin^2\theta ]^{-\frac{1}{2}}(  -ac\cos\theta-c^2\cos^2\theta+c^2\sin^2\theta ) \notag\\
   =& - [(a\cos\theta+c)^2+b^2\sin^2\theta ]^{-\frac{3}{2}}(   ac\sin\theta +c^2\sin \theta\cos\theta )^2\notag\\
  &+[(a\cos\theta+c)^2+b^2\sin^2\theta ]^{-\frac{1}{2}}(   -ac\cos\theta -c^2\cos2\theta ).  \label{xxsecapl9}
\end{align}
One should note that $\frac{ \rm{d}^2\xi_1}{\rm{d}\theta^2}+\frac{ \rm{d}^2\tilde{\xi}_1}{\rm{d}\theta^2}=0$.  If $\frac{ \rm{d}^2\xi_1}{\rm{d}\theta^2}$ is nonpositive in the first quadrant,
then we get from (\ref{xxsecapl8}) and (\ref{xxsecapl9}) that
\begin{align}
& - (   ac\sin\theta -c^2\sin \theta\cos\theta )^2
  +[(a\cos\theta-c)^2+b^2\sin^2\theta ](   ac\cos\theta -c^2\cos2\theta )\leq0,\notag\\
 &- (   ac\sin\theta +c^2\sin \theta\cos\theta )^2
 +[(a\cos\theta+c)^2+b^2\sin^2\theta ](  - ac\cos\theta -c^2\cos2\theta )\geq0, \notag
\end{align}
where $\theta\in(0,\frac{\pi}{2})$. Subtracting the first inequality by the second, we obtain
$$4 ac^3\sin^2\theta \cos\theta+2ac\cos\theta ( a^2\cos^2\theta+c^2+b^2\sin^2\theta ) +4 ac^3\cos2\theta\cos\theta\leq0,$$
that is,
$$ 2  c^2\cos^2\theta  +   a^2\cos^2\theta+c^2+b^2\sin^2\theta   \leq0,$$
where $ \theta\in(0,\frac{\pi}{2})$. A contradiction.
Hence $\frac{ \rm{d}^2\xi_1}{\rm{d}\theta^2}$ is positive in $(0,\frac{\pi}{2})$. One needs only to consider $\frac{ \rm{d}^2\xi_1}{\rm{d}\theta^2}|_{\pi-\theta}$, $\frac{ \rm{d}^2\xi_1}{\rm{d}\theta^2}|_{\pi+\theta}$, and $\frac{ \rm{d}^2\xi_1}{\rm{d}\theta^2}|_{2\pi-\theta}$ for the other three cases.

(3) See Appendix \ref{Appendix B}.
\end{proof}

From Proposition \ref{xxsec5.3}, we have the following Table 5.1.
\begin{table}[h]
\centering
\begin{tabular}{c|c|c|c|c|c|c|c|c|c}
   \hline
       & 0& $(0,\frac{\pi}{2})$ & $\frac{\pi}{2} $ &  $(\frac{\pi}{2} , \pi)$  & $\pi$ & $(\frac{\pi}{2} ,\frac{3\pi}{2})$ & $\frac{3\pi}{2} $ &  $(\frac{3\pi}{2} , 2\pi)$  & $2\pi$   \\
    \hline
$\xi_1$ &$a-c$& $\nearrow$ & $a$& $\nearrow$& a+c& $\searrow$ & $a$& $\searrow$& a-c \\
    \hline
$\frac{ \rm{d}\xi_1}{\rm{d}\theta}$ &0 &$+$ $\nearrow$ & $c$ &$+$ $\searrow$  & $0$  & $-$  $\searrow$   & $-c$ & $-$ $\nearrow$&  0 \\
    \hline
$\frac{ \rm{d}^2\xi_1}{\rm{d}\theta^2} $&  $c$& $+ $ $\searrow$&  0 &$-$  $\searrow$  & $ -c$  &$-$  $\nearrow$  &0 & $+ $  $\nearrow$    &    c\\
    \hline
    $\frac{ \rm{d}^3\xi_1}{\rm{d}\theta^3} $&  & $- $&    &$-$  & $  $  &$+$  &  & $+ $  &    \\
    \hline
 \end{tabular}
 \vspace{3mm}

 \textbf{Table 5.1}
\end{table}

With the help of (\ref{xxsecapl2}), we have known that the behaviour of any moving point $P$ on (\ref{xxsecapl1}) can be seen as the composition of vibration and rotation.  By utilizing (\ref{xxsecapl5}), we can obtain a similar result as follows.
\begin{proposition}\label{xxsec5.4} Let $P$ be a moving point on (\ref{xxsecapl1}), and let its parameter be $\theta$.
 \begin{enumerate}
  \item[$\mathrm{(i)}$] For $\mathcal{D}=|\overrightarrow{A_1P}|$, we have
  \begin{align}
\frac{ \rm{d}^2\mathcal{D}}{\rm{d}\theta^2}
   =&\frac{-(   ac\sin\theta -c^2\sin \theta\cos\theta )^2+[(a\cos\theta-c)^2+b^2\sin^2\theta ](   ac\cos\theta -c^2\cos2\theta )}{[(a\cos\theta-c)^2+b^2\sin^2\theta ]^{\frac{3}{2}}} , \notag\\
  \frac{ \rm{d}\mathcal{D}}{\rm{d}\theta}|_{\theta=0}&=0,
   \mathcal{D}_{\theta=0}=a-c.\notag
\end{align}
\item[$\mathrm{(ii)}$] The rotational speed of $\overrightarrow{A_1P}$ with respect to $\theta$
 is $$\frac{b(a-c\cos\theta)}{(a\cos\theta-c)^2+b^2\sin^2\theta }.$$
\end{enumerate}
\end{proposition}
\begin{proof} (i)  Note that $\mathcal{D}=\sqrt{(a\cos\theta-c)^2+b^2\sin^2\theta }$\, (equals to $\xi_1(\theta)$ in (\ref{xxsecapl6})), so the results follow from  Proposition \ref{xxsec5.3} and its proof, as well as Table 5.1.

(ii) We begin by noticing that $$\vec{e}_{\overrightarrow{A_1P}}=\frac{\{a\cos\theta-c, b\sin\theta\}}{\sqrt{(a\cos\theta-c)^2+b^2\sin^2\theta}}.$$
Its derivative is
\begin{align}\frac{\mathrm{d}\vec{e}_{\overrightarrow{A_1P}}}{\mathrm{d}\theta}
=&-\frac{1}{2}[ (a\cos\theta-c)^2+b^2\sin^2\theta ]^{-\frac{3}{2}}\notag\\&\times[ 2(a\cos\theta-c)(-a\sin\theta) + 2 b^2\sin\theta\cos\theta]\{a\cos\theta-c, b\sin\theta\}\notag\\&+[ (a\cos\theta-c)^2+b^2\sin^2\theta ]^{-\frac{1}{2}}\{-a\sin\theta, b\cos\theta\} \notag\\
=&[ (a\cos\theta-c)^2+b^2\sin^2\theta ]^{-\frac{3}{2}}\notag\\&\times[ (a\cos\theta-c)(a\sin\theta) - b^2\sin\theta\cos\theta]\{a\cos\theta-c, b\sin\theta\}\notag\\&+[ (a\cos\theta-c)^2+b^2\sin^2\theta ]^{-\frac{3}{2}}[ (a\cos\theta-c)^2+b^2\sin^2\theta ]\{-a\sin\theta, b\cos\theta\} \notag\\
=&[ (a\cos\theta-c)^2+b^2\sin^2\theta ]^{-\frac{3}{2}}\notag\\&\times( -ac\sin\theta+a^2\sin\theta\cos\theta-b^2\sin\theta\cos\theta)\{a\cos\theta-c, b\sin\theta\}\notag\\
&+[ (a\cos\theta-c)^2+b^2\sin^2\theta ]^{-\frac{3}{2}}\notag\\&\times(a^2\cos^2\theta+c^2+b^2\sin^2\theta-2ac\cos\theta)\{-a\sin\theta, b\cos\theta\}\notag\\
=&[ (a\cos\theta-c)^2+b^2\sin^2\theta ]^{-\frac{3}{2}}\notag\\&\times
\{-a^2c\sin\theta\cos\theta+a^3\sin\theta\cos^2\theta-ab^2\sin\theta\cos^2\theta\notag\\
&\ \ \ \ \ \ +ac^2\sin\theta-a^2c\sin\theta\cos\theta+b^2c\sin\theta\cos\theta
, \notag\\&\ \ \ \  \ \ \ \  -abc\sin^2\theta+a^2b\sin^2\theta\cos\theta-b^3\sin^2\theta\cos\theta\}\notag\\& +[ (a\cos\theta-c)^2+b^2\sin^2\theta ]^{-\frac{3}{2}}\notag\\&\times\{-a^3\sin\theta\cos^2\theta-ac^2\sin\theta-ab^2\sin^3\theta+2a^2c\sin\theta\cos\theta,\notag\\&
\ \ \  \  \ \ \ \  \ \ \ \   a^2b\cos^3\theta+bc^2\cos\theta+b^3\sin^2\theta\cos\theta-2abc\cos^2\theta \}\notag\\
=&[ (a\cos\theta-c)^2+b^2\sin^2\theta ]^{-\frac{3}{2}}\notag\\&\times
 \{-2a^2c\sin\theta\cos\theta+a^3\sin\theta\cos^2\theta-ab^2\sin\theta\cos^2\theta+ac^2\sin\theta+b^2c\sin\theta\cos\theta
, \notag\\&\ \ \ \ \ \ -abc\sin^2\theta+a^2b\sin^2\theta\cos\theta-b^3\sin^2\theta\cos\theta\}\notag\\&
 +[ (a\cos\theta-c)^2+b^2\sin^2\theta ]^{-\frac{3}{2}}\notag\\&\times\{-a^3\sin\theta\cos^2\theta-ac^2\sin\theta-ab^2\sin^3\theta+2a^2c\sin\theta\cos\theta,\notag\\&
\ \ \ \ \ \ a^2b\cos^3\theta+bc^2\cos\theta+b^3\sin^2\theta\cos\theta-2abc\cos^2\theta \} \notag\\
=&[ (a\cos\theta-c)^2+b^2\sin^2\theta ]^{-\frac{3}{2}}
\{-ab^2\sin\theta\cos^2\theta+b^2c\sin\theta\cos\theta-ab^2\sin^3\theta,
\notag\\
& -abc\sin^2\theta+a^2b\sin^2\theta\cos\theta+a^2b\cos^3\theta+bc^2\cos\theta-2abc\cos^2\theta\}\notag\\
=&[ (a\cos\theta-c)^2+b^2\sin^2\theta ]^{-\frac{3}{2}}
\notag\\
&\times\{-ab^2\sin\theta+b^2c\sin\theta\cos\theta
, -abc+a^2b\cos\theta+bc^2\cos\theta-abc\cos^2\theta\}\notag\\
=&[ (a\cos\theta-c)^2+b^2\sin^2\theta ]^{-\frac{3}{2}}\notag\\
&\times \{b^2\sin\theta(  -a  +c\cos\theta),
 ab(a \cos\theta-c )+bc\cos\theta(c-a\cos\theta)\}\notag\\
 =&[ (a\cos\theta-c)^2+b^2\sin^2\theta ]^{-\frac{3}{2}}   \{b^2\sin\theta(  -a  +c\cos\theta),
 b(a \cos\theta-c )(a-c\cos\theta)\}\notag\\
 =& \frac{b(a-c\cos\theta) \{-b\sin\theta,
  a \cos\theta-c \}}{[ (a\cos\theta-c)^2+b^2\sin^2\theta ]^{\frac{3}{2}}}.\notag\end{align}
Consequently,
\begin{align}
|\frac{\mathrm{d}\vec{e}_{\overrightarrow{A_1P}}}{\mathrm{d}\theta}|
&=\frac{b(a-c\cos\theta)}{(a\cos\theta-c)^2+b^2\sin^2\theta} .\notag
\end{align}
\end{proof}

Proposition \ref{xxsec5.4} can be concisely summarized as follows:
\begin{enumerate}
 \item[$\textcircled{1}$] the second-order derivative of $|\overrightarrow{A_1P}|$ with respect to  $\theta$ has an average value of $0$ over $[0,2\pi]$, and it  takes  the value $0$ at $\theta=\frac{\pi}{2}, \frac{3\pi}{2}$;
\item[$\textcircled{2}$] the average rotational speed of $\overrightarrow{A_1P}$ with respect to  $\theta$ is 1 in $[0, \pi]$ and $[\pi, 2\pi]$.
\end{enumerate}

The following is now almost obvious.
\begin{theorem}\label{xxsec5.5} Let $A_1(c,0)$ be a fixed point of $\mathbb{R}^2$, and let $a>c>0, b=\sqrt{a^2-c^2}$. Suppose that $P$ is a moving point in $\mathbb{R}^2$ satisfying the following conditions:
\begin{enumerate}
  \item[$\mathrm{1)}$] for $\mathcal{D}=|\overrightarrow{A_1P}|$, we have
  \begin{align}
&\frac{ \rm{d}^2\mathcal{D}}{\rm{d}\theta^2}
   =\frac{-(   ac\sin\theta -c^2\sin \theta\cos\theta )^2+[(a\cos\theta-c)^2+b^2\sin^2\theta ](   ac\cos\theta -c^2\cos2\theta )}{[(a\cos\theta-c)^2+b^2\sin^2\theta ]^{\frac{3}{2}}},  \notag\\
 &\frac{ \rm{d}\mathcal{D}}{\rm{d}\theta}|_{\theta=0}=0, \mathcal{D}|_{\theta=0}=a-c;  \notag
\end{align}
\item[$\mathrm{2)}$]
  $\frac{\mathrm{d}\vec{e}_{\overrightarrow{A_1P}}}{\mathrm{d}\theta}=\frac{ b(a-c\cos\theta)\{-b\sin\theta,
  a \cos\theta-c \}}{[ (a\cos\theta-c)^2+b^2\sin^2\theta ]^{\frac{3}{2}}}$,
 $\vec{e}_{\overrightarrow{A_1P}}|_{\theta=0}=\{1,0\}$.
\end{enumerate}
Here $0\leq\theta\leq2\pi$. Then the trajectory of $P$ is (\ref{xxsecapl1}).
\end{theorem}
\begin{proof}
The proof closely follows that of Theorem  \ref{xxsec5.2}.
\end{proof}

\textbf{Remark.}  We just discussed the behaviour of the moving points on ellipses with frame (\ref{xxsecpla2}) and  (\ref{xxsecpla9}).
In fact, we can also use the local rotating frame  (\ref{xxsecpla12}).
At this point, it can be known from  formula (\ref{xxsecpla21}) and Proposition \ref{xxsec2.6} (ii) that the local second-order derivative and rotational speed function are $\frac{c^2\sin\theta \cos\theta   }{\sqrt{a^2\sin^2\theta+b^2\cos^2\theta}}$ and
$\frac{ab}{2(a^2\sin^2\theta+ b^2\cos^2\theta)}$, respectively.
Hence, it can be said that we have obtained three different approaches to forming an ellipse.  Their commonality is that  all three have a binary property (vibration and rotation).  Finally, it should be noted that this new binary mathematical  formation mechanism of an ellipse based on vibration and rotation do not conflict with those classic mechanisms (for example, differentiate  (\ref{xxsecapl1}) directly with respect to $\theta$ twice), rather, they supplement them. All of these tell us that an ellipse can be formed through a variety of mechanisms.

\appendix{}

\section{Some complicated computations involved in $\S$ 4.2.}
\label{Appendix A}
This section is devoted to the computations of  (\ref{xxsecsur12}), (\ref{xxsecsur15})-(\ref{xxsecsur17})
in $\S$ 4.2. These subtle and complicated computations are fairly necessary for  our main result in this subsection. To make our exposition self-contained, we compile them into this appendix for the reader to follow our argument.

The calculation of (\ref{xxsecsur12}):
\begin{align}
&\overrightarrow{r}'''(u(t),v(t))\notag\\
=&
\sum_{i,j,k,l}(\Gamma_{ij}^k)_l\frac{\mathrm{d}u^l}{\mathrm{d}t}\frac{\mathrm{d}u^i}{\mathrm{d}t}\frac{\mathrm{d}u^j}{\mathrm{d}t}\overrightarrow{r}_k
+\sum_{i,j,k}\Gamma_{ij}^k\frac{\mathrm{d}^2u^i}{\mathrm{d}t^2}\frac{\mathrm{d}u^j}{\mathrm{d}t}\overrightarrow{r}_k
+\sum_{i,j,k}\Gamma_{ij}^k\frac{\mathrm{d}u^i}{\mathrm{d}t}\frac{\mathrm{d}^2u^j}{\mathrm{d}t^2}\overrightarrow{r}_k\notag\\
&+\sum_{i,j,k,l}\Gamma_{ij}^k\frac{\mathrm{d}u^i}{\mathrm{d}t}\frac{\mathrm{d}u^j}{\mathrm{d}t}\overrightarrow{r}_{kl}\frac{\mathrm{d}u^l}{\mathrm{d}t}
+\sum_{i,j,k}(L_{ij})_k\frac{\mathrm{d}u^k}{\mathrm{d}t}\frac{\mathrm{d}u^i}{\mathrm{d}t}\frac{\mathrm{d}u^j}{\mathrm{d}t}\overrightarrow{n}\notag\\
&+\sum_{i,j}L_{ij}\frac{\mathrm{d}^2u^i}{\mathrm{d}t^2}\frac{\mathrm{d}u^j}{\mathrm{d}t}\overrightarrow{n}+\sum_{i,j}L_{ij}\frac{\mathrm{d}u^i}{\mathrm{d}t}\frac{\mathrm{d}^2u^j}{\mathrm{d}t^2}\overrightarrow{n}
+\sum_{i,j,k}L_{ij}\frac{\mathrm{d}u^i}{\mathrm{d}t}\frac{\mathrm{d}u^j}{\mathrm{d}t}\overrightarrow{n}_k\frac{\mathrm{d}u^k}{\mathrm{d}t}\notag\\
&+\sum_{i,j}\frac{\mathrm{d}^2u^i}{\mathrm{d}t^2}\overrightarrow{r}_{ij}\frac{\mathrm{d}u^j}{\mathrm{d}t}+\sum_i\frac{\mathrm{d}^3u^i}{\mathrm{d}t^3}\overrightarrow{r}_i\notag\\
=&
\sum_{i,j,k,l}(\Gamma_{ij}^k)_l\frac{\mathrm{d}u^i}{\mathrm{d}t}\frac{\mathrm{d}u^j}{\mathrm{d}t}\frac{\mathrm{d}u^l}{\mathrm{d}t}\overrightarrow{r}_k
+\sum_{i,j,k}\Gamma_{ij}^k\frac{\mathrm{d}^2u^i}{\mathrm{d}t^2}\frac{\mathrm{d}u^j}{\mathrm{d}t}\overrightarrow{r}_k
+\sum_{i,j,k}\Gamma_{ij}^k\frac{\mathrm{d}u^i}{\mathrm{d}t}\frac{\mathrm{d}^2u^j}{\mathrm{d}t^2}\overrightarrow{r}_k\notag\\
&+\sum_{i,j,k,l,m}\Gamma_{ij}^k\frac{\mathrm{d}u^i}{\mathrm{d}t}\frac{\mathrm{d}u^j}{\mathrm{d}t}\Gamma_{kl}^m\overrightarrow{r}_{m}\frac{\mathrm{d}u^l}{\mathrm{d}t}
+\sum_{i,j,k,l}\Gamma_{ij}^k\frac{\mathrm{d}u^i}{\mathrm{d}t}\frac{\mathrm{d}u^j}{\mathrm{d}t}L_{kl}\overrightarrow{n}\frac{\mathrm{d}u^l}{\mathrm{d}t}\notag\\
&+\sum_{i,j,k}(L_{ij})_k\frac{\mathrm{d}u^i}{\mathrm{d}t}\frac{\mathrm{d}u^j}{\mathrm{d}t}\frac{\mathrm{d}u^k}{\mathrm{d}t}\overrightarrow{n}
+\sum_{i,j}L_{ij}\frac{\mathrm{d}^2u^i}{\mathrm{d}t^2}\frac{\mathrm{d}u^j}{\mathrm{d}t}\overrightarrow{n}+\sum_{i,j}L_{ij}\frac{\mathrm{d}u^i}{\mathrm{d}t}\frac{\mathrm{d}^2u^j}{\mathrm{d}t^2}\overrightarrow{n}\notag\\
&-\sum_{i,j,k,l,m}L_{ij}\frac{\mathrm{d}u^i}{\mathrm{d}t}\frac{\mathrm{d}u^j}{\mathrm{d}t}L_{km}g^{ml}\overrightarrow{r}_{l}\frac{\mathrm{d}u^k}{\mathrm{d}t}\notag\\
&+\sum_{i,j,k}\frac{\mathrm{d}^2u^i}{\mathrm{d}t^2}\Gamma_{ij}^k\overrightarrow{r}_{k}\frac{\mathrm{d}u^j}{\mathrm{d}t}
+\sum_{i,j}\frac{\mathrm{d}^2u^i}{\mathrm{d}t^2}L_{ij}\overrightarrow{n}\frac{\mathrm{d}u^j}{\mathrm{d}t}
+\sum_i\frac{\mathrm{d}^3u^i}{\mathrm{d}t^3}\overrightarrow{r}_i\notag\\
=&
\sum_{i,j,k,l}(\Gamma_{ij}^k)_l\frac{\mathrm{d}u^i}{\mathrm{d}t}\frac{\mathrm{d}u^j}{\mathrm{d}t}\frac{\mathrm{d}u^l}{\mathrm{d}t}\overrightarrow{r}_k
+\sum_{i,j,k}\Gamma_{ij}^k\frac{\mathrm{d}^2u^i}{\mathrm{d}t^2}\frac{\mathrm{d}u^j}{\mathrm{d}t}\overrightarrow{r}_k
+\sum_{i,j,k}\Gamma_{ij}^k\frac{\mathrm{d}u^i}{\mathrm{d}t}\frac{\mathrm{d}^2u^j}{\mathrm{d}t^2}\overrightarrow{r}_k\notag\\
&+\sum_{i,j,k,l,m}\Gamma_{ij}^k\Gamma_{kl}^m\frac{\mathrm{d}u^i}{\mathrm{d}t}\frac{\mathrm{d}u^j}{\mathrm{d}t}\frac{\mathrm{d}u^l}{\mathrm{d}t}\overrightarrow{r}_{m}
+\sum_{i,j,k,l}\Gamma_{ij}^kL_{kl}\frac{\mathrm{d}u^i}{\mathrm{d}t}\frac{\mathrm{d}u^j}{\mathrm{d}t}\frac{\mathrm{d}u^l}{\mathrm{d}t}\overrightarrow{n}\notag\\
&+\sum_{i,j,k}(L_{ij})_k\frac{\mathrm{d}u^i}{\mathrm{d}t}\frac{\mathrm{d}u^j}{\mathrm{d}t}\frac{\mathrm{d}u^k}{\mathrm{d}t}\overrightarrow{n}
+\sum_{i,j}L_{ij}\frac{\mathrm{d}^2u^i}{\mathrm{d}t^2}\frac{\mathrm{d}u^j}{\mathrm{d}t}\overrightarrow{n}+\sum_{i,j}L_{ij}\frac{\mathrm{d}u^i}{\mathrm{d}t}\frac{\mathrm{d}^2u^j}{\mathrm{d}t^2}\overrightarrow{n}\notag\\
&-\sum_{i,j,k,l,m}L_{ij}L_{km}g^{ml}\frac{\mathrm{d}u^i}{\mathrm{d}t}\frac{\mathrm{d}u^j}{\mathrm{d}t}\frac{\mathrm{d}u^k}{\mathrm{d}t}\overrightarrow{r}_{l}\notag\\
&+\sum_{i,j,k}\Gamma_{ij}^k\frac{\mathrm{d}^2u^i}{\mathrm{d}t^2}\frac{\mathrm{d}u^j}{\mathrm{d}t}\overrightarrow{r}_{k}
+\sum_{i,j}L_{ij}\frac{\mathrm{d}^2u^i}{\mathrm{d}t^2}\frac{\mathrm{d}u^j}{\mathrm{d}t}\overrightarrow{n}
+\sum_i\frac{\mathrm{d}^3u^i}{\mathrm{d}t^3}\overrightarrow{r}_i\notag\\
=&
\sum_{i,j,k,l}(\Gamma_{ij}^k)_l\frac{\mathrm{d}u^i}{\mathrm{d}t}\frac{\mathrm{d}u^j}{\mathrm{d}t}\frac{\mathrm{d}u^l}{\mathrm{d}t}\overrightarrow{r}_k
+\sum_{i,j,k}\Gamma_{ij}^k\frac{\mathrm{d}^2u^i}{\mathrm{d}t^2}\frac{\mathrm{d}u^j}{\mathrm{d}t}\overrightarrow{r}_k
+\sum_{i,j,k}\Gamma_{ij}^k\frac{\mathrm{d}u^i}{\mathrm{d}t}\frac{\mathrm{d}^2u^j}{\mathrm{d}t^2}\overrightarrow{r}_k\notag\\
&+\sum_{i,j,k,l,m}\Gamma_{ij}^m\Gamma_{ml}^k\frac{\mathrm{d}u^i}{\mathrm{d}t}\frac{\mathrm{d}u^j}{\mathrm{d}t}\frac{\mathrm{d}u^l}{\mathrm{d}t}\overrightarrow{r}_{k}
+\sum_{i,j,k,l}\Gamma_{ij}^kL_{kl}\frac{\mathrm{d}u^i}{\mathrm{d}t}\frac{\mathrm{d}u^j}{\mathrm{d}t}\frac{\mathrm{d}u^l}{\mathrm{d}t}\overrightarrow{n}\notag\\
&+\sum_{i,j,k}(L_{ij})_k\frac{\mathrm{d}u^i}{\mathrm{d}t}\frac{\mathrm{d}u^j}{\mathrm{d}t}\frac{\mathrm{d}u^k}{\mathrm{d}t}\overrightarrow{n}
+\sum_{i,j}L_{ij}\frac{\mathrm{d}^2u^i}{\mathrm{d}t^2}\frac{\mathrm{d}u^j}{\mathrm{d}t}\overrightarrow{n}+\sum_{i,j}L_{ij}\frac{\mathrm{d}u^i}{\mathrm{d}t}\frac{\mathrm{d}^2u^j}{\mathrm{d}t^2}\overrightarrow{n}\notag\\
&-\sum_{i,j,k,l,m}L_{ij}L_{lm}g^{mk}\frac{\mathrm{d}u^i}{\mathrm{d}t}\frac{\mathrm{d}u^j}{\mathrm{d}t}
\frac{\mathrm{d}u^l}{\mathrm{d}t}\overrightarrow{r}_{k}\notag\\
&+\sum_{i,j,k}\Gamma_{ij}^k\frac{\mathrm{d}^2u^i}{\mathrm{d}t^2}\frac{\mathrm{d}u^j}{\mathrm{d}t}\overrightarrow{r}_{k}
+\sum_{i,j}L_{ij}\frac{\mathrm{d}^2u^i}{\mathrm{d}t^2}\frac{\mathrm{d}u^j}{\mathrm{d}t}\overrightarrow{n}
+\sum_k\frac{\mathrm{d}^3u^k}{\mathrm{d}t^3}\overrightarrow{r}_k\notag\\
=&
\sum_{i,j,k,l}(\Gamma_{ij}^k)_l\frac{\mathrm{d}u^i}{\mathrm{d}t}\frac{\mathrm{d}u^j}{\mathrm{d}t}\frac{\mathrm{d}u^l}{\mathrm{d}t}\overrightarrow{r}_k
+2\sum_{i,j,k}\Gamma_{ij}^k\frac{\mathrm{d}^2u^i}{\mathrm{d}t^2}\frac{\mathrm{d}u^j}{\mathrm{d}t}\overrightarrow{r}_k
+\sum_{i,j,k}\Gamma_{ij}^k\frac{\mathrm{d}u^i}{\mathrm{d}t}\frac{\mathrm{d}^2u^j}{\mathrm{d}t^2}\overrightarrow{r}_k\notag\\
&+\sum_{i,j,k,l,m}\Gamma_{ij}^m\Gamma_{ml}^k\frac{\mathrm{d}u^i}{\mathrm{d}t}\frac{\mathrm{d}u^j}{\mathrm{d}t}\frac{\mathrm{d}u^l}{\mathrm{d}t}\overrightarrow{r}_{k}
-\sum_{i,j,k,l,m}L_{ij}L_{lm}g^{mk}\frac{\mathrm{d}u^i}{\mathrm{d}t}\frac{\mathrm{d}u^j}{\mathrm{d}t}
\frac{\mathrm{d}u^l}{\mathrm{d}t}\overrightarrow{r}_{k}\notag\\
&
+\sum_k\frac{\mathrm{d}^3u^k}{\mathrm{d}t^3}\overrightarrow{r}_k+\sum_{i,j,k,l}\Gamma_{ij}^kL_{kl}\frac{\mathrm{d}u^i}{\mathrm{d}t}\frac{\mathrm{d}u^j}{\mathrm{d}t}\frac{\mathrm{d}u^l}{\mathrm{d}t}\overrightarrow{n}+\sum_{i,j,k}(L_{ij})_k\frac{\mathrm{d}u^i}{\mathrm{d}t}\frac{\mathrm{d}u^j}{\mathrm{d}t}\frac{\mathrm{d}u^k}{\mathrm{d}t}\overrightarrow{n}\notag\\
&+2\sum_{i,j}L_{ij}\frac{\mathrm{d}^2u^i}{\mathrm{d}t^2}\frac{\mathrm{d}u^j}{\mathrm{d}t}\overrightarrow{n}+\sum_{i,j}L_{ij}\frac{\mathrm{d}u^i}{\mathrm{d}t}\frac{\mathrm{d}^2u^j}{\mathrm{d}t^2}\overrightarrow{n}
.\notag
\end{align}

The calculation of (\ref{xxsecsur15}):


The calculation of (\ref{xxsecsur16}):
%

The calculation of (\ref{xxsecsur17}): This is almost a repetition of (\ref{xxsecsur16}), so we omit the computational process.

\section{The Proof of (3) in Proposition \ref{xxsec5.3}}
\label{Appendix B}

Proposition \ref{xxsec5.3} is an important result in this article. The computations involved in the proof of (3) in the proposition   are extremely tedious. So we complete it in this appendix.

\begin{proof} By invoking (\ref{xxsecapl8})  we obtain
\begin{align}
\frac{ \rm{d}^3\xi_1}{\rm{d}\theta^3}=& 3 [(a\cos\theta-c)^2+b^2\sin^2\theta ]^{-\frac{5}{2}} (    ac\sin\theta -c^2\sin \theta\cos\theta)^3\notag\\
& -2[(a\cos\theta-c)^2+b^2\sin^2\theta ]^{-\frac{3}{2}} (   ac\sin\theta -c^2\sin \theta\cos\theta)
(   ac\cos\theta -c^2\cos2\theta ) \notag\\
&-[(a\cos\theta-c)^2+b^2\sin^2\theta ]^{-\frac{3}{2}}(   ac\sin\theta -c^2\sin \theta\cos\theta)( ac\cos\theta- c^2\cos2\theta  )\notag\\
&+[(a\cos\theta-c)^2+b^2\sin^2\theta ]^{-\frac{1}{2}}(   -ac\sin\theta +2c^2\sin2 \theta)\notag\\
=& 3 [(a\cos\theta-c)^2+b^2\sin^2\theta ]^{-\frac{5}{2}} (    ac\sin\theta -c^2\sin \theta\cos\theta)^3\notag\\
& -3 [(a\cos\theta-c)^2+b^2\sin^2\theta ]^{-\frac{5}{2}}[(a\cos\theta-c)^2+b^2\sin^2\theta ] (   ac\sin\theta -c^2\sin \theta\cos\theta)\notag\\
&\times(   ac\cos\theta -c^2\cos2\theta ) \notag\\
&+[(a\cos\theta-c)^2+b^2\sin^2\theta ]^{-\frac{5}{2}}[(a\cos\theta-c)^2+b^2\sin^2\theta ]^2( 2c^2\sin2\theta- ac\sin\theta  ).
 \label{B1}
\end{align}
 In view of the relation (\ref{xxsecapl9}),  we get
\begin{align}
\frac{ \rm{d}^3\tilde{\xi}_1}{\rm{d}\theta^3}=& -3 [(a\cos\theta+c)^2+b^2\sin^2\theta ]^{-\frac{5}{2}} (    ac\sin\theta +c^2\sin \theta\cos\theta)^3\notag\\
& -  2[(a\cos\theta+c)^2+b^2\sin^2\theta ]^{-\frac{3}{2}} (  ac\sin\theta +c^2\sin \theta\cos\theta)(   ac\cos\theta +c^2\cos2\theta ) \notag\\
& +  [(a\cos\theta+c)^2+b^2\sin^2\theta ]^{-\frac{3}{2}} (  ac\sin\theta +c^2\sin \theta\cos\theta)(   -ac\cos\theta -c^2\cos2\theta ) \notag\\
&+[(a\cos\theta+c)^2+b^2\sin^2\theta ]^{-\frac{1}{2}}(ac\sin\theta+ 2c^2\sin2\theta)\notag\\
=& -3 [(a\cos\theta+c)^2+b^2\sin^2\theta ]^{-\frac{5}{2}} (    ac\sin\theta +c^2\sin \theta\cos\theta)^3\notag\\
& -3 [(a\cos\theta+c)^2+b^2\sin^2\theta ]^{-\frac{5}{2}}[(a\cos\theta+c)^2+b^2\sin^2\theta ] (   ac\sin\theta +c^2\sin \theta\cos\theta)\notag\\
&\times(   ac\cos\theta +c^2\cos2\theta ) \notag\\
&+[(a\cos\theta+c)^2+b^2\sin^2\theta ]^{-\frac{5}{2}}[(a\cos\theta+c)^2+b^2\sin^2\theta ]^2( 2c^2\sin2\theta+ ac\sin\theta  ).
 \label{B2}
\end{align}
 Before proceeding, one should note that $\frac{ \rm{d}^3\xi_1}{\rm{d}\theta^3}+\frac{ \rm{d}^3\tilde{\xi}_1}{\rm{d}\theta^3}=0$. If $\frac{ \rm{d}^3\xi_1}{\rm{d}\theta^3} \geq0$, $\theta\in (0,\pi)$,
then it follows from (\ref{B1}) and (\ref{B2}) that
 \begin{align}
& 3   (    ac\sin\theta -c^2\sin \theta\cos\theta)^3\notag\\
& -3  [(a\cos\theta-c)^2+b^2\sin^2\theta ] (   ac\sin\theta -c^2\sin \theta\cos\theta)
(   ac\cos\theta -c^2\cos2\theta ) \notag\\
&+ [(a\cos\theta-c)^2+b^2\sin^2\theta ]^2( 2c^2\sin2\theta- ac\sin\theta  )\geq0,\label{B3}\\
 &-3   (    ac\sin\theta +c^2\sin \theta\cos\theta)^3\notag\\
& -3  [( a\cos\theta+c)^2+b^2\sin^2\theta ] (   ac\sin\theta +c^2\sin \theta\cos\theta)
(   ac\cos\theta +c^2\cos2\theta ) \notag\\
&+ [( a\cos\theta+c)^2+b^2\sin^2\theta ]^2(  2c^2\sin2\theta+ ac\sin\theta  )\leq0, \label{B4}
\end{align}
where $\theta\in (0,\pi)$. Subtracting (\ref{B4}) by (\ref{B3}) leads to
\begin{align}
&-3   (    ac\sin\theta +c^2\sin \theta\cos\theta)^3- 3   (    ac\sin\theta -c^2\sin \theta\cos\theta)^3\notag\\
& -3  [( a\cos\theta+c)^2+b^2\sin^2\theta ] (   ac\sin\theta +c^2\sin \theta\cos\theta)
(    ac\cos\theta +c^2\cos2\theta ) \notag\\
& + 3  [(a\cos\theta-c)^2+b^2\sin^2\theta ] (   ac\sin\theta -c^2\sin \theta\cos\theta)
(   ac\cos\theta -c^2\cos2\theta ) \notag\\
&+ [( a\cos\theta+c)^2+b^2\sin^2\theta ]^2(  2c^2\sin2\theta+ ac\sin\theta  )\notag\\
&- [(a\cos\theta-c)^2+b^2\sin^2\theta ]^2( 2c^2\sin2\theta- ac\sin\theta  )   \leq0, \notag
\end{align}
where $\theta\in(0,\pi)$. Expanding the above relation, we arrive at
\begin{align}
&-3   c^3\sin^3\theta(    a^3 +3 a^2 c  \cos\theta + 3a   c^2  \cos^2\theta+c ^3 \cos^3\theta)\notag\\& - 3 c^3\sin^3\theta  (    a^3 -3 a^2 c  \cos\theta + 3a   c^2  \cos^2\theta-c ^3 \cos^3\theta) \notag\\
& -3  (  a^2\cos^2\theta+2ac\cos\theta+c^2+b^2\sin^2\theta  )\notag\\&\ \ \ \times(   a^2c^2\sin\theta\cos\theta +ac^3\sin\theta\cos2\theta+ac^3\sin \theta\cos^2\theta   +c^4\sin \theta\cos\theta\cos2\theta)\notag\\
& + 3  (  a^2\cos^2\theta-2ac\cos\theta+c^2+b^2\sin^2\theta  ) \notag\\&\ \ \ \times(   a^2c^2\sin\theta\cos\theta -ac^3\sin\theta\cos2\theta-ac^3\sin \theta\cos^2\theta   +c^4\sin \theta\cos\theta\cos2\theta)\notag\\
&+ [  (a^2\cos^2\theta+c^2+b^2\sin^2\theta)^2+4ac\cos\theta(a^2\cos^2\theta+c^2+b^2\sin^2\theta)+4a^2c^2\cos^2\theta ] \notag\\&\ \ \ \times(  2c^2\sin2\theta+ ac\sin\theta  )\notag\\
&- [  (a^2\cos^2\theta+c^2+b^2\sin^2\theta)^2-4ac\cos\theta(a^2\cos^2\theta+c^2+b^2\sin^2\theta)+4a^2c^2\cos^2\theta ]
\notag\\&\ \ \ \times( 2c^2\sin2\theta- ac\sin\theta  )   \leq0, \notag
\end{align}
where $\theta\in (0,\pi)$. Rewrite this as
\begin{align}
&-3   c^3\sin^3\theta(    a^3 +3 a^2 c  \cos\theta + 3a   c^2  \cos^2\theta+c ^3 \cos^3\theta) \notag\\&- 3 c^3\sin^3\theta  (    a^3 -3 a^2 c  \cos\theta + 3a   c^2  \cos^2\theta-c ^3 \cos^3\theta) \notag\\
&-3  ( a^2\cos^2\theta +c^2+b^2\sin^2\theta )\notag\\ &\ \ \ \times (   a^2c^2\sin\theta\cos\theta +ac^3\sin\theta\cos2\theta+ac^3\sin \theta\cos^2\theta   +c^4\sin \theta\cos\theta\cos2\theta)\notag\\
& + 3  ( a^2\cos^2\theta +c^2+b^2\sin^2\theta )\notag\\&\ \ \  \times (   a^2c^2\sin\theta\cos\theta-ac^3\sin\theta\cos2\theta -ac^3\sin \theta\cos^2\theta   +c^4\sin \theta\cos\theta\cos2\theta)\notag\\
&-3 (2ac\cos\theta  ) (   a^2c^2\sin\theta\cos\theta +ac^3\sin\theta\cos2\theta+ac^3\sin \theta\cos^2\theta   +c^4\sin \theta\cos\theta\cos2\theta)\notag\\
& + 3  (  -2ac\cos\theta ) (   a^2c^2\sin\theta\cos\theta -ac^3\sin\theta\cos2\theta -ac^3\sin \theta\cos^2\theta  +c^4\sin \theta\cos\theta\cos2\theta)\notag\\
&+ [  (a^2\cos^2\theta+c^2+b^2\sin^2\theta)^2+4ac\cos\theta(a^2\cos^2\theta+c^2+b^2\sin^2\theta)\notag\\&\ \ \ +4a^2c^2\cos^2\theta ] (  2c^2\sin2\theta  )\notag\\
&- [  (a^2\cos^2\theta+c^2+b^2\sin^2\theta)^2-4ac\cos\theta(a^2\cos^2\theta+c^2+b^2\sin^2\theta)\notag\\&\ \ \ +4a^2c^2\cos^2\theta ]( 2c^2\sin2\theta   )  \notag\\
&+[  (a^2\cos^2\theta+c^2+b^2\sin^2\theta)^2+4ac\cos\theta(a^2\cos^2\theta+c^2+b^2\sin^2\theta)\notag\\&\ \ \ +4a^2c^2\cos^2\theta ]
(   ac\sin\theta  )\notag\\
&- [  (a^2\cos^2\theta+c^2+b^2\sin^2\theta)^2-4ac\cos\theta(a^2\cos^2\theta+c^2+b^2\sin^2\theta)\notag\\&\ \ \ +4a^2c^2\cos^2\theta ]
(  - ac\sin\theta  )   \leq0, \notag
\end{align}
where $\theta\in(0,\pi)$.   Simplifying the above inequality gives
\begin{align}
&-6   c^3\sin^3\theta(    a^3  + 3a   c^2  \cos^2\theta) \notag\\
& - 6  ( a^2\cos^2\theta +c^2+b^2\sin^2\theta ) (ac^3\sin\theta\cos2\theta +  ac^3\sin \theta\cos^2\theta)\notag\\
& + 6  (  -2ac\cos\theta ) (     a^2c^2\sin\theta\cos\theta  +c^4\sin \theta\cos\theta\cos2\theta )\notag\\
&+ 2[ 4ac\cos\theta(a^2\cos^2\theta+c^2+b^2\sin^2\theta)] (  2c^2\sin2\theta  )\notag\\
&+ 2[   (a^2\cos^2\theta+c^2+b^2\sin^2\theta)^2+4a^2c^2\cos^2\theta] (   ac\sin\theta  )
  \leq0, \notag
\end{align}
that is,
\begin{align}
&2ac\sin\theta[-3   c^2\sin^2\theta(    a^2  + 3    c^2  \cos^2\theta) \notag\\
&\ \ \ \ \ \ \ \ \ \ \ \  - 3  ( a^2\cos^2\theta +c^2+b^2\sin^2\theta ) (   c^2\cos2\theta +c^2 \cos^2\theta)\notag\\
&\ \ \ \ \ \ \ \ \ \ \ \  + 3  (  -2 \cos\theta ) (     a^2c^2 \cos\theta  +c^4 \cos\theta\cos2\theta )\notag\\
&\ \ \ \ \ \ \ \ \ \ \ \  + 4 \cos\theta(a^2\cos^2\theta+c^2+b^2\sin^2\theta)  (  4c^2\cos\theta  )\notag\\
&\ \ \ \ \ \ \ \ \  \ \ \  + (a^2\cos^2\theta+c^2+b^2\sin^2\theta)^2+4a^2c^2\cos^2\theta]
  \leq0,   \label{B5}
\end{align}
where $\theta\in(0,\pi)$.    Next let us do the following computation:\begin{align}
&-3   c^2\sin^2\theta(    a^2  + 3    c^2  \cos^2\theta) \notag\\
& - 3  ( a^2\cos^2\theta +c^2+b^2\sin^2\theta ) (  c^2\cos2\theta + c^2 \cos^2\theta)\notag\\
& + 3  (  -2 \cos\theta ) (     a^2c^2 \cos\theta  +c^4 \cos\theta\cos2\theta )\notag\\
&+    4 \cos\theta(a^2\cos^2\theta+c^2+b^2\sin^2\theta)  (  4c^2\cos\theta  )\notag\\
&+      (a^2\cos^2\theta+c^2+b^2\sin^2\theta)^2+4a^2c^2\cos^2\theta\notag\\
=&-3   c^2\sin^2\theta(    a^2  + 3    c^2  \cos^2\theta) \notag\\
& - 3 ( c^2\cos^2\theta+a^2) (  c^2\cos2\theta+  c^2 \cos^2\theta)\notag\\
& -6 \cos\theta  (     a^2c^2 \cos\theta  +c^4 \cos\theta\cos2\theta )\notag\\
&+  16c^2\cos^2\theta (c^2\cos^2\theta+a^2)  \notag\\
&+    (c^2\cos^2\theta+a^2)^2+4a^2c^2\cos^2\theta\notag\\
=&-3      a^2c^2\sin^2\theta -9    c^4\sin^2\theta  \cos^2\theta \notag\\
&- 3 c^4\cos^2\theta \cos2\theta - 3      c^4 \cos^4\theta  -3a^2c^2\cos2\theta- 3   a^2    c^2 \cos^2\theta \notag\\
& -6      a^2c^2 \cos^2\theta  -6   c^4 \cos^2\theta\cos2\theta  \notag\\
&+  16c^4\cos^4\theta +16 a^2c^2\cos^2\theta  \notag\\
&+     c^4\cos^4\theta+a^4 +2 a^2 c^2\cos^2\theta+4a^2c^2\cos^2\theta
\notag\\
=&-3      a^2c^2\sin^2\theta -9    c^4\sin^2\theta  \cos^2\theta \notag\\
& - 3 c^4\cos^2\theta (\cos^2\theta-\sin^2\theta)- 3      c^4 \cos^4\theta -3a^2c^2 (\cos^2\theta-\sin^2\theta) - 3   a^2    c^2 \cos^2\theta   \notag\\
& -6      a^2c^2 \cos^2\theta  -6   c^4 \cos^2\theta (\cos^2\theta-\sin^2\theta) \notag\\
&+  16c^4\cos^4\theta +16 a^2c^2\cos^2\theta  \notag\\
&+     c^4\cos^4\theta+a^4 +6a^2c^2\cos^2\theta
 \notag\\
=&-3      a^2c^2\sin^2\theta -9    c^4\sin^2\theta  \cos^2\theta \notag\\
& -6 c^4\cos^4\theta  + 3 c^4\cos^2\theta\sin^2\theta  - 6a^2c^2 \cos^2\theta+ 3a^2c^2 \sin^2\theta  \notag\\
& -6      a^2c^2 \cos^2\theta  -6   c^4 \cos^4\theta  +6   c^4 \cos^2\theta\sin^2\theta  \notag\\
&+  16c^4\cos^4\theta +16 a^2c^2\cos^2\theta  \notag\\
&+     c^4\cos^4\theta+a^4 +6a^2c^2\cos^2\theta\notag\\
=&-3      a^2c^2\sin^2\theta -9    c^4\sin^2\theta  \cos^2\theta \notag\\
&  +3a^2c^2 \sin^2\theta + 3 c^4\cos^2\theta\sin^2\theta -6 c^4\cos^4\theta - 6a^2c^2 \cos^2\theta \notag\\
&+6   c^4 \cos^2\theta\sin^2\theta -6   c^4 \cos^4\theta-6      a^2c^2 \cos^2\theta  \notag\\
&+  16c^4\cos^4\theta +16 a^2c^2\cos^2\theta  \notag\\
&+     c^4\cos^4\theta +6a^2c^2\cos^2\theta+a^4\notag\\
=&   -6 c^4\cos^4\theta - 6a^2c^2 \cos^2\theta \notag\\
& -6   c^4 \cos^4\theta-6      a^2c^2 \cos^2\theta  \notag\\
&+  16c^4\cos^4\theta +16 a^2c^2\cos^2\theta  \notag\\
&+     c^4\cos^4\theta +6a^2c^2\cos^2\theta+a^4\notag\\
=&5c^4\cos^4\theta +10 a^2c^2\cos^2\theta  +a^4 . \notag
\end{align}
Taking the above result into (\ref{B5}), we assert that
$$2ac\sin\theta(5c^4\cos^4\theta +10 a^2c^2\cos^2\theta  +a^4   )\leq0,$$where $\theta\in(0,\pi)$. A contradiction. This implies $\frac{ \rm{d}^3\xi_1}{\rm{d}\theta^3} < 0$, $\theta\in (0,\pi)$. By an analogous manner, we can show that $\frac{ \rm{d}^3\xi_1}{\rm{d}\theta^3} > 0$, $\theta\in (\pi,2\pi)$.

\end{proof}

\end{document}